\documentclass[12pt]{article}

\input xy

\xyoption{all}

\usepackage{amssymb,amsbsy,amsthm,amsmath,graphicx,epsfig}
\usepackage{mathabx,times,mathrsfs}

\usepackage{sectsty,tocloft,hyperref}

\partfont{\normalsize}

\sectionfont{\normalsize}

\subsectionfont{\mdseries\itshape\normalsize}

\newtheorem{thm}{Theorem}[section]
\newtheorem*{thm*}{Theorem}
\newtheorem{lem}[thm]{Lemma}
\newtheorem{prop}[thm]{Proposition}
\newtheorem{cor}[thm]{Corollary}

\newtheorem{dfn}[thm]{Definition}

\newtheorem{ques}[thm]{Question}

\theoremstyle{remark}
\newtheorem{ex}[thm]{Example}
\newtheorem*{rmk}{Remark}
\newtheorem*{rmks}{Remarks}

\renewcommand{\bf}[1]{\boldsymbol{#1}}
\renewcommand{\rm}[1]{\mathrm{#1}}

\newcommand{\bbE}{\mathbb{E}}

\newcommand{\bbN}{\mathbb{N}}

\newcommand{\bbR}{\mathbb{R}}

\newcommand{\bbZ}{\mathbb{Z}}

\newcommand{\sfE}{\mathbb{E}}%{\boldsymbol{\mathsf{E}}}
\newcommand{\sfP}{\mathbb{P}}%{\boldsymbol{\mathsf{P}}}

\newcommand{\rmD}{\mathrm{D}}
\newcommand{\rmH}{\mathrm{H}}
\newcommand{\rmI}{\mathrm{I}}

\renewcommand{\d}{\mathrm{d}}
\newcommand{\rme}{\mathrm{e}}
\newcommand{\rmh}{\mathrm{h}}

\newcommand{\B}{\mathscr{B}}

\newcommand{\F}{\mathscr{F}}
\newcommand{\G}{\mathscr{G}}
\newcommand{\scrH}{\mathscr{H}}

\renewcommand{\O}{\Omega}

\renewcommand{\a}{\alpha}
\renewcommand{\b}{\beta}
\newcommand{\eps}{\varepsilon}
\newcommand{\g}{\gamma}
\renewcommand{\k}{\kappa}
\renewcommand{\l}{\lambda}
\renewcommand{\o}{\omega}
\newcommand{\s}{\sigma}

\renewcommand{\Pr}{\mathrm{Prob}}
\newcommand{\Lip}{\mathrm{Lip}}

\renewcommand{\t}[1]{\widetilde{#1}}
\newcommand{\ol}[1]{\overline{#1}}
\renewcommand{\to}{\longrightarrow}
\renewcommand{\phi}{\varphi}

\newcommand{\bl}{\mathrm{block}}
\newcommand{\TC}{\mathrm{TC}}
\newcommand{\DTC}{\mathrm{DTC}}

\begin{document}

\title{\Large{MEASURE CONCENTRATION AND THE WEAK PINSKER PROPERTY}}

\author{Tim Austin}

\date{}

\maketitle

\begin{abstract}
Let $(X,\mu)$ be a standard probability space. An automorphism $T$ of $(X,\mu)$ has the weak Pinsker property if for every $\eps > 0$ it has a splitting into a direct product of a Bernoulli shift and an automorphism of entropy less than $\eps$.  This property was introduced by Thouvenot, who asked whether it holds for all ergodic automorphisms.

This paper proves that it does.  The proof actually gives a more general result.  Firstly, it gives a relative version: any factor map from one ergodic automorphism to another can be enlarged by arbitrarily little entropy to become relatively Bernoulli.  Secondly, using some facts about relative orbit equivalence, the analogous result holds for all free and ergodic measure-preserving actions of a countable amenable group.

The key to this work is a new result about measure concentration.  Suppose now that $\mu$ is a probability measure on a finite product space $A^n$, and endow this space with its Hamming metric. We prove that $\mu$ may be represented as a mixture of other measures in which (i) most of the weight in the mixture is on measures that exhibit a strong kind of concentration, and (ii) the number of summands is bounded in terms of the difference between the Shannon entropy of $\mu$ and the combined Shannon entropies of its marginals.
	\end{abstract}

\renewcommand{\cfttoctitlefont}{\bfseries\normalsize}
\renewcommand{\cftpartfont}{\bfseries\small}
\renewcommand{\cftpartpagefont}{\bfseries\small}
\renewcommand{\cftsecfont}{\bfseries\small}
\renewcommand{\cftsecpagefont}{\bfseries\small}
\renewcommand{\cftsubsecfont}{\small}
\renewcommand{\cftsubsecpagefont}{\small}

\setcounter{tocdepth}{1}

\tableofcontents

\section{Introduction}

\subsection*{The weak Pinsker property}

In this paper, a \textbf{measure-preserving system} or \textbf{automorphism} consists of (i) a standard probability space $(X,\mu)$ and (ii) a measurable and $\mu$-preserving transformation $T:X\to X$ with a measurable inverse.  We write $\B_X$ for the $\s$-algebra of $X$ when it is needed.

One of the most fundamental invariants of an automorphism is its entropy.  Kolmogorov and Sinai first brought the notion of entropy to bear on questions of ergodic theory in~\cite{Kol58,Kol59,Sin59}.  They showed that Shannon's entropy rate for stationary, finite-state stochastic processes is monotone under factor maps of processes, and deduced several new non-isomorphism results, most famously among Bernoulli shifts.

In~\cite{Sin64}, Sinai proved a result about entropy and ergodic theory of a different nature.  He showed that any ergodic automorphism $(X,\mu,T)$ of positive entropy $h$ admits factor maps onto any Bernoulli shift of entropy at most $h$.  This marked a turn in entropy theory towards synthetic results, which deduce the existence of an isomorphism or factor map from some modest assumption about invariants, such as Sinai's inequality between entropies.  These synthetic results are generally much more delicate.  They produce maps between systems using complicated combinatorics and analysis, and the maps that result are rarely canonical.

The most famous result of this kind is Ornstein's theorem~\cite{Orn70,Orn70c}, which shows that two Bernoulli shifts of equal (finite or infinite) entropy are isomorphic.  Together with Kolmogorov and Sinai's original entropy calculations, this completes the classification of Bernoulli shifts up to isomorphism. 

The real heart of Ornstein's work goes far beyond Bernoulli shifts: it provides necessary and sufficient conditions for an arbitrary stationary process to be isomorphic to a Bernoulli shift.  A careful introduction to this sophisticated theory can be found in the contemporary monographs~\cite{Shi73} and~\cite{Orn74}.  An intuitive discussion without many proofs can be found in~\cite{Orn13}, and the historical account in~\cite{Kat07} puts Ornstein's work in a broader context.  Finally, the survey~\cite{Tho02} collects statements and references for much of the more recent work in this area.

Since Ornstein's original work, several other necessary and sufficient conditions for Bernoullicity have also been explored, with far-reaching consequences for the classification of automorphisms. On the one hand, many well-known systems turn out to be isomorphic to Bernoulli shifts: see, for instance,~\cite[Section 6]{Tho02} and the references given there.  On the other hand, many examples can be constructed that are distinct from Bernoulli shifts, in spite of having related properties such as being K-automorphisms~\cite{Orn73,OrnShi73,Kal82}.

In particular, the subsequent non-isomorphism results included counterexamples~\cite{Orn73b,Orn73c} to an early, ambitious conjecture of Pinsker~\cite{Pin60}.  Let us say that a system $(X,\mu,T)$ \textbf{splits} into two of its factors $(Y,\nu,S)$ and $(Y',\nu',S')$ if there is an isomorphism
\[(X,\mu,T) \to (Y\times Y',\nu\times \nu',S\times S').\]
Such an isomorphism is called a \textbf{splitting}.  Pinsker's conjecture asserted that every ergodic automorphism splits into a factor of entropy zero and a K-automorphism. Even after Ornstein found his counterexamples, the class of systems with this structure attracted considerable interest.  Of particular importance is the subclass of systems that split into a zero-entropy factor and a Bernoulli factor.

In general, if
\[\pi:(X,\mu,T) \to (Y,\nu,S)\]
is a factor map of automorphisms, then $(X,\mu,T)$ is \textbf{relatively Bernoulli over $\pi$} if there are a Bernoulli system $\bf{B}$ and a factor map $\phi:(X,\mu,T)\to \bf{B}$ such that the combined map $x\mapsto (\pi(x),\phi(x))$ is a splitting. Since we restrict attention to standard probability spaces, it is equivalent to require that (i) $\pi$ and $\phi$ are independent maps on the probability space $(X,\mu)$ and (ii) $\pi$ and $\phi$ together generate $\B_X$ modulo the $\mu$-negligible sets.

The study of systems that split into zero-entropy and Bernoulli factors led Thouvenot to develop a `relative' version of Ornstein's conditions for Bernoullicity~\cite{Tho75,Tho75b}. Then, in~\cite{Tho77}, he proposed a replacement for the structure in Pinsker's conjecture.  According to Thouvenot, an automorphism has the \textbf{weak Pinsker property} if for every $\eps > 0$ it has a splitting into a Bernoulli shift and a system of entropy less than $\eps$.  In the Introduction to~\cite{Tho77}, Thouvenot wrote,
\begin{quotation}
`The meaning of this definition is that the ``structure'' of these systems lies in factors of arbitrarily small entropy and that their randomness is essentially driven by a Bernoulli process.'
\end{quotation}

Although Pinsker's conjecture and the weak Pinsker property are very close, it is worth being aware of the following important difference.  If an ergodic automorphism has a splitting as in Pinsker's conjecture, then the zero-entropy factor is essentially unique: it must generate the Pinsker $\s$-algebra of the original automorphism.  However, there is nothing canonical about the splittings promised by the weak Pinsker property.  This is illustrated very clearly by a result of Kalikow~\cite{Kal12}.  For a well-known example of a non-Bernoulli K-automorphism, he exhibits two splittings, each into a Bernoulli and a non-Bernoulli factor, such that the $\s$-algebras generated by the two non-Bernoulli factors have trivial intersection.

In this paper we prove that all ergodic automorphisms have the weak Pinsker property.  Our proof actually gives a generalization of this fact, formulated relative to another fixed factor of the automorphism.

\vspace{7pt}

\noindent\textbf{Theorem A.}\quad \emph{Let $\pi:(X,\mu,T) \to (Y,\nu,S)$ be a factor map of ergodic automorphisms. For every $\eps > 0$, this map admits a factorization
\begin{center}
$\phantom{i}$\xymatrix{ (X,\mu,T) \ar^\pi[rr]\ar_{\pi_1}[dr] & & (Y,\nu,S) \\
& (\t{Y},\t{\nu},\t{S}) \ar_{\pi_2}[ur]}
\end{center}
in which $(X,\mu,T)$ is relatively Bernoulli over $\pi_1$ and $(\t{Y},\t{\nu},\t{S})$ has relative entropy less than $\eps$ over $\pi_2$.}

\vspace{7pt}

In fact, our method naturally yields an even more general result: the extension of Theorem A to free and ergodic measure-preserving actions of arbitrary countable amenable groups.  This is Theorem~\ref{thm:A+} below.  In our approach this requires essentially no more work than Theorem A, because even for Theorem A we appeal to some relative orbit equivalence theory which applies equally easily to any free amenable group action.

By a result of Fieldsteel~\cite{Fie84}, Theorem A also implies the weak Pinsker property for flows; we discuss this further in Subsection~\ref{subs:other-groups}.

\subsection*{Measure concentration}

Within ergodic theory, the proof of Theorem A requires a careful use of Thouvenot's relative Ornstein theory.  Some of the elements we need have apparently not been published before, but they are quite natural extensions of well-known results.  We state and prove these results carefully where we use them below, but they should not be thought of as really new.

The main innovations of this paper are results in discrete probability, and do not mention ergodic theory at all.  They concern the important phenomenon of measure concentration.  This phenomenon is already implicitly used in Ornstein's original work. Its role in ergodic theory is made clearer by results of Marton and Shields~\cite{MarShi94}, who call it the `blowing up property'. (Also, see Section~\ref{sec:rel-Orn} for some older references in Ornstein theory about the related notion of `extremality'.) Measure concentration has other important applications across probability theory~\cite{Tal94,Tal96,Ver--HDPbook}, geometry~\cite[Chapter 3$\frac{1}{2}$]{Gro99}, and Banach space theory~\cite{MilmanSchech86}.

Let $A$ be a finite alphabet. We study the Cartesian product spaces $A^n$ for large $n$, endowed with the normalized Hamming metrics $d_n$ (see Subsection~\ref{subs:Ham-met}).  On the set of all probability measures $\Pr(A^n)$, let $\ol{d_n}$ be the transportation metric arising from $d_n$ (see Subsection~\ref{subs:meas-trans}).  We focus on a specific form of measure concentration expressed by a transportation inequality.  Given $\mu \in \Pr(A^n)$ and also parameters $\k > 0$ and $r > 0$, we say $\mu$ satisfies the inequality T$(\k,r)$ if the following holds for any other $\nu \in \Pr(A^n)$:
\[\ol{d_n}(\nu,\mu) \leq \frac{1}{\k}\rmD(\nu\,\|\,\mu) + r,\]
where $\rmD$ denotes Kullback--Leibler divergence.  Such inequalities are introduced more carefully and generally in Section~\ref{sec:meas-conc}.

Inequalities of roughly this kind go back to works of Marton about product measures, starting with~\cite{Mar86}. See also~\cite{Tal96b} for related results in Euclidean spaces. To see how T$(\k,r)$ implies concentration in the older sense of L\'evy, consider a measurable subset $U \subseteq A^n$ with $\mu(U) \geq 1/2$, and apply T$(\k,r)$ to the measure $\nu := \mu(\,\cdot\,|\,U)$.  The result is the bound
\[\ol{d_n}(\mu(\,\cdot\,|\,U),\,\mu) \leq \frac{1}{\k}\rmD(\mu(\,\cdot\,|\,U)\,\|\,\mu) + r \leq \frac{\log 2}{\k} + r.\]
If $\k$ is large and $r$ is small, then it follows that $\ol{d_n}(\mu(\,\cdot\,|\,U),\,\mu)$ is small.  This means that there is a coupling $\l$ of $\mu(\,\cdot\,|\,U)$ and $\mu$ for which the integral
\[\int d_n(\bf{a},\bf{a}')\,\l(\d\bf{a},\d\bf{a}')\]
is small.  By Markov's inequality, this $\l$ must mostly be supported on pairs of strings $(\bf{a},\bf{a}')$ that are close in $d_n$. Since the first marginal of $\l$ is supported on $U$, and the second is the whole of $\mu$, this is possible only if some small Hamming-neighbourhood of $U$ has almost full measure according to $\mu$.

Marton's original paper contains a more complete discussion of the relation between transportation inequalities and more traditional notions of measure concentration on metric spaces.  Their basic connection to Ornstein theory is discussed in~\cite{MarShi94}.

Our first new result about concentration asserts that an arbitrary probability measure $\mu$ on $A^n$ can be represented as a mixture of other measures, most of which do satisfy a good transportation inequality, and with the number of summands controlled in terms of a quantity that indicates how much $\mu$ differs from being a product measure.

\vspace{7pt}

\noindent\textbf{Theorem B.}\quad \emph{For any $\eps,r > 0$ there exists $c > 0$ such that the following holds.  Let $\mu \in \Pr(A^n)$, let $\mu_{\{i\}} \in \Pr(A)$ be its marginals for $i=1,2,\dots,n$, and set
	\begin{equation}\label{eq:Bdef}
E := \sum_{i=1}^n\rmH(\mu_{\{i\}}) - \rmH(\mu).
\end{equation}
Then $\mu$ may be written as a mixture
\begin{equation}\label{eq:conv-comb}
\mu = p_1\mu_1 + \dots + p_m\mu_m
\end{equation}
	so that
	\begin{enumerate}
		\item[(a)] $m \leq c\rme^{cE}$,
		\item[(b)] $p_1 < \eps$, and
		\item[(c)] the measure $\mu_j$ satisfies T$(r n/1200,r)$ for every $j = 2,3,\dots,m$.
	\end{enumerate} }

\vspace{7pt}

The constant 1200 appearing here is presumably far from optimal, but I have not tried to improve it.

Beware that the inequality T$(\k,r)$ gets stronger as $\k$ \emph{increases} or $r$ \emph{decreases}, and so knowing T$(rn/1200,r)$ for one particular value of $r > 0$ does not imply it for any other value of $r$, larger or smaller.

The constant $c$ provided by Theorem B depends on $\eps$ and $r$, but not on the alphabet $A$.  In fact, if we allow $c$ to depend on $A$, then Theorem B is easily subsumed by simpler results.  To see this we must consider two cases.  On the one hand, Marton's concentration result for product measures (see Proposition~\ref{prop:Marton} below) gives a small constant $\a$ depending on $r$ such that, if $E < \a n$, then $\mu$ itself is close in $\ol{d_n}$ to the product of its marginals. Using this, one can find a subset $U\subseteq A^n$ with $\mu(U)$ close to $1$ and such that $\mu(\,\cdot\,|\,U)$ is highly concentrated: see Proposition~\ref{prop:T-under-dbar} below for a precise statement.  So in case $E < \a n$, we can obtain Theorem B with the small error term plus only one highly concentrated term.  On the other hand, if $E \geq \a n$, then the partition of $A^n$ into \emph{singletons} has only
\[|A|^n = \rme^{\log|A| n} \leq \rme^{\frac{\log|A|}{\a}E}\]
parts.  This bound has the form of (a) above if we allow $c$ to be $\log|A|/\a$, and point masses are certainly highly concentrated.  In view of these two arguments, Theorem B is useful only when $|A|$ is large in terms of $\eps$, $r$ and $E$.  That turns out to be the crucial case for our application to ergodic theory.

Once Theorem B is proved, we also prove the following variant.

\vspace{7pt}

\noindent\textbf{Theorem C.}\quad \emph{For any $\eps,r > 0$ there exist $c > 0$ and $\k >0$ such that, for any alphabet $A$, the following holds for all sufficiently large $n$.  Let $\mu \in \Pr(A^n)$, and let $E$ be as in~\eqref{eq:Bdef}.  Then there is a partition
	\[A^n = U_1 \cup \dots \cup U_m\]
such that
	\begin{enumerate}
		\item[(a)] $m \leq c\rme^{cE}$,
		\item[(b)] $\mu(U_1) < \eps$, and
		\item[(c)] $\mu(U_j) > 0$ and the conditioned measure $\mu(\,\cdot\,|U_j)$ satisfies T$(\k n,r)$ for every $j= 2,3,\dots,m$.
\end{enumerate} }

\vspace{7pt}

Qualitatively, this is stronger than Theorem B: one recovers the qualitative result of Theorem B by letting $p_j = \mu(U_j)$ and $\mu_j = \mu(\,\cdot\,|U_j)$ for each $j$.  As in Theorem B, the proof gives a value for $\k$ which is a small constant multiple of $r$, but this time we have not recorded the precise constant in the statement.  Also, the statement of Theorem C requires that $n$ be sufficiently large in terms of $\eps$, $r$ and $A$, but I do not know whether the dependence on $A$ is really needed.

Theorem C is proved using Theorem B in Section~\ref{sec:C}.  The key idea is a random partitioning procedure which can be used to turn the mixture from Theorem B into a suitable partition.

Theorem C is the key to our proof of Theorem A.  In that proof, Theorem C is applied repeatedly for smaller and smaller values of $r$ (and with different choices of $A$ and $n$ each time), and the resulting partitions are then combined into a construction of the new factor $\pi_1$ of Theorem A.  In fact, for this application it suffices to know a special case of Theorem C which has a shorter proof, because in the ergodic theoretic setting some relevant extra hypotheses can be obtained from the Shannon--McMillan theorem.  However, we choose to prove and then cite the full strength of Theorem C.  We do so both for the sake of greater generality and because it reduces the complexity of the proof of Theorem A.

In Theorems B and C, the starting hypothesis is a bound on the difference of entropies in~\eqref{eq:Bdef}.  The use of this particular quantity underlies the application to proving Theorem A, because we are able to control this quantity in the ergodic theoretic setting.  This control is a simple consequence of the fact that, in a stationary stochastic process, the normalized Shannon entropies of long finite strings converge to the limiting entropy rate of the process.  The resulting link from Theorem C to Theorem A is explained in Step 2 of the construction in Subsection~\ref{subs:preA1}, and specifically by Lemma~\ref{lem:good-set}.  That specific explanation can largely be appreciated before reading the rest of the paper, although the rest of Subsection~\ref{subs:preA1} requires more preparation.

The proof of Theorem B is not constructive.  It gives no reason to expect that the summands in the resulting mixture are unique, even approximately.  The same applies to the partition produced by Theorem C.  This issue may reflect the fact that the splittings promised by the weak Pinsker property are not at all canonical.  However, I do not know of concrete examples of measures $\mu$ on $A^n$ for which one can prove that several decompositions as in Theorem B are possible, none of them in any way `better' than the others.  It would be worthwhile to understand this issue more completely.

For our application of these decomposition results to ergodic theory it suffices to let $\eps = r$.  However, the statements and proofs of these results are clearer if the roles of these two tolerances are kept separate.

\subsection*{Structure and randomness}

Theorems B and C bear a resemblance to some key structural results in extremal and additive combinatorics.  They may be regarded as `rough classifications' of measures on $A^n$, in the sense proposed by Gowers in~\cite{Gow00}.  Although Gowers' principal examples come from those branches of combinatorics, he suggests that rough classifications should be useful across large parts of mathematics: see his Subsection 3.8.  Theorems B and C are examples in a new setting.

Let us discuss this further for Theorem B. Theorem B expresses an arbitrary measure on $A^n$ as a mixture of other measures, not too many in number, and most of them highly concentrated.  We may fit this description into the following terms:
\begin{itemize}
	\item A kind of `structure' which has controlled `complexity'. A suitable indication of this `complexity' is the number of summands in the mixture~\eqref{eq:conv-comb}.
	\item A kind of `randomness' relative to that structure.  In this case measure concentration plays the part of randomness: most of the weight in~\eqref{eq:conv-comb} is on measures $\mu_i$ that exhibit strong concentration.  It is natural to regard concentration as a manifestation of `randomness' because product measures --- the `most random' of all --- are very highly concentrated.
\end{itemize}

To a first approximation, we can identify these notions of `structure' and `randomness' with the terms that Thouvenot used to describe the meaning of the weak Pinsker property in the sentence from~\cite{Tho77} cited above.  Put very roughly, the small number of summands in~\eqref{eq:conv-comb} turns into the smallness of the relative entropy $\rmh(\t{\nu},\t{S}\,|\,\pi_2)$ in Theorem A.  (This description ignores several extra constructions and arguments that are needed before the proof of Theorem A is complete, but they are not of the same order of magnitude as Theorem B itself.)

A few years after Gowers' paper, Tao's ICM address~\cite{Tao07--ICM} emphasized that a dichotomy between `structure' and `randomness' can be found in all the different proofs of Szemer\'edi's theorem (one of the highlights of additive combinatorics).  Tao describes how this dichotomy reappears in related work in ergodic theory and harmonic analysis. The flavour of ergodic theory in Tao's paper is different from ours: he refers to the theory that originates in Furstenberg's multiple recurrence theorem~\cite{Fur77}, which makes no mention of Kolmogorov--Sinai entropy.  However, the carrier of ergodic theoretic `structure' is the same in both settings: a special factor (distal in Furstenberg's work, low-entropy in ours) relative to which the original automorphism exhibits some kind of `randomness' (relative weak mixing, respectively relative Bernoullicity).  Thus, independently and after an interval of thirty years, Tao's identification of this useful dichotomy aligns with Thouvenot's description of the weak Pinsker property cited above.

The parallel with parts of extremal and additive combinatorics resurfaces inside the proof of Theorem B.  In the most substantial step of that proof, the representation of $\mu$ as a mixture is obtained by a `decrement' argument in a certain quantity that describes how far $\mu$ is from being a product measure.  (The `decrement' argument does not quite produce the mixture in~\eqref{eq:conv-comb} --- some secondary processing is still required.) This proof is inspired by the `energy increment' proof Szemer\'edi's regularity lemma and various similar proofs in additive combinatorics.  Indeed, one of the main innovations of the present paper is finding the right quantity in which to seek a suitable decrement.  It turns out to be the `dual total correlation', a classical measure of multi-variate mutual information (see Section~\ref{sec:DTC}).  We revisit the comparison with Szemer\'edi's regularity lemma at the end of Subsection~\ref{subs:DTC-dec-intro}.

\subsection*{Outline of the rest of the paper}

Part I begins with background on Shannon entropy and related quantities, metric probability spaces, and measure concentration.  Then the bulk of this part is given to the proofs of Theorems B and C.  The final section in this part, Section~\ref{sec:ext}, introduces `extremality', a different notion of concentration which is more convenient for applications to ergodic theory.  However, none of Part I itself involves any ergodic theory, and this part may be relevant to other applications in metric geometry or probability.  Section~\ref{sec:ext} is logically independent from the proofs of Theorems B and C. 

Part II gives a partial account of Thouvenot's relative Ornstein theory, tailored to the specific needs of the proof of Theorem A.  This part makes no mention of Theorems B or C.  Its only reliance on Part I is for some of the general machinery in Section~\ref{sec:ext}.

Finally, Part III completes the proof of Theorem A and its generalization to countable amenable groups, and collects some consequences and remaining open questions.  In this part we draw together the threads from Parts I and II.  In particular, the relevance of Theorem C to proving Theorem A is explained in Section~\ref{sec:preA}.

\subsection*{Acknowledgements}

My sincere thanks go to Jean-Paul Thouvenot for sharing his enthusiasm for the weak Pinsker property with me, and for patient and thoughtful answers to many questions over the last few years.  I have also been guided throughout this project by some very wise advice from Dan Rudolph, whom I had the great fortune to meet before the end of his life.  More recently, I had helpful conversations with Brandon Seward, Mike Saks, Ofer Zeitouni and Lewis Bowen.

Lewis Bowen also read an early version of this paper carefully and suggested several important changes and corrections.  Other valuable suggestions came from Yonatan Gutman, Jean-Paul Thouvenot and an anonymous referee.

Much of the work on this project was done during a graduate research fellowship from the Clay Mathematics Institute.

\part{MEASURES ON FINITE CARTESIAN PRODUCTS}

\section{Background on kernels, entropy, and mutual information}

We write $\Pr(X)$ for the set of all probability measures on a standard measurable space $X$.  We mostly use this notation when $X$ is a finite set.  In that case, given $\mu \in \Pr(X)$ and $x \in X$, we often write $\mu(x)$ instead of $\mu(\{x\})$.

\subsection{Kernels, joint distributions, and fuzzy partitions}\label{subs:kers}

Let $\O$ and $X$ be two measurable spaces with $X$ standard.  A \textbf{kernel} from $\O$ to $X$ is a measurable function $\o\mapsto \mu_\o$ from $\O$ to the set $\Pr(X)$, where the latter is given its usual measurable structure (generated, for instance, by the vague topology resulting from any choice of compact metric on $X$ which generates $\B_X$).  We often denote such a kernel by $\mu_\bullet$. Kernels arise as regular conditional distributions in probability theory; we meet many such cases below.

Suppose in addition that $P$ is a probability measure on $\O$. Then $P$ and $\mu_\bullet$ together define a probability measure on $\O\times X$ with its product $\s$-algebra, given formally by the integral
\[\int_\O \delta_\o\times \mu_\o\ P(\d\o).\]
See, for instance,~\cite[Theorem 10.2.1]{Dud--book} for the rigorous construction.  We denote this new measure by $P\ltimes \mu_\bullet$.  Following some information theorists, we refer to it as the \textbf{hookup} of $P$ and $\mu_\bullet$.  It is a coupling of $P$ with the averaged measure
\begin{equation}\label{eq:bary}
	\mu = \int_\O \mu_\omega\,P(\d \omega).
\end{equation}
Such an average of the measures in a kernel is called a \textbf{mixture}. Often we start with the measure $\mu$ and then consider ways to express it as a mixture.  If $\O$ and $\mu$ are non-trivial then there are many ways to do this.  Once an expression as in~\eqref{eq:bary} has been chosen, we may also consider a pair of random variables $(\zeta,\xi)$ whose joint distribution is $P\ltimes \mu_\bullet$.  Then $\xi$ by itself has distribution $\mu$. Loosely motivated by terminology in~\cite[Section II.5]{FellerVolII}, we call such a pair of random variables a \textbf{randomization} of the mixture.

Now consider a standard probability space $(X,\mu)$ and a measurable function $\rho:X\to [0,\infty)$ satisfying
\[0 < \int \rho\,\d\mu < \infty.\]
Then we define
\[\mu_{|\rho} := \frac{\rho}{\int\rho\,\d\mu}\cdot \mu.\]
Similarly, if $U \subseteq X$ has positive measure, then we often write
\[\mu_{|U} := \mu(\,\cdot\,|\,U).\]

A \textbf{fuzzy partition} on $X$ is a finite tuple $(\rho_1,\dots,\rho_k)$ of measurable functions from $X$ to $[0,1]$ satisfying
\begin{equation}\label{eq:fuzzy}
\rho_1 + \rho_2 + \cdots + \rho_k = 1.
\end{equation}
This is more conventionally called a `partition of unity', but in the present setting I feel `fuzzy partition' is less easily confused with partitions that consist of sets.  If $P_1$, \dots, $P_k$ are a partition of $X$ into measurable subsets, then they give rise to the fuzzy partition $(1_{P_1},\dots,1_{P_k})$, which has the special property of consisting of indicator functions.

A fuzzy partition $(\rho_1,\dots,\rho_k)$ gives rise to a decomposition of $\mu$ into other measures:
\[\mu = \rho_1\cdot \mu + \cdots + \rho_k\cdot \mu.\]
We may also write this as a mixture
\begin{equation}\label{eq:fuzzy-mix}
p_1\cdot \mu_1 + \cdots + p_k\cdot \mu_k
\end{equation}
by letting $p_i := \int \rho_i\,\d\mu$ and $\mu_j:= \mu_{|\rho_j}$, where terms for which $p_i = 0$ are interpreted as zero.

The stochastic vector $p := (p_1,p_2,\dots,p_k)$ may be regarded as a probability distribution on $\{1,2,\dots,k\}$, and the family $(\mu_j)_{j=1}^k$ may be regarded as a kernel from $\{1,2,\dots,k\}$ to $X$. We may therefore construct the hookup $p\ltimes \mu_\bullet$.  This hookup has a simple formula in terms of the fuzzy partition:
\[(p\ltimes \mu_\bullet)(\{j\}\times A) = \int_A \rho_j\,\d\mu \quad \hbox{for}\ 1 \leq j \leq k\ \hbox{and measurable}\ A\subseteq X.\]
One can also reverse the roles of $X$ and $\{1,2,\dots,k\}$ here, and regard $p\ltimes \mu_\bullet$ as the hookup of $\mu$ to the kernel
\[x\mapsto (\rho_1(x),\dots,\rho_k(x))\]
from $X$ to $\Pr(\{1,2,\dots,k\})$: the fact that this \emph{is} a kernel is precisely~\eqref{eq:fuzzy}.

\subsection{Entropy and mutual information}\label{subs:basic}

If $\xi$ and $\zeta$ are finite-valued random variables on a probability space $(\O,\F,P)$, then $\rmH(\xi)$, $\rmH(\xi\,|\,\zeta)$ denote Shannon entropy and conditional Shannon entropy, respectively.  We also use the notation $\rmH(\mu)$ for the Shannon entropy of a probability distribution $\mu$ on a finite set. We assume familiarity with the basic properties of entropy, particularly the chain rule.  A standard exposition is~\cite[Chapter 2]{CovTho06}.

More generally, one can define the conditional entropy $\rmH(\xi\,|\,\G)$ whenever $\xi$ is a finite-valued random variable and $\G$ is a $\s$-subalgebra of $\F$.  To write it in terms of unconditional entropy, let $\o \mapsto \mu_\o$ be a conditional distribution for $\xi$ given $\G$. Then we have
\begin{equation}\label{eq:cond-ent-formula}
\rmH(\xi\,|\,\G) = \int \rmH(\mu_\o)\,P(\d \o).
\end{equation}
See, for instance,~\cite[Section 12]{Bil65}.

If $F \in \F$ has $P(F) > 0$, then we sometimes use the notation
\begin{equation}\label{eq:ent-on-event}
\rmH(\xi\,|\,F) := \rmH_{P(\,\cdot\,|\,F)}(\xi) \quad \hbox{and} \quad \rmH(\xi\,|\,\zeta;F):= \rmH_{P(\,\cdot\,|\,F)}(\xi\,|\,\zeta).
\end{equation}

For a fixed alphabet $A$, Shannon's entropy function $\rmH$ is concave on $\Pr(A)$.  The next inequality gives a useful way to reverse this concavity, up to an additional error term.

\begin{lem}\label{lem:near-convex}
	Consider a finite mixture
	\[\mu = p_1\mu_1 + \cdots + p_k\mu_k\]
	of probability measures on a finite set $A$.  Then
	\[\rmH(\mu) \leq \sum_{j=1}^k p_j\rmH(\mu_j) + \rmH(p_1,\dots,p_k).\]
\end{lem}

\begin{proof}
	Let $(\zeta,\xi)$ be a randomization of the mixture, so this random pair takes values in $\{1,2,\dots,k\}\times A$.  Then $\zeta$ and $\xi$ have marginal distributions $(p_1,\dots,p_k)$ and $\mu$, respectively.  Now monotonicity of $\rmH$ and the chain rule give
	\[\rmH(\mu) = \rmH(\xi) \leq \rmH(\zeta,\xi) = \rmH(\xi\,|\,\zeta) + \rmH(\zeta)
	= \sum_{j=1}^k p_j\rmH(\mu_j) + \rmH(p_1,\dots,p_k).\]
\end{proof}

In addition to entropy and conditional entropy, we make essential use of two other related quantities: mutual information, and Kullback--Leibler (`KL') divergence.  These are also standard in information theory, but appear less often in ergodic theory.  For finite-valued random variables they are introduced alongside entropy in~\cite[Chapter 2]{CovTho06}.

If $\xi$ and $\zeta$ are finite-valued, then their mutual information is
\[\rmI(\xi\,;\,\zeta) := \rmH(\xi) - \rmH(\xi\,|\,\zeta).\]
An exercise in the chain rule gives the alternative formula
\[\rmI(\xi\,;\,\zeta) = \rmH(\xi) + \rmH(\zeta) - \rmH(\xi,\zeta).\]
From this it follows that $\rmI(\xi\,;\,\zeta)$ is symmetric in $\xi$ and $\zeta$.  We define a conditional mutual information such as $\rmI(\xi\,;\,\zeta\,|\,F)$ or $\rmI(\xi\,;\,\zeta\,|\,\a)$ by conditioning all the entropies in the formulae above on the set $F$ or random variable $\a$.  Mutual information and its conditional version satisfy their own chain rule, similar to that for entropy: see, for instance,~\cite[Theorem 2.5.2]{CovTho06}.

KL divergence is a way of comparing two probability measures, say $\mu$ and $\nu$, on the same measurable space $X$. The KL divergence of $\nu$ with respect to $\mu$ is $+\infty$ unless $\nu$ is absolutely continuous with respect to $\mu$. In that case it is given by
\begin{equation}\label{eq:dfn-D}
\rmD(\nu\,\|\,\mu) := \int \frac{\d\nu}{\d\mu}\log \frac{\d\nu}{\d\mu}\,\d\mu = \int \log \frac{\d\nu}{\d\mu}\,\d\nu.
\end{equation}
This may still be $+\infty$, since $\frac{\d\nu}{\d\mu}\log \frac{\d\nu}{\d\mu}$ need not lie in $L^1(\mu)$.  Since the function $t\mapsto t\log t$ is strictly convex, Jensen's inequality gives that $\rmD(\nu\,\|\,\mu) \geq 0$, with equality if and only if $\d\nu/\d\mu = 1$, hence if and only if $\nu = \mu$.

KL divergence appears in various elementary but valuable formulae for conditional entropy and mutual information. Let the random variables $\zeta$ and $\xi$ take values in the finite sets $A$ and $B$. Let $\mu \in \Pr(A)$, $\nu \in \Pr(B)$, and $\l \in \Pr(A\times B)$ be the distribution of $\zeta$, of $\xi$, and of the pair $(\zeta,\xi)$ respectively.  Then a standard calculation gives
\begin{equation}\label{eq:I}
	\rmI(\xi\,;\,\zeta) = \rmD(\l\,\|\,\mu\times \nu).
\end{equation}

The next calculation is also routine, but we include a proof for completeness.

\begin{lem}\label{lem:I-and-KL}
Let $(\O,\F,P)$ be a probability space and let $X$ be a finite set. Let $\xi:\O \to X$ be measurable and let $\mu$ be its distribution.  Let $\G \subseteq \scrH$ be $\s$-subalgebras of $\F$, and let $\mu_\bullet$ and $\nu_\bullet$ be conditional distributions of $\xi$ given $\G$ and $\scrH$, respectively. Then
	\[\rmH(\xi\,|\,\G) - \rmH(\xi\,|\,\scrH) = \int \rmD(\nu_\o\,\|\,\mu_\o)\ P(\d \o).\]
In particular, applying this with $\G$ trivial, we have
	\[\rmH(\xi) - \rmH(\xi\,|\,\scrH) = \int \rmD(\nu_\o\,\|\,\mu)\ P(\d \o).\]
\end{lem}

\begin{proof}
Since $X$ is finite, the definition of KL divergence gives
	\begin{multline*}
		\int \rmD(\nu_\o\,\|\,\mu_\o)\ P(\d \o) = \int \sum_{x \in X}\nu_\o(x)\log \frac{\nu_\o(x)}{\mu_\o(x)}\ P(\d \o)\\
 = \int \sum_{x\in X}\nu_\o(x)\log \nu_\o(x)\ P(\d \o) - \int \sum_{x\in X}\nu_\o(x)\log \mu_\o(x)\ P(\d \o).
	\end{multline*}
For each $x \in X$ we have
\[\nu_\o(x) = P(\{\xi = x\}\,|\,\scrH)(\o) \quad \hbox{and} \quad \mu_\o(x) = P(\{\xi = x\}\,|\,\G)(\o) \quad P\hbox{-a.s.}\]
Therefore, by the tower property of conditional expectation, for each fixed $x\in X$ the function $\mu_\o(x)$ is a conditional expectation of $\nu_\o(x)$ onto $\G$.  Using this in the second integral above, that formula becomes
\[\int \sum_{x\in X}\nu_\o(x)\log \nu_\o(x)\ P(\d \o) - \int \sum_{x\in X}\mu_\o(x)\log \mu_\o(x)\ P(\d \o).\]
This equals $[-\rmH(\xi\,|\,\scrH) + \rmH(\xi\,|\,\G)]$, by~\eqref{eq:cond-ent-formula}.
\end{proof}

Now consider a fuzzy partition $(\rho_1,\dots,\rho_k)$ on a finite set $X$ and the associated mixture~\eqref{eq:fuzzy-mix}. Let $(\zeta,\xi)$ be a randomization of this mixture. We define
\begin{equation}\label{eq:Imu}
\rmI_\mu(\rho_1,\dots,\rho_k) := \rmI(\xi\,;\,\zeta).
\end{equation}
This notion of mutual information for a measure and a fuzzy partition plays a central role later in Part I.  Using Lemma~\ref{lem:I-and-KL} it may be evaluated as follows. 

\begin{cor}\label{cor:I-and-KL}
With $\mu$ and $(\rho_1,\dots,\rho_k)$ as above, we have
\[\rmI_\mu(\rho_1,\dots,\rho_k) = \sum_{j=1}^k p_j\cdot \rmD(\mu_{|\rho_j}\,\|\,\mu), \quad \hbox{where}\ p_j := \int \rho_j\,\d\mu.\]
\end{cor}

\begin{proof}
	Let $(\zeta,\xi)$ be the randomization above, and apply the second formula of Lemma~\ref{lem:I-and-KL} with $\scrH$ the $\s$-algebra generated by $\zeta$.  The result follows because $\mu_{|\rho_j}$ is the conditional distribution of $\xi$ given the event $\{\zeta = j\}$, and because
	\[\rmI_\mu(\rho_1,\dots,\rho_k) = \rmI(\xi\,;\,\zeta) = \rmH(\xi) - \rmH(\xi\,|\,\zeta).\]
	\end{proof}

It is also useful to have a version of the chain rule for mutual information in terms of fuzzy partitions.  To explain this, let us consider two fuzzy partitions $(\rho_i)_{i=1}^k$ and $(\rho'_j)_{j=1}^\ell$ on a finite set $X$, and assume there is a partition
\[\{1,2,\dots,\ell\} = J_1 \cup J_2 \cup\cdots \cup J_k\]
into nonempty subsets such that
\begin{equation}\label{eq:refine}
\rho_i = \sum_{j \in J_i}\rho_j' \quad \hbox{for}\ i=1,2,\dots,k.
\end{equation}
This is a natural generalization of the notion of refinement for partitions consisting of sets.  However, beware that for a general pair of fuzzy partitions there may be several ways to partition $\{1,2,\dots,\ell\}$ into $k$ subsets so that~\eqref{eq:refine} holds.

From~\eqref{eq:refine}, it follows that the tuple of ratios
\begin{equation}\label{eq:newfuzzy}
\Big(\frac{\rho'_j}{\rho_i}\Big)_{j \in J_i}
\end{equation}
define a new fuzzy partition for each $1 \leq i \leq k$.  To be precise, these ratios are defined only on the set $\{\rho_i > 0\}$.  Outside that set, each $\rho_j'$ for $j \in J_i$ is also zero, by~\eqref{eq:refine}.  We extend the new fuzzy partition~\eqref{eq:newfuzzy} arbitrarily to the set $\{\rho_i = 0\}$.

Let $(\zeta,\xi)$ be a randomization of the mixture obtained from $\mu$ and $(\rho'_j)_{j=1}^\ell$.  This means that $(\zeta,\xi)$ are a random pair in $\{1,2,\dots,\ell\}\times X$, and their joint distribution is given by a hookup of the kind constructed at the end of the previous subsection.  Define a third random variable $\ol{\zeta}$ taking values in $\{1,2,\dots,k\}$ by setting
\[\ol{\zeta} = i \quad \hbox{when} \quad \zeta \in J_i.\]
The pair $(\ol{\zeta},\xi)$ may be regarded as a `coarsening' of $(\zeta,\xi)$ in the first entry.  From the relation~\eqref{eq:refine} it follows that the pair $(\ol{\zeta},\xi)$ are a randomization of the mixture obtained from $\mu$ and $(\rho_i)_{i=1}^k$. Similarly, if we condition on the event $\{\ol{\zeta} = i\}$, then $(\zeta,\xi)$ become a randomization of the mixture obtained from $\mu_{|\rho_i}$ and the new fuzzy partition~\eqref{eq:newfuzzy}.  Since the set $\{\rho_i = 0\}$ is negligible according to $\mu_{|\rho_i}$, our choice of extension of~\eqref{eq:newfuzzy} to this set is unimportant.

Now we can simply write out the chain rule
\[\rmI(\xi\,;\,\zeta) = \rmI(\xi\,;\,\ol{\zeta}) + \rmI(\xi\,;\,\zeta\,|\,\ol{\zeta})\]
(see~\cite[Theorem 2.5.2]{CovTho06}) in terms of all these fuzzy partitions.  The result is
\begin{equation}\label{eq:fuzzy-chain}
\rmI_\mu(\rho_1',\dots,\rho_\ell') = \rmI_\mu(\rho_1,\dots,\rho_k) + \sum_{i=1}^k\Big(\int \rho_i\,\d\mu\Big)\cdot \rmI_{\mu_{|\rho_i}}\Big(\Big(\frac{\rho'_j}{\rho_i}\Big)_{j \in J_i}\Big).
\end{equation}
An immediate corollary is worth noting in itself: whenever there is a partition $J_1$, \dots, $J_k$ of $\{1,2,\dots,\ell\}$ for which~\eqref{eq:refine} holds, we must have
\begin{equation}\label{eq:refine-more-inf}
	\rmI_\mu(\rho_1',\dots,\rho_\ell') \geq \rmI_\mu(\rho_1,\dots,\rho_k).
	\end{equation}

\begin{rmk}
	KL divergence is defined for pairs of measures on general measurable spaces.  As a result, we may use~\eqref{eq:I} to extend the definition of mutual information to general pairs of random variables, not necessarily discrete.  The basic properties of mutual information still hold, with proofs that require only small adjustments. Mutual information is already defined and studied this way in the original reference~\cite{KullLei51}.  See~\cite[Section 2.1]{Pin64} for a comparison with other definitions at this level of generality.
	
	The main theorems of the present paper do not need this extra generality, but KL divergence is central to the background ideas in Section~\ref{sec:meas-conc} below, where the more general setting is appropriate.
	\end{rmk}

\subsection{A classical variational principle}

Let $(X,\mu)$ be a probability space and let $f \in L^\infty(\mu)$.  Let $\langle \cdot\rangle$ denote integration with respect to $\mu$, and let
\[M_\mu(f) := \langle \rme^f\rangle \quad \hbox{and} \quad C_\mu(f) := \log M_\mu(f).\]
Regarding $f$ as a random variable, the functions of $t\in \bbR$ given by $M_\mu(tf)$ and $C_\mu(tf)$ are its moment generating function and cumulant generating function, respectively.

\begin{dfn}
For any $\mu$ and $f$ as above, the associated \textbf{Gibbs measure} is the probability measure
\[\mu_{|\rme^f} = \frac{\rme^f}{\langle \rme^f\rangle}\cdot \mu.\]
\end{dfn}

Note that the normalizing constant in this definition is equal to $M_\mu(f)$.

The importance of Gibbs measures derives from the following classical variational principle.  It is essentially Gibbs' own principle from his work on the foundations of statistical physics, phrased in modern, abstract terms.  Rigorous mathematical treatments go back at least to Kullback's work~\cite{Kull54} in information theory and Sanov's~\cite{San61} on large deviations.  We include the short proof for completeness, since we need the result at one point in Subsection~\ref{subs:BobGot} below.  Closely related results appear in~\cite[Theorem 12.1.1 and Problem 12.2]{CovTho06}.

\begin{lem}\label{lem:maxent}
For $\mu$ and $f$ as above, the function of $\nu \in \Pr(X)$ given by the expression
\[\rmD(\nu\,\|\,\mu) - \int f\,\d\nu\]
achieves its unique minimum at $\nu = \mu_{|\rme^f}$, where it is equal to $-C_\mu(f)$.
\end{lem}

\begin{proof}
If $\nu$ is not absolutely continuous with respect to $\mu$, then $\rmD(\nu\,\|\,\mu) = +\infty$, and therefore $\nu$ is not a candidate for minimizing the expression in question.  So suppose that $\nu$ is absolutely continuous with respect to $\mu$.

Since $\mu$ and $\mu_{|\rme^f}$ are mutually absolutely continuous, $\nu$ is also absolutely continuous with respect to $\mu_{|\rme^f}$.  Let $\rho$ be the Radon--Nikodym derivative $\d\nu/\d\mu_{|\rme^f}$.  Then
\[\frac{\d\nu}{\d\mu}= \rho\cdot \frac{\rme^f}{\langle \rme^f\rangle},\]
where $\langle \cdot\rangle$ denotes integration with respect to $\mu$, as before. Substituting this into the definition of $\rmD(\nu\,\|\,\mu)$, we obtain
\begin{align*}
\rmD(\nu\,\|\,\mu) &= \int \log \Big(\rho \frac{\rme^f}{\langle \rme^f\rangle}\Big)\ \d\nu\\
&= \int \log \rho\ \d\nu + \int f\ \d\nu - \log \langle \rme^f\rangle\\
&= \rmD(\nu\,\|\,\mu_{|\rme^f}) + \int f\ \d\nu - C_\mu(f).
\end{align*}
Re-arranging, this becomes
\[\rmD(\nu\,\|\,\mu) - \int f\,\d\nu = \rmD(\nu\,\|\,\mu_{|\rme^f}) - C_\mu(f).\]
The term $\rmD(\nu\,\|\,\mu_{|\rme^f})$ is non-negative, and zero if and only if $\nu = \mu_{|\rme^f}$.
\end{proof}

\section{Total correlation and dual total correlation}\label{sec:DTC}

Let $\xi = (\xi_1,\dots,\xi_n)$ be a tuple of finite-valued random variables on a probability space $(\O,\F,P)$.  If $n=2$ then the mutual information $\rmI(\xi_1\,;\,\xi_2)$ provides a canonical way to quantify the dependence between them.  But this has no single generalization to the cases $n > 2$. Rather, a range of options are available, suitable for different purposes.  Early studies of these options, and the many relations between them, include~\cite{McG54},~\cite{Wat60} and~\cite{Han75}.  An idea of the breadth of these options can be obtained from~\cite{Crooks--mutinfnote} and the references given there.

Two different quantities of this kind play crucial roles in this paper.

\subsection{Total correlation}

A simple and natural choice for `multi-variate mutual information' is the difference
\[\sum_{i=1}^n \rmH(\xi_i) - \rmH(\xi).\]
This is always non-negative by the subadditivity of Shannon entropy.  It goes back at least to Watanabe's paper~\cite{Wat60}, so following him we call it the \textbf{total correlation}.  It agrees with mutual information if $n=2$. It depends only on the joint distribution $\mu$ of the tuple $\xi$, so we may denote it by either $\TC(\xi)$ or $\TC(\mu)$.  This leaves the role of the $n$ separate coordinates to the reader's understanding.

If $A$ is a finite set and $\mu$ is a probability measure on $A^n$, then $\TC(\mu)$ is precisely the quantity $E$ that appears in Theorems B and C.  It is the key feature of such a finite-dimensional distribution that we are able to control when we apply Theorem C during the proof of Theorem A.  This control is exerted in Subsection~\ref{subs:preA1}.

The total correlation of $\mu$ appears in various classical estimates of probability theory and statistics.  For instance, it is central to Csisz\'ar's approach to conditional limit theorems for product measures~\cite{Csi84}, where it is the quantity that must be bounded to prove convergence in information.  In these applications it is often written in the alternative form
\begin{equation}\label{eq:TC2}
\TC(\mu) = \rmD(\mu\,\|\,\mu_{\{1\}}\times \cdots \times \mu_{\{n\}}),
\end{equation}
where $\mu_{\{i\}}$ is the $i^{\rm{th}}$ marginal of $\mu$ on $A$.  This equation for $\TC(\mu)$ generalizes~\eqref{eq:I}.  The generalization is easily proved by induction on $n$, applying~\eqref{eq:I} at each step.

Since total correlation is given by a difference of entropy values, it is neither convex nor concave as a function of $\mu$.  However, we do have the following useful approximate concavity.

\begin{lem}\label{lem:TC-approx-concave}
	Consider a finite mixture
	\[\mu = p_1\mu_1 + p_2\mu_2 + \dots + p_k\mu_k\]
	of probability measures on $A^n$.  Then
	\begin{equation}\label{eq:near-convex}
		\sum_{j=1}^k p_j\rm{TC}(\mu_j) \leq \TC(\mu) + \rmH(p_1,\dots,p_k).
	\end{equation}
\end{lem}

\begin{proof}
	Let $\xi_i:A^n \to A$ be the $i^{\rm{th}}$ coordinate projection for $1 \leq i \leq n$.  On the one hand, the concavity of $\rmH$ gives
	\[\sum_{j=1}^k p_j\rmH_{\mu_j}(\xi_i) \leq \rmH_\mu(\xi_i) \quad \hbox{for each}\ i=1,2,\dots,n.\]
	On the other hand, Lemma~\ref{lem:near-convex} gives
	\[\sum_{j=1}^k p_j\rmH(\mu_j) \geq \rmH(\mu) - \rmH(p_1,\dots,p_k).\]
	The desired result is a linear combination of these inequalities.
\end{proof}

\begin{cor}\label{cor:cond-and-TC}
	If $\mu$ is a probability measure on $A^n$, $U\subseteq A^n$, and $\mu(U) > 0$, then
	\[\TC(\mu_{|U}) \leq \frac{1}{\mu(U)}(\TC(\mu) + \log 2).\]
\end{cor}

\begin{proof}
	Simply apply Lemma~\ref{lem:TC-approx-concave} to the decomposition
	\[\mu = \mu(U)\cdot \mu_{|U} + \mu(A^n\setminus U)\cdot \mu_{|A^n\setminus U},\]
	then drop the term involving $\mu_{|A^n\setminus U}$ and recall that any two-set partition has entropy at most $\log 2$.
\end{proof}

A simple example shows that one cannot hope for any useful inequality that reverses~\eqref{eq:near-convex}.

\begin{ex}\label{ex:DTC-far-convex}
	For $p \in [0,1]$, let $\nu_p$ be the $p$-biased distribution on $\{0,1\}$.  Now let
	\[\mu := \frac{1}{2}(\nu_p^{\times n} + \nu_q^{\times n})\]
	for some distinct $p$ and $q$.  Then $\mu_{\{i\}} = \nu_{(p+q)/2}$ for each $1 \leq i \leq n$, and so
	\[\TC(\mu) = n\rmH\Big(\frac{p+q}{2},1 - \frac{p+q}{2}\Big) - \rmH\Big(\frac{\nu_p^{\times n} + \nu_q^{\times n}}{2}\Big).\]
	By Lemma~\ref{lem:near-convex}, we have
	\begin{equation}\label{eq:H-mix-small}
\rmH\Big(\frac{\nu_p^{\times n} + \nu_q^{\times n}}{2}\Big) \leq \log 2 + \frac{\rmH(p,1-p)n + \rmH(q,1-q)n}{2}.
\end{equation}
	Therefore
	\[\TC(\mu) \geq \Big[\rmH\Big(\frac{p+q}{2},1 - \frac{p+q}{2}\Big) - \frac{\rmH(p,1-p) + \rmH(q,1-q)}{2}\Big]n - \log 2.\]
	This grows linearly with $n$, since $p$ and $q$ are distinct and $\rmH$ is strictly concave.  But $\mu$ is an average of just two product measures, each of which has total correlation zero. \qed
\end{ex}

\subsection{Dual total correlation}

In this subsection we start to use the notation
\[[n] := \{1,2,\dots,n\}.\]
It reappears often later in the paper.

Our second multi-variate generalization of mutual information is the quantity
\begin{equation}\label{eq:Shearer}
\rmH(\xi) - \sum_{i=1}^n \rmH(\xi_i\,|\,\xi_{[n]\setminus i}),
\end{equation}
where
\[\xi_{[n]\setminus i} := (\xi_1,\dots,\xi_{i-1},\xi_{i+1},\dots,\xi_n).\]
When $n=2$, this again agrees with the mutual information $\rmI(\xi_1\,;\,\xi_2)$.  This generalization seems to originate in Han's paper~\cite{Han75}, where it is called the \textbf{dual total correlation}.  Like total correlation, it depends only on the joint distribution $\mu$ of $\xi$. We denote it by either $\DTC(\xi)$ or $\DTC(\mu)$.

By applying the chain rule to the first term in~\eqref{eq:Shearer} we obtain
\begin{align*}
\DTC(\xi) &= \sum_{i=1}^n\rmH(\xi_i\,|\,\xi_1,\dots,\xi_{i-1}) - \sum_{i=1}^n \rmH(\xi_i\,|\,\xi_{[n]\setminus i})\\
&=\sum_{i=1}^n\big[\rmH(\xi_i\,|\,\xi_1,\dots,\xi_{i-1}) - \rmH(\xi_i\,|\,\xi_{[n]\setminus i})\big].
\end{align*}
All these terms are non-negative, since Shannon entropy cannot increase under extra conditioning, so $\DTC \geq 0$.

The ergodic theoretic setting produces measures with control on their $\TC$.  The relevance of $\DTC$ to our paper is less obvious.  However, we find below that $\DTC(\mu)$ is more easily related than $\TC(\mu)$ to concentration properties of $\mu$.  As a result, we first prove an analog of Theorem B with $\DTC(\mu)$ in place of $\TC(\mu)$: Theorem~\ref{thm:bigdecomp2} below.  We then convert this into Theorem B itself by showing that a bound on $\TC(\mu)$ implies a related bound on $\DTC(\mu')$, where $\mu'$ is a projection of $\mu$ onto a slightly lower-dimensional product of copies of $A$: see Lemma~\ref{lem:trimming} below.  The first of these two stages is the more subtle.  It is based on a `decrement' argument for dual total correlation, in which we show that if $\mu$ fails to be concentrated then it may be written as a mixture of two other measures whose dual total correlations are strictly smaller on average.  This special property of $\DTC$ is the key to the whole proof, and it seems to have no analog for $\TC$.  This is why we need the two separate stages described above.

Now suppose that $\zeta$ is another finite-valued random variable on the same probability space as the tuple $\xi$. Then the \textbf{conditional dual total correlation} of the tuple $\xi$ given $\zeta$ is obtained by conditioning all the Shannon entropies that appear in~\eqref{eq:Shearer}:
\[\DTC(\xi\,|\,\zeta) := \rmH(\xi\,|\,\zeta) - \sum_{i=1}^n \rmH(\xi_i\,|\,\xi_{[n]\setminus i},\zeta).\]

\begin{lem}\label{lem:Shearer-change}
	In the setting above, we have
	\begin{equation}\label{eq:new-I-top}
	\DTC(\xi) - \DTC(\xi\,|\,\zeta) = \rmI(\xi\,;\,\zeta) - \sum_{i=1}^n\rmI(\xi_i\,;\,\zeta\,|\,\xi_{[n]\setminus i}).
	\end{equation}
\end{lem}

\begin{proof}
	This follows from these standard identities:
	\[\rmH(\xi) - \rmH(\xi\,|\,\zeta) = \rmI(\xi\,;\,\zeta)\]
	and
	\[\rmH(\xi_i\,|\,\xi_{[n]\setminus i}) - \rmH(\xi_i\,|\,\xi_{[n]\setminus i},\zeta) = \rmI(\xi_i\,;\,\zeta\,|\,\xi_{[n]\setminus i}).\]
\end{proof}

We sometimes refer to the right-hand side of~\eqref{eq:new-I-top} as the \textbf{DTC decrement} associated to $\xi$ and $\zeta$.  This quantity can actually take either sign, but it is positive in the cases that we need later.

\begin{rmk}
	For each $i$, another appeal to the chain rule gives
	\[\rmH(\xi) = \rmH(\xi_{[n]\setminus i}) + \rmH(\xi_i\,|\,\xi_{[n]\setminus i}).\]
	Using this to substitute for $\rmH(\xi_i\,|\,\xi_{[n]\setminus i})$ in $\DTC(\xi)$ and simplifying, we obtain
	\[\DTC(\xi) = \sum_{i=1}^n \rmH(\xi_{[n]\setminus i}) - (n-1)\rmH(\xi).\]
	In this form, the non-negativity of $\DTC(\xi)$ is a special case of some classic inequalities proved by Han in~\cite{Han78} (see also~\cite[Section 17.6]{CovTho06}).
	
	Han's inequalities were part of an abstract study of non-negativity among various multivariate generalizations of mutual information.  Shearer proved an even more general inequality at about the same time with a view towards combinatorial applications, which then appeared in~\cite{ChungGraFraShe86}.  For any subset $S = \{i_1,\dots,i_\ell\} \subseteq [n]$, let us write $\xi_S = (\xi_{i_1},\dots,\xi_{i_\ell})$.  If $\mathscr{S}$ is a family of subsets of $[n]$ with the property that every element of $[n]$ lies in at least $k$ members of $\mathscr{S}$, then Shearer's inequality asserts that
	\begin{equation}\label{eq:gen-Shearer}
	\sum_{S \in \mathscr{S}}\
	\rmH(\xi_S) \geq k\rmH(\xi).
	\end{equation}
	A simple proof is given in the survey~\cite{Radha01}, together with several applications to combinatorics.
	
	When $\mathscr{S}$ consists of all subsets of $[n]$ of size $n-1$, the inequality~\eqref{eq:gen-Shearer} is just the non-negativity of $\DTC$.  However, in principle one could use any family $\mathscr{S}$ and integer $k$ as above to define a kind of mutual information among the random variables $\xi_1$, \dots, $\xi_n$: simply use the gap between the two sides in~\eqref{eq:gen-Shearer}.
	
	I do not see any advantage to these other quantities over $\DTC$ for the purposes of this paper.  But in settings where the joint distribution of $(\xi_1,\dots,\xi_n)$ has some known special structure, it could be worth exploring how that structure is detected by the right choice of such a generalized mutual information, and what consequences it has for other questions about those random variables.
\end{rmk}

\begin{rmk}
The sum of conditional entropies that we subtract in~\eqref{eq:Shearer} has its own place in the information theory literature.  Verd\'u and Weissman~\cite{VerWei08} call it the `erasure entropy' of $(\xi_1,\dots,\xi_n)$, and use it to define the limiting erasure entropy rate of a stationary ergodic source.  They derive operational interpretations in connection with (i) Gibbs sampling from a Markov random field and (ii) recovery of a signal passed through a binary erasure channel in the limit of small erasure probability. Their paper also makes use of the sum of conditional mutual informations that appears in the second right-hand term in~\eqref{eq:new-I-top}, which they call `erasure mutual information'.  For this quantity they give an operational interpretation in terms of the decrease in channel capacity due to sporadic erasures of the channel outputs.

Concerning erasure entropy, see also Proposition~\ref{prop:erasure-rate} below.
	\end{rmk}

\subsection{The relationship between total correlation and dual total correlation}

Although $\TC(\mu)$ and $\DTC(\mu)$ both quantify some kind of correlation among the coordinates under $\mu$, they can behave quite differently.

\begin{ex}\label{ex:prod-mix}
	Let $\nu_p$ be the $p$-biased distribution on $\{0,1\}$, let $p \neq q$, and let $\mu = (\nu_p^{\times n} + \nu_q^{\times n})/2$, as in Example~\ref{ex:DTC-far-convex}.  In that example $\TC(\mu)$ is at least ${c(p,q)n - \log 2}$, where
	\[c(p,q) = \rmH\Big(\frac{p+q}{2},1 - \frac{p+q}{2}\Big) - \frac{\rmH(p,1-p) + \rmH(q,1-q)}{2} > 0.\]
	
On the other hand, once $n$ is large, the two measures $\nu_p^{\times (n-1)}$ and $\nu_q^{\times (n-1)}$ are very nearly disjoint.  More precisely, there is a partition $(U,U^{\rm{c}})$ of $\{0,1\}^{n-1}$ such that $\nu_q^{\times (n-1)}(U)$ and $\nu_p^{\times (n-1)}(U^{\rm{c}})$ are both exponentially small in $n$.  (Finding the best such partition is an application of the classical Neyman--Pearson lemma in hypothesis testing~\cite[Theorem 11.7.1]{CovTho06}.) It follows that $\mu\{\xi_{[n-1]}\in U\}$ and $\mu\{\xi_{[n-1]} \in U^{\rm{c}}\}$ are both exponentially close to $1/2$. Using Bayes' theorem and this partition, we obtain that the conditional distribution of $\xi_i$ given the event ${\{\xi_{[n]\setminus i} = \bf{z}\}}$ is very close to $\nu_p$ if $\bf{z} \in U$, and very close to $\nu_q$ if $\bf{z} \in U^{\rm{c}}$. In both of these approximations the total variation of the error is exponentially small in $n$.  Therefore, re-using the estimate~\eqref{eq:H-mix-small},
\begin{align*}
\DTC(\mu) &= \rmH_\mu(\xi) - \sum_{i=1}^n\rmH_\mu(\xi_i\,|\,\xi_{[n]\setminus i})\\
&\leq \log 2 + \frac{\rmH(p,1-p)n + \rmH(q,1-q)n}{2} \\
&\quad - n\cdot \mu\{\xi_{[n-1]} \in U\}\cdot \rmH(p,1-p) - n\cdot \mu\{\xi_{[n-1]} \in U^{\rm{c}}\}\cdot \rmH(q,1-q)\\
&\quad  + O(n\rme^{- c'n})\\
&\leq \log 2 + O(n\rme^{- c'n})
\end{align*}
for some $c' > 0$.

Thus, in this example, $\TC(\mu)$ is larger than $\DTC(\mu)$ by roughly a multiplicative factor of $n$. \qed
\end{ex}

\begin{ex}\label{ex:TC-DTC-different}
	Let $A := \bbZ/q\bbZ$ for some $q$, and let $\mu$ be the uniform distribution on the subgroup
	\[U:= \{(a_1,\dots,a_n) \in A^n:\ a_1 +\cdots + a_n = 0\}.\]
	Let $\xi = (\xi_1,\dots,\xi_n)$ be the $n$-tuple of $A$-valued coordinate projections. Then
	\[\rmH(\mu) = \rmH_\mu(\xi) = (n-1)\log q\]
	and
	\[\rmH_\mu(\xi_i) = \log q \quad \hbox{for each}\ i,\]
	so
	\[\TC(\mu) = \log q.\]
	On the other hand, given an element of $U$, any $n-1$ of its coordinates determine the last coordinate, and so
	\[\rmH(\xi_i\,|\,\xi_{[n]\setminus i}) = 0 \quad \hbox{for each}\ i.\]
	Therefore
	\[\DTC(\mu) = (n-1)\log q.\]
	
	In this example, $\DTC(\mu)$ is larger than $\TC(\mu)$ almost by a factor of $n$.  In fact, this example also achieves the largest possible $\DTC$ with an alphabet of size $q$, since one always has
	\begin{align*}
	\DTC(\mu) &= \rmH_\mu(\xi) - \sum_{i=1}^n\rmH_\mu(\xi_i\,|\,\xi_{[n]\setminus i})\\
	&= \rmH_\mu(\xi_{[n-1]}) + \rmH_\mu(\xi_n\,|\,\xi_{[n-1]}) - \sum_{i=1}^n\rmH_\mu(\xi_i\,|\,\xi_{[n]\setminus i})\\
	&= \rmH_\mu(\xi_{[n-1]}) - \sum_{i=1}^{n-1}\rmH_\mu(\xi_i\,|\,\xi_{[n]\setminus i})\\
	&\leq \rmH_\mu(\xi_{[n-1]})\\
	&\leq (n-1)\log q.
	\end{align*}
\qed
\end{ex}

As described in the previous subsection, our route to Theorem B goes through the analogous theorem with $\DTC$ in place of $\TC$: Theorem~\ref{thm:bigdecomp2} below.  After proving Theorem~\ref{thm:bigdecomp2}, we need some estimate relating $\DTC$ and $\TC$ in order to deduce Theorem B.  Example~\ref{ex:TC-DTC-different} shows that this estimate cannot simply bound $\DTC$ by a multiple of $\TC$, since the factor of $n$ which appears in that example is too large to lead to the correct dependence in Theorem B, part (a).

The key here is that a much better bound on $\DTC$ by $\TC$ is available if we allow ourselves to discard a few coordinates in $A^n$.

\begin{lem}[Discarding coordinates to control $\DTC$ using $\TC$]\label{lem:trimming}
	For any $\mu \in \Pr(A^n)$ and $r > 0$ there exists $S \subseteq [n]$ with $|S| \geq (1-r)n$ and such that \[\DTC(\mu_S) \leq r^{-1}\TC(\mu).\]
\end{lem}

Let $\xi = (\xi_1,\dots,\xi_n)$ have joint distribution $\mu$. Beware that, on the left-hand side of the inequality above, $\DTC(\mu_S)$ refers to the dual total correlation among the smaller family of random variabes $(\xi_i)_{i\in S}$.  In the proof below, we use $\DTC$ and $\TC$ for measures on product spaces of several different dimensions.

\begin{proof}
	Let us write $\a := r^{-1}\TC(\mu)/n$ for brevity.
	
	By the definition of $\DTC$ and the subadditivity of $\rmH$, we have
	\begin{equation}\label{eq:DTC-easy-bd}
\DTC(\mu) \leq \sum_{i=1}^n\big[\rmH(\xi_i) - \rmH(\xi_i\,|\,\xi_{[n]\setminus i})\big].
\end{equation}
	If the measure $\mu$ happens to satisfy $\rmH(\xi_i\,|\,\xi_{[n]\setminus i}) \geq \rmH(\xi_i) - \a$ for every $i$, then~\eqref{eq:DTC-easy-bd} immediately gives $\DTC(\mu) \leq \a n$.
	
	So suppose there is an $i_1 \in [n]$ for which $\rmH(\xi_{i_1}\,|\,\xi_{[n]\setminus i_1}) < \rmH(\xi_{i_1}) - \a$.  Then that failure and the chain rule give
	\begin{align*}
		\TC(\mu_{[n]\setminus i_1}) &= \Big(\sum_{i\in [n]\setminus i_1}\rmH(\xi_i)\Big) - \rmH(\xi_{[n]\setminus i_1})\\
		&= \Big(\sum_{i\in [n]\setminus i_1}\rmH(\xi_i)\Big) - \big[\rmH(\xi) - \rmH(\xi_{i_1}\,|\,\xi_{[n]\setminus i_1})\big]\\
		&= \Big(\sum_{i=1}^n\rmH(\xi_i)\Big) - \rmH(\xi) - \big[\rmH(\xi_{i_1}) - \rmH(\xi_{i_1}\,|\,\xi_{[n]\setminus i_1})\big]\\
		&< \TC(\mu) - \a
	\end{align*}
	for the projection $\mu_{[n]\setminus i_1}$ of $\mu$ to $A^{[n]\setminus i_1}$.
	
	We can now repeat the argument above. If every $i \in [n]\setminus i_1$ satisfies
	\[\rmH(\xi_i\,|\,\xi_{[n]\setminus \{i_1,i\}}) \geq \rmH(\xi_i) - \a,\]
	then the $(n-1)$-dimensional analog of~\eqref{eq:DTC-easy-bd} gives $\DTC(\mu_{[n]\setminus i_1}) \leq \a(n-1) \leq \a n$, and we select $S := [n]\setminus i_1$.  Otherwise, there is an $i_2 \in [n]\setminus i_1$ such that
	\[\rmH(\xi_{i_2}) - \rmH(\xi_{i_2}\,|\,\xi_{[n]\setminus \{i_1,i_2\}}) > \a,\]
	and then removing $i_2$ leads to
	\[\TC(\mu_{[n]\setminus \{i_1,i_2\}}) < \TC(\mu_{[n]\setminus i_1}) - \a < \TC(\mu) - 2\a.\]
	
	Since total correlations are non-negative, we cannot repeat this argument more than $k$ times, where $k$ is the largest integer for which
	\[\TC(\mu) - k\a > 0.\]
	From the definition of $\a$, this amounts to $k < r n$.  Thus, after removing at most this many elements from $[n]$, we are left with a subset $S$ of $[n]$ having both the required properties.
\end{proof}

Lemma~\ref{lem:trimming} is nicely illustrated by Example~\ref{ex:TC-DTC-different}. In that example, the projection of $\mu$ to any $n-1$ coordinates is a product measure, so the $\DTC$ of any such projection collapses to zero.

In Subsection~\ref{subs:B} we use Lemma~\ref{lem:trimming} (together with Lemma~\ref{lem:stab-lift} below) to finish the deduction of Theorem B from Theorem~\ref{thm:bigdecomp2}.

\section{Background on metrics and measure transportation}\label{sec:meas-trans}

\subsection{Normalized Hamming metrics}\label{subs:Ham-met}

For any finite set $A$ and $n \in \bbN$, the \textbf{normalized Hamming metric} on $A^n$ is defined by
\begin{equation}\label{eq:norm-Ham}
d_n(\bf{a},\bf{a}') := \frac{|\{i \in [n]:\ a_i\neq a_i'\}|}{n},
\end{equation}
where $\bf{a} = (a_1,\dots,a_n)$, $\bf{a}' = (a_1',\dots,a_n') \in A^n$.  All the metric spaces that appear in Parts II and III have the form $(A^n,d_n)$ or some slight modification of it.  We never use Hamming metrics that are not normalized.

More generally, if $(K,d)$ is any metric space, then $K^n$ may be endowed with the \textbf{normalized Hamming average} of copies of $d$, defined by
\begin{equation}\label{eq:norm-Ham2}
d_n(\bf{x},\bf{x}') := \frac{1}{n}\sum_{i=1}^n d(x_i,x'_i) \quad \forall \bf{x},\bf{x}' \in K^n.
\end{equation}
We refer to this generalization occasionally in the sequel, but our main results concern the special case~\eqref{eq:norm-Ham}.

\begin{rmk}
In the alternative formula~\eqref{eq:TC2} for total correlation, the right-hand side makes sense for a measure $\mu$ on $K^n$ for any measurable space $K$.  This suggests an extension of Theorem B to measures $\mu$ on $(K^n,d_n)$, where $(K,d)$ is a compact metric space and $d_n$ is as in~\eqref{eq:norm-Ham2}. In fact, this generalization is easily obtained from Theorem B itself by partitioning $K$ into finitely many subsets of diameter at most $\delta$, using this partition to quantize the measure $\mu$, and then letting $\delta \to 0$.

On the other hand, various steps in the proof of Theorem B are simpler in the discrete case, particularly once we switch to arguing about $\DTC$ instead of $\TC$.  We therefore leave this generalization of Theorem B aside.
\end{rmk}

\subsection{Transportation metrics}\label{subs:meas-trans}

Let $(K,d)$ be a compact metric space. Let $\Pr(K)$ be the space of Borel probability measures on $K$, and endow it with the transportation metric
\[\ol{d}(\mu,\nu) := \inf\Big\{\int d(x,y)\,\l(\d x,\d y):\ \l\ \hbox{a coupling of}\ \mu\ \hbox{and}\ \nu\Big\}.\]
Let $\Lip_1(K)$ be the space of all $1$-Lipschitz functions $K \to\bbR$.  The transportation metric is indeed a metric on $\Pr(K)$, and it generates the vague topology.  See, for instance,~\cite[Section 11.8]{Dud--book}.

The transportation metric has a long history, and is variously associated with the names Monge, Kantorovich, Rubinstein, Wasserstein and others.  Within ergodic theory, the special case when $(K,d)$ is a Hamming metric space $(A^n,d_n)$ is the finitary version of Ornstein's `d-bar' metric, which Ornstein introduced independently in the course of his study of isomorphisms between Bernoulli shifts.  Here I have adopted a notation based on that connection to ergodic theory.  For a Hamming metric space $(A^n,d_n)$, the associated transportation metric on $\Pr(A^n)$ is $\ol{d_n}$.

The transportation metric is closely related to many branches of geometry and to the field of optimal transportation: see, for instance,~\cite[Section 3$\frac{1}{2}$.10]{Gro99},~\cite{GanMcC96}, and~\cite{Levin90} for more on these connections.  Within ergodic theory, it has also been explored beyond Ornstein's original uses for it, particularly by Vershik and his co-workers~\cite{Ver00,Ver04,Ver10}.

The heart of the theory is the following dual characterization of this metric.

\begin{thm}[Monge--Kantorovich--Rubinstein duality]\label{thm:MKR}
	Any $\mu,\nu \in \Pr(K)$ satisfy
	\[\ol{d}(\mu,\nu) = \sup_{f \in \Lip_1(K)}\Big[\int f\,\d\nu - \int f\,\d\mu\Big].\]
	\qed
\end{thm}

The precise statement of this result seems to originate in the papers~\cite{Kantor42,Kantor48,KantorRubin57}; see~\cite[Theorem 11.8.2]{Dud--book} for a standard modern treatment.

One immediate consequence of Theorem~\ref{thm:MKR} is that $\ol{d}(\nu,\mu)$ is always bounded by $\frac{1}{2}\|\nu-\mu\|$ times the diameter of $(K,d)$.  The following corollary improves this slightly, giving a useful estimate for certain purposes later.

\begin{cor}\label{cor:dbar-and-TV}
Let $\mu,\nu \in \Pr(K)$, and let $B_1$, \dots, $B_m$ be a partition of $K$ into Borel subsets.  Then
\[\ol{d}(\nu,\mu) \leq \sum_{i=1}^m \mu (B_i) \cdot \rm{diam}(B_i) + \frac{1}{2}\Big(\sum_{i=1}^m |\nu (B_i) - \mu (B_i)|\Big)\cdot \rm{diam}(K).\]
\end{cor}

\begin{proof}
Let $f\in \rm{Lip}_1(K)$. For each $i = 1,2,\dots,m$ we have
\begin{align*}
\int_{B_i}f\,\d\nu - \int_{B_i}f\,\d\mu &\leq \nu (B_i)\cdot \sup (f|B_i) - \mu (B_i)\cdot \inf (f|B_i) \\
&= \mu(B_i)\cdot (\sup (f|B_i) - \inf (f|B_i))\\
& \qquad \qquad \qquad \qquad + (\nu(B_i) - \mu(B_i))\cdot \sup (f|B_i)\\
&\leq \mu(B_i)\cdot \rm{diam}(B_i) + |\nu (B_i) - \mu (B_i)|\cdot \|f\|,
\end{align*}
where $\|\cdot\|$ is the supremum norm. Summing over $i$, this gives
\[\int f\,\d\nu - \int f\,\d\mu \leq \sum_{i=1}^m\mu(B_i)\cdot \rm{diam}(B_i) + \sum_{i=1}^m|\nu (B_i) - \mu (B_i)|\cdot \|f\|.\]
Since $\nu(K) = \mu(K) = 1$, the left-hand side here is unchanged if we add any constant value to $f$.  We may therefore shift $f$ by a constant so that $\|f\| \leq \frac{1}{2}\rm{diam}(K)$.  We complete the proof by substituting this bound on $\|f\|$ above and then taking the supremum over all $f$.
\end{proof}

To recover the bound by $\frac{1}{2}\|\nu-\mu\|\rm{diam}(K)$ from this corollary, observe that if we take $B_1$, $B_2$, \dots, $B_m$ to be a partition into sets of very small diameter, then we can make the first right-hand term in Corollary~\ref{cor:dbar-and-TV} as small as we please, and the second term is still always bounded by $\frac{1}{2}\|\nu-\mu\|\rm{diam}(K)$.

Here is another important estimate for our later purposes.  It gives an upper bound on the transportation distance from a measure on a product space to a product of two other measures.  It is well known, but we include a proof for commpleteness.

\begin{lem}\label{lem:dist-from-prod-meas}
Let $(K,d_K)$ and $(L,d_L)$ be compact metric spaces, let $0 < \a < 1$, and let $d$ be the following metric on $K\times L$:
\begin{equation}\label{eq:prod-met}
d((x,y),(x',y')) := \a d_K(x,x') + (1-\a) d_L(y,y').
\end{equation}
Let $\mu \in \Pr(K)$, $\nu \in \Pr(L)$ and $\l \in \Pr(K \times L)$.  Finally, let $\l_K$ be the marginal of $\l$ on $K$, and let $x\mapsto \l_{L,x}$ be a conditional distribution for the $L$-coordinate given the $K$-coordinate under $\l$.

Then
\[\ol{d}(\l,\mu\times \nu) \leq \a \ol{d_K}(\l_K,\mu) + (1-\a)\int_K \ol{d_L}(\l_{L,x},\nu)\,\l_K(\d x).\]
\end{lem}

\begin{proof}
	This can be proved directly from the definition or using Theorem~\ref{thm:MKR}.  The latter approach is slightly shorter.
	
	Let $f:K\times L\to \bbR$ be $1$-Lipschitz for the metric $d$.  Define the function $\ol{f}:K\to \bbR$ by
	\[\ol{f}(x) := \int_L f(x,y)\,\nu(\d y).\]
	
	From~\eqref{eq:prod-met} it follows that the function $f(x,\cdot)$ is $(1-\a)$-Lipschitz on the metric space $(L,d_L)$ for each fixed $x\in K$.  Therefore
	\begin{align*}
		\int_L [f(x,y) - \ol{f}(x)]\,\l_{L,x}(\d y) &= \int_L f(x,y)\,\l_{L,x}(\d y) - \int_L f(x,y)\,\nu(\d y)\\ &\leq (1-\a)\ol{d_L}(\l_{L,x},\nu).
		\end{align*}
	
	Also, the function $\ol{f}$ satisfies
	\[|\ol{f}(x) - \ol{f}(x')| \leq \int_L |f(x,y) - f(x',y)|\,\nu(\d y) \leq \a d_K(x,x') \quad \forall x,x' \in K,\]
	so $\ol{f}$ is $\a$-Lipschitz on $(K,d_K)$. Using the estimates above and Fubini's theorem, we now obtain
	\begin{align*}
\int_{K\times L} f\,\d\l - \int_{K\times L} f\,\d(\mu\times \nu) & = \int_{K\times L} f\,\d\l - \int_K \ol{f}\,\d\l_K + \int_K \ol{f}\,\d\l_K - \int_K \ol{f}\,\d\mu\\
&= \int_K\Big[\int_L [f(x,y) - \ol{f}(x)]\,\l_{L,x}(\d y)\Big]\,\l_K(\d x) \\
&\quad + \int_K \ol{f}\,\d\l_K - \int_K \ol{f}\,\d\mu\\
&\leq (1-\a)\int_K \ol{d_L}(\l_{L,x},\nu)\,\l_K(\d x) + \a\ol{d_K}(\l_K,\mu).
\end{align*}
Taking the supremum over $f$, Theorem~\ref{thm:MKR} completes the proof.
\end{proof}

\subsection{A first connection to total correlation}

The simple lemma below is needed only for one example later in the paper, but it also begins to connect this section with the previous ones.

\begin{lem}\label{lem:near-prod-TC-small}
Let $\nu \in \Pr(A)$, let $\mu\in \Pr(A^n)$, and let $\delta:= \ol{d_n}(\mu,\nu^{\times n})$.  Then
\[\TC(\mu) \leq 2\big(\rmH(\delta,1-\delta) + \delta \log |A|\big)n.\]
\end{lem}

\begin{proof}
	Regard the coordinates in $A^n\times A^n$ as a pair of copies of each of the coordinates in $A^n$. For each $i=1,2,\dots,n$, let $\xi_i:A^n\times A^n\to A$ (respectively, $\zeta_i$) be the projection to the first (respectively, second) copy of the $i^{\rm{th}}$ coordinate in $A^n$. Let $\l \in \Pr(A^n\times A^n)$ be a coupling of $\mu$ and $\nu^{\times n}$ such that
	\[\int d_n\,\d\l = \frac{1}{n}\sum_{i=1}^n\l\{\xi_i\neq \zeta_i\} = \delta.\]
	Let $\delta_i := \l\{\xi_i\neq \zeta_i\}$ for each $i$. The tuple $\xi$ has distribution $\mu$ under $\l$, and the tuple $\zeta$ has distribution $\nu^{\times n}$.
	
	Now the chain rule, monotonicity under conditioning, and several applications of Fano's inequality~\cite[Section 2.10]{CovTho06} give
	\[\rmH_\l(\xi_i) \leq \rmH_\l(\zeta_i) + \rmH(\delta_i,1 - \delta_i) + \delta_i\log|A| \]
	for each $i$, and also
\[\rmH_\l(\zeta) \leq \rmH_\l(\xi) + \sum_{i=1}^n\rmH_\l(\zeta_i\,|\,\xi_i) \leq \rmH_\l(\xi) + \sum_{i=1}^n\big(\rmH(\delta_i,1-\delta_i) + \delta_i\log |A|\big).\]

Therefore
\[\TC(\mu) = \TC(\xi) = \sum_{i=1}^n\rmH_\l(\xi_i) - \rmH_\l(\xi) \leq \TC(\zeta) + 2\sum_{i=1}^n\big(\rmH(\delta_i,1-\delta_i) + \delta_i\log |A|\big). \]
Since $\TC(\zeta) =0$ and $\rmH$ is concave, this is at most
\[2\big(\rmH(\delta,1-\delta) + \delta \log |A|\big)n.\]
\end{proof}

\section{Preliminaries on measure concentration}\label{sec:meas-conc}

Many of the first important applications of measure concentration were to the local theory of Banach spaces.  Classic accounts with an emphasis on these applications can be found in~\cite{MilmanSchech86,Milman88}: in particular,~\cite[Chapter 6]{MilmanSchech86} is an early introduction to the general metric-space framework and gives several examples of it.  Independently, a concentration inequality due to Margulis found an early application to information theory in~\cite{AhlGacKor76}. More recently, Gromov has described a very broad class of concentration phenomena for general metric spaces in~\cite[Chapter 3$\frac{1}{2}$]{Gro01}.  Within probability theory, measure concentration is now a large subject in its own right, with many aspects and many applications.  Ledoux' monograph~\cite{Led--book} is dedicated to these. Some of the most central results for product measures can be found in Talagrand's classic work~\cite{Tal95}, which contributed some very refined estimates in that setting. Concentration now appears as a branch of probability theory in a few advanced textbooks: Vershynin's~\cite{Ver--HDPbook} is a recent example.

Although we need only a few of the important ideas here, they must be altered slightly from their usual presentation, in ways that are explained more carefully below.  I therefore give complete proofs for all but the most classical results in this section.

In this paper, our principal notion of measure concentration takes the form of certain transportation inequalities that we call `T-inequalities': see Definition~\ref{dfn:T} below.  These tie together the metric and measure theoretic ideas of the preceding sections.  Subsections~\ref{subs:trans-ineq} and~\ref{subs:trans-ineq-egs} discuss some history and basic examples.  Then Subsection~\ref{subs:BobGot} introduces an equivalent formulation of T-inequalities in terms of bounds on the exponential moments of $1$-Lipschitz functions.  This equivalent formulation gives an easy way to establish some basic stability properties of T-inequalities.

Although Theorems B and C are focused on T-inequalities, at a key point in the proof of Theorem B we must switch to concentration inequalities of a different kind.  We call them L-inequalities because of their relation to logarithmic Sobolev inequalities.  They are formulated in Definition~\ref{dfn:L}, and the rest of Subsection~\ref{subs:logSob} explains how they are related to T-inequalities.

The rest of this section concerns triples $(K,d,\mu)$ consisting of a compact metric space $(K,d)$ and a measure $\mu \in \Pr(K)$.  We refer to such a triple as a \textbf{metric probability space}. We shall not consider `metric measure spaces' that are not compact or have mass other than $1$; indeed, the main theorems later in the paper all concern metrics on finite sets.

\subsection{Transportation inequalities}\label{subs:trans-ineq}

Here is the definition from the Introduction, placed in a more general setting:

\begin{dfn}\label{dfn:T}
Let $\k,r > 0$. The metric probability space $(K,d,\mu)$ satisfies the \textbf{T-inequality with parameters $\k$ and $r$}, or \textbf{T}$(\k,r)$, if any other $\nu \in \Pr(K)$ satisfies
\[\ol{d}(\nu,\mu) \leq \frac{1}{\k}\rmD(\nu\,\|\,\mu) + r.\]
\end{dfn}

In this definition, `T' stands for `transportation', but beware that other authors use `T' for a variety of different inequalities involving transportation metrics. Indeed, Definition~\ref{dfn:T} may be seen as a linearization of a more standard family of transportation inequalities: those of the form
\begin{equation}\label{eq:SQT}
\ol{d}(\nu,\mu) \leq \sqrt{\frac{1}{\k}\rmD(\nu\,\|\,\mu)}.
\end{equation}
Here we refer to~\eqref{eq:SQT} as a \textbf{square-root transportation inequality}.  The square root is a natural and important feature in many settings where such inequalities have been proved, such as product measures and Hamming metrics (discussed further below).

The papers~\cite{Mar96} and~\cite{BobGot99} include good accounts of how square-root transportation inequalities relate to other notions of concentration.  Those accounts are easily adapted to the T-inequalities of Definition~\ref{dfn:T}.  A survey of recent developments on a wide variety of transportation inequalities can be found in~\cite{GozLeo10}.

If $(K,d,\mu)$ satisfies~\eqref{eq:SQT}, then it also satisfies a whole family of T-inequalities as in Definition~\ref{dfn:T}.  This is because the inequality of arithmetic and geometric means gives
\begin{equation}\label{eq:AMGM}
\sqrt{\frac{1}{\k}\rmD(\nu\,\|\,\mu)} \leq \frac{1}{4\k r}\rmD(\nu\,\|\,\mu) + r 
\end{equation}
for any $r > 0$.  However, in our Theorems B and C it is important that we fix $r > 0$ in advance and then seek an inequality of the form T$(\k n,r)$ for some $\k$.  This fixed value of $r$ may be viewed as a `distance cutoff': the T-inequality we obtain is informative only about distances greater than or equal to $r$.  Example~\ref{ex:need-r} in the next subsection shows that this limitation is necessary.

If $(K,d,\mu)$ has diameter at most $1$ then it clearly satisfies T$(\k,1)$ for every $\k > 0$. Such a space also satisfies T-inequalities with arbitrarily small values of $r$.  To see this, first observe from the diameter bound that
\[\ol{d}(\nu,\mu) \leq \frac{1}{2}\|\nu - \mu\| \quad \hbox{for any}\ \nu,\mu \in \Pr(K).\]
Next, Pinsker's classical inequality~\cite[p15]{Pin64} (with the optimal constant established later in~\cite{Csi67},~\cite{Kull67} and~\cite[Theorem 6.11]{Kem69}) provides
\begin{equation}\label{eq:Pinsker}
\|\mu - \nu\| \leq \sqrt{2\rmD(\nu\,\|\,\mu)}.
\end{equation}
Combining these inequalities with~\eqref{eq:AMGM}, we obtain that $(K,d,\mu)$ satisfies T$(8r,r)$ for every $r > 0$.

If an inequality T$(\k,r)$ improves on this trivial estimate, then its strength lies in the trade-off between $r$ and $\k$.  Product spaces with Hamming metrics and product measures provide classical examples of stronger T-inequalities.  Let $(K,d)$ be a metric space, and suppose for simplicity that its diameter is at most $1$.  Let $d_n$ be the metric on $K^n$ given by~\eqref{eq:norm-Ham2}. Starting with the simple estimate T$(8r,r)$ for any probability measure on $(K,d)$, one shows by induction on $n$ that any product measure on $(K^n,d_n)$ satisfies a T-inequality with constants that improve as $n$ increases.  The inductive argument amounts to combining Lemma~\ref{lem:dist-from-prod-meas} with the chain rule for KL divergence.  This elegant approach to measure concentration for product spaces is the main contribution of Marton~\cite{Mar86} (with improvements and generalizations in~\cite{Mar96}).  The result is as follows.

\begin{prop}[See~{\cite[Lemma 1]{Mar86}} and~{\cite[eqn (1.4)]{Mar96}}]\label{prop:Marton}
	If $(K,d)$ has diameter at most $1$, if $d_n$ is given by~\eqref{eq:norm-Ham2}, and if $\mu$ is a product measure on $K^n$, then
	\begin{equation}\label{eq:Marton}
	\ol{d_n}(\nu,\mu) \leq \sqrt{\frac{1}{2n}\rmD(\nu\,\|\,\mu)} \quad \forall \nu \in \Pr(K^n).
	\end{equation}
	In particular, combining with~\eqref{eq:AMGM}, the space $(K^n,d_n,\mu)$ satisfies T$(8 r n,r)$ for all $r > 0$. \qed
	\end{prop}

Marton's paper~\cite{Mar96} also describes how these inequalities may be seen as sharper and cleaner versions of estimates that Ornstein already uses in his work on isomorphisms between Bernoulli shifts.

In view of the equality~\eqref{eq:TC2}, Proposition~\ref{prop:Marton} implies a kind of reversal of Lemma~\ref{lem:near-prod-TC-small}.

We use Proposition~\ref{prop:Marton} directly at one point in the proof of Theorem C below.  It is also a strong motivation for the specific form of the T-inequalities obtained in Theorems B and C.

Square-root transportation inequalities have been proved for various other probability measures on metric spaces, such as Gaussian measures in Euclidean spaces~\cite{Tal96b}. Marton and others have also extended her work to various classes of weakly dependent stochastic processes: see, for instance,~\cite{Mar96b,Mar98,Dem97,Sam00}.

\subsection{Two families of examples}\label{subs:trans-ineq-egs}

The next examples show that in Theorem B we must fix the distance cutoff $r$ in advance.

\begin{ex}\label{ex:need-r}
Fix a sequence of values $1 > \delta_1 > \delta_2 > \dots$ such that
\begin{equation}\label{eq:deltas}
	\delta_n\to 0 \quad \hbox{but} \quad \sqrt{n}\delta_n \to \infty.
	\end{equation}
For each $n$, let $C_n \subseteq \{0,1\}^n$ be a code of maximal cardinality such that its minimum distance is at least $\delta_n$ in the normalized Hamming metric.  See, for instance,~\cite[Chapter 4]{Welsh--book} for these notions from coding theory.  By the Gilbert--Varshamov bound~\cite[Section 4.2]{Welsh--book} and simple estimates on Hamming-ball volumes, we have
\begin{equation}\label{eq:Cn-fast}
|C_n| \geq 2^{(1 - o(1))n}.
\end{equation}
Letting $\mu_n$ be the uniform distribution on $C_n$, it follows that
\begin{equation}\label{eq:TC-slow}
\TC(\mu_n) \leq n\log 2 - \log |C_n|  = o(1).
\end{equation}

Now let $\nu$ be any measure which is absolutely continuous with respect to $\mu_n$.  Then $\nu$ is supported on $C_n$.  Fix $\k > 0$.  Assume that $\nu$ satisfies the square-root transportation inequality with constant $\k n$: that is, the analog of~\eqref{eq:SQT} holds with $\nu$ as the reference measure instead of $\mu$ and with $\k n$ in place of $\k$.  If $n$ is sufficiently large, then we can deduce from this that $\nu$ must put most of its mass on a single point.  To do this, consider any subset $U \subseteq C_n$ such that $\nu(U) \leq 1/2$, and set $a:= \nu(U)$.  Any coupling of $\nu_{|U}$ and $\nu$ must transport mass at least $a$ from $U$ to $U^{\rm{c}}$.  Since any two elements of $C_n$ are at least distance $\delta_n$ apart, this implies
\begin{equation}\label{eq:dn-delta}
\ol{d_n}(\nu,\nu_{|U}) \geq a\delta_n.
\end{equation}
On the other hand, since $\nu(U) \geq a$, a simple calculation gives
\[\rmD(\nu_{|U}\,\|\,\nu) = \log \frac{1}{\nu(U)} \leq \log \frac{1}{a}.\]
Therefore our assumed square-root transportation inequality gives
\[\ol{d_n}(\nu,\nu_{|U}) \leq \sqrt{\frac{\log(1/a)}{\k}}\cdot\frac{1}{\sqrt{n}}.\]
By the second part of~\eqref{eq:deltas}, this is compatible with~\eqref{eq:dn-delta} only if $a$ is $o(1)$ as $n\to\infty$. It follows that no choice of partition $(U,U^\rm{c})$ can separate $\nu$ into two substantial pieces once $n$ is large, and hence that $\nu$ must be mostly supported on a single point.

Now suppose we have a mixture which represents $\mu_n$ and with most of the weight on terms that satisfy a good square-root transportation inequality. Then the argument above shows that those terms must be close to delta masses.  This is possible only if the number of terms is roughly $|C_n|$.  By~\eqref{eq:Cn-fast} and~\eqref{eq:TC-slow}, this number is much larger than $\exp(O(\TC(\mu_n)))$.

This example shows that one cannot hope to prove Theorem B with a square-root transportation inequality in place of T$(\k n,r)$.  Simply by choosing $\delta_n$ to decay sufficiently slowly, many similar possibilities can also be ruled out: for instance, replacing the square root with an even smaller power cannot repair the problem.

Of course, this example does not contradict Theorem B itself. For any \emph{fixed} $r > 0$, the lower bound in~\eqref{eq:dn-delta} is less than $r$ once $n$ is large enough, and for those $n$ the terms in the mixture provided by Theorem B need not consist of delta masses.  Indeed, I suspect that the classical `random' construction of good codes $C_n$ (see again~\cite[Chapter 4]{Welsh--book}) yields measures $\mu_n$ that already satisfy some inequality of the form T$(\k n,r)$, provided $n$ is large enough in terms of $r$.  It could be interesting to try to prove this.
 \qed
\end{ex}

In view of Proposition~\ref{prop:Marton}, it is worth emphasizing that many measures other than product measures also exhibit concentration.  The next examples are very simple, but already hint at the diversity of such measures.

\begin{ex}\label{ex:prod-img-meas}
	Let $n = 2\ell$ be even, let $A$ be a finite set, and let $B:= A\times A$.  Then we have an obvious bijection
\begin{multline*}
F:B^\ell\to A^n:\big((a_{1,1},a_{1,2}),(a_{2,1},a_{2,2}),\dots,(a_{\ell,1},a_{\ell,2})\big) \\ \mapsto (a_{1,1},a_{1,2},a_{2,1},a_{2,2},\dots,a_{\ell,1},a_{\ell,2}).
\end{multline*}
	Let $\nu \in \Pr(B)$ and let $\mu:= F_\ast\nu^{\times \ell}$.
	
 If $d_\ell$ and $d_n$ denote the normalized Hamming metrics on $B^\ell$ and $A^n$ respectively, then $F$ is $1$-Lipschitz from $d_\ell$ to $d_n$.  This is because each coordinate in $B^\ell$ gives two output coordinates under $F$, but the normalization constant in $d_n$ is twice that in $d_\ell$.  It follows that $\mu$ inherits any T-inequality that is satisfied by $\nu^{\times \ell}$, for instance by using Proposition~\ref{prop:BobGot} below. Since Proposition~\ref{prop:Marton} gives that $\nu^{\times \ell}$ satisfies T$(8r\ell,r)$, which we may write as T$(4rn,r)$, the same is true of $\mu$.
	
However, if $\nu$ is not a product measure on $A\times A$, then $\mu$ is not truly a product measure on the product space $A^n$.  In fact, $\mu$ can be far away from any true product measure on $A^n$. In simple cases we can deduce this from Lemma~\ref{lem:near-prod-TC-small}. For instance, suppose that $\nu$ is the uniform distribution on the diagonal $\{(a,a):\ a\in A\}$.  Then
\[\rmH(\mu) = \rmH(\nu^{\times \ell}) = \ell\log |A| = \frac{n}{2}\log |A|.\]
But on each coordinate in $A^n$, $\mu$ projects to the uniform distribution.  Therefore
	\[\sum_{i=1}^n\rmH(\mu_{\{i\}}) = n\log|A|.\]
	Putting these calculations together, we obtain
	\[\TC(\mu) = n\log|A|/2.\]
	In view of Lemma~\ref{lem:near-prod-TC-small}, this implies a uniform lower bound on $\ol{d_n}(\mu,\mu_0^{\times n})$ for any $\mu_0 \in \Pr(A)$.
	
	The last example also satisfies $\rmH_\mu(\xi_i\,|\,\xi_{[n]\setminus i}) = 0$ for every $i$, since each coordinate in $A\times A$ determines the other under the diagonal distribution $\nu$.  Therefore we also have
	\[\DTC(\mu) = \rmH(\mu) - 0 = n\log|A|/2.\]
	
	Other choices of $\nu$ lead to even more striking properties.  Let us sketch some of these without giving complete details.  Assume $|A|$ is very large, and let $C \subseteq A\times A$ have size $|A|^2/2$ and be $\eps$-uniform for some very small $\eps$, in the sense of quasirandomness for bipartite graphs (see, for instance,~\cite[Section IV.5]{Bol98}).  For any fixed $\eps  > 0$, if $|A|$ is large enough, then a random choice of $C$ is $\eps$-uniform with high probability. Let $\nu$ be the uniform distribution on $C$, and let $\mu := F_\ast\nu^{\times \ell}$.
	
	An easy calculation gives that $\TC(\mu)$ is at most $(n/2)\log 2$.  In fact this is roughly an equality, up to a multiple of $n$ which is small depending on the uniformity of $C$, and $\DTC(\mu)$ is also roughly $(n/2)\log 2$ up to a similar error.
	
	We need the following consequence of uniformity: if $M$ is any large constant, and if $C$ is sufficiently uniform in terms of $M$, then there is a small constant $c > 0$ depending only on $M$ such that
	\begin{equation}\label{eq:quasirand}
(\mu_1\times \mu_2)((A\times A)\setminus C)\geq c
\end{equation}
	whenever $\rmH(\mu_1)$ and $\rmH(\mu_2)$ are both at least $\log |A| - M$. To see this, let $\g$ be the uniform distribution on $A$, and recall that with this choice we have $\rmD(\mu_i\,\|\,\g) = \log|A| - \rmH(\mu_i)$.  Therefore the assumed lower bound on $\rmH(\mu_i)$ implies the upper bound
	\[\rmD(\mu_i\,\|\,\g) = \sum_{a \in A}\g(a)\frac{\mu_i(a)}{\g(a)}\log\frac{\mu_i(a)}{\g(a)} \leq M \quad \hbox{for}\ i=1,2.\]
	Using the strict convexity of the function $t\mapsto t\log t$, one obtains from this bound a positive constant $c_1 \geq 1$, depending only on $M$, such that the set
	\[D_i := \{a:\ \g(a)/c_1 \leq \mu_i(a) \leq c_1\g(a)\}\]
	satisfies $\mu_i(D_i) \geq 1/c_1$ for each $i=1,2$.  In view of the definition of $D_i$, this set must also satisfy
	\begin{equation}\label{eq:Di-big}
	\g(D_i) = \frac{|D_i|}{|A|}\geq \frac{1}{c_1}\mu_i(D_i) \geq \frac{1}{c^2_1}.
	\end{equation}
	Because of these sets, we have
	\begin{equation}\label{eq:quasirand2}
(\mu_1\times \mu_2)((A\times A)\setminus C) \geq (\mu_1\times \mu_2)((D_1\times D_2)\setminus C) \geq \frac{|(D_1\times D_2)\setminus C|}{c_1^2|A|^2}.
\end{equation}
	If $C$ is sufficiently uniform in terms of $c_1$, then it and its complement must both intersect the product set $D_1\times D_2$ in roughly half its elements.  The uniformity can be applied here because~\eqref{eq:Di-big} includes a fixed lower bound on the ratios $|D_i|/|A|$. It then follows that the fraction at the end of~\eqref{eq:quasirand2} is at least
	\[\frac{|D_1\times D_2|}{4c_1^2|A|^2} \geq \frac{1}{4c_1^4},\]
	by another appeal to~\eqref{eq:Di-big}.  This proves~\eqref{eq:quasirand} with $c := 1/4c_1^4$.
	
	Having arranged~\eqref{eq:quasirand}, we now consider the function on $A^n$ defined by
	\[f(\bf{a}):= \frac{1}{n}\big|\big\{i \in \{1,2,\dots,\ell\}:\ (a_{2i-1},a_{2i}) \not\in C\big\}\big|.\]
It is $1$-Lipschitz for the metric $d_n$.  Using~\eqref{eq:quasirand}, Markov's inequality, and some simple double counting, one finds another small but absolute constant $c'$ such that the following holds: for any true product measure $\mu_1\times \cdots \times \mu_n$ on $A^n$, if
\[\rmH(\mu_1\times \cdots \times \mu_n) = \rmH(\mu_1) + \cdots + \rmH(\mu_n) \geq n\log|A| - nM/100,\]
then also
\[\int f\,\d(\mu_1\times \cdots \times \mu_n) \geq c'.\]
Therefore, by Theorem~\ref{thm:MKR} applied with the $1$-Lipschitz function $f$, we have
\begin{equation}\label{eq:prod-far}
\ol{d_n}(\mu_1\times \cdots\times \mu_n,\mu') \geq c'
\end{equation}
for any measure $\mu'$ which is supported on $F(C^\ell) = \{f=0\}$.

On the other hand, $\mu := F_\ast\nu^{\times \ell}$ is supported on $F(C^\ell)$.  If we write it as a mixture of any other measures, they must all be supported on $F(C^\ell)$ too.  So now suppose we have such a mixture
\[\mu = p_1\mu_1 + \cdots + p_m\mu_m\]
with $m = \exp(O(\TC(\mu))) = \exp(O(n))$.  Since
\[\rmH(\mu) = \ell\log|C| = n\log|A| - (n/2)\log 2,\]
and since any measure on $A^n$ has entropy at most $n\log|A|$, Lemma~\ref{lem:near-convex} implies that most of the mass in the mixture is on terms $\mu_j$ for which $\rmH(\mu_j)> n\log|A| - O(n)$, where `$O(n)$' is independent of $|A|$.  In particular, if we previously chose $M$ large enough, then this lower bound on entropy is greater than $n\log|A| -nM/100$. Combining this with~\eqref{eq:prod-far}, we deduce that most of the mass in this mixture is on measures $\mu_j$ that are at least distance $c'$ from any product measure.

This gives an example $\mu$ that is highly concentrated, but cannot be written as a mixture of measures that are close in $\ol{d_n}$ to product measures and where the number of terms is $\exp(O(\TC(\mu))$.  This shows that the `good' summands in Theorem B cannot be limited to measures that are close in $\ol{d_n}$ to products --- other examples of highly concentrated measures are sometimes necessary. \qed
\end{ex}

The ideas in Example~\ref{ex:prod-img-meas} can be generalized in many ways.  For instance, many other subsets and measures can be constructed by imposing more complicated pairwise constraints on the coordinates in $A^n$, or by imposing constraints on triples or even larger sets of coordinates.  I expect that many such constructions exhibit measure concentration, although certainly many others do not.

However, all the cases in which I know how to prove measure concentration share a certain structure: like the measures $F_\ast\nu^{\times \ell}$ in Example~\ref{ex:prod-img-meas}, they are $O(1)$-Lipschitz images of product measures on other product spaces.  In that example, we use the map $F$ to `hide' the product structure of $\nu^{\times \ell}$, with the effect that $F_\ast\nu^{\times \ell}$ is far away from any product measure on $A^n$. But one might feel that this still has the essence of a product measure, and it suggests the following general question.

\begin{ques}\label{ques:near-prod-img}
	Is it true that for any $r,\k,\eps > 0$ there exists an $L > 0$ for which the following holds for all $n \in \bbN$?
	\begin{quote}
If $(A^n,d_n,\mu)$ satisfies T$(\k n,r)$, then there exist an integer $\ell \in [n/L,Ln]$, a finite set $B$, a distribution $\nu$ on $B$, and an $L$-Lipschitz map $F:B^\ell\to A^n$ such that
\[\ol{d_n}(\mu,F_\ast\nu^{\times \ell}) < \eps.\]
\end{quote}
	\end{ques}

The methods of the present paper do not seem to shed any light on this question.  An answer in either direction would be very interesting.  If the answer is Yes, then we have found a rather strong characterization of concentrated  measures on $(A^n,d_n)$.  If it is No, then finding a counterexample seems to require finding a new property of product measures which (i) survives under pushforward by Lipschitz maps but (ii) is not implied by concentration alone.

If Question~\ref{ques:near-prod-img} has a positive answer, then it could be regarded as an analog of a result about stationary processes: the equivalence between measure concentration of the finite-dimensional distributions and being a factor of a Bernoulli shift (see~\cite[Subsection IV.3.c]{Shi96}).  However, Question~\ref{ques:near-prod-img} assumes much less structure than that result.  This, in turn, could have consequences for the ergodic theory of non-amenable acting groups, for which factors of Bernoulli systems are far from being understood.

\subsection{Exponential moment bounds and some consequences}\label{subs:BobGot}

Using Monge--Kantorovich--Rubinstein duality, a T-inequality is equivalent to an exponential moment bound for $1$-Lipschitz functions.  This fact is essentially due to Bobkov and G\"{o}tze~\cite{BobGot99}, although they work with square-root transportation inequalities.  The statement and proof in our linearized setting are easily obtained by modifying their argument.

\begin{prop}[Bobkov--G\"{o}tze equivalence]\label{prop:BobGot}
For a metric probability space $(K,d,\mu)$, the following are equivalent:
\begin{itemize}
	\item[(a)] The space $(K,d,\mu)$ satisfies T$(\k,r)$;
	\item[(b)] Any $1$-Lipschitz function $f$ satisfies
	\[M_\mu(\k f) \leq \rme^{\k\langle f\rangle + \k r}.\]
\end{itemize}
\end{prop}

\begin{proof}
	Theorem~\ref{thm:MKR} implies that condition (a) is equivalent to
\[\k\int f\,\d\nu - \k\int f\,\d\mu \leq \rmD(\nu\,\|\,\mu) + \k r\]
for all $\nu \in \Pr(K)$ and $1$-Lipschitz $f$.  Re-arranging, this is equivalent to
\begin{equation}\label{eq:equiv-ineq}
- \k\int f\,\d\mu \leq \k r + \Big(\rmD(\nu\,\|\,\mu) - \k\int f\,\d\nu\Big)
\end{equation}
for all such $\nu$ and $f$.

For a given $f$, Lemma~\ref{lem:maxent} asserts that the difference
\[\rmD(\nu\,\|\,\mu) - \k\int f\,\d\nu\]
is minimized by the Gibbs measure $\nu = \mu_{|\rme^{\k f}}$, and for this choice of $\nu$ that difference is equal to $- C_\mu(\k f)$. Therefore~\eqref{eq:equiv-ineq} is equivalent to the inequality
\[C_\mu(\k f) \leq  \k \int f\,\d\mu + \k r.\]
Exponentiating, we arrive at condition (b).
\end{proof}

Bobkov and G\"{o}tze's original result~\cite[Theorem 1.3]{BobGot99} asserts that the square-root transportation inequality~\eqref{eq:SQT} is equivalent to the inequality
\[M_\mu(tf) \leq \rme^{t\langle f\rangle + t^2/4\k} \]
for all $1$-Lipschitz $f:K\to \bbR$ and all $t \in \bbR$.  This inequality is also easily linearized to arrive at part (b) of Proposition~\ref{prop:BobGot}.

A first application of Bobkov and G\"{o}tze's equivalence is a second proof of Proposition~\ref{prop:Marton}.  Let $(K^n,d_n,\mu)$ be as in that proposition. By the equivalence, it suffices to prove that any $1$-Lipschitz function $f:K^n\to\bbR$ and $t\in \bbR$ satisfy
\begin{equation}\label{eq:BobGot-moment}
\int \rme^{tf}\,\d\mu \leq \rme^{t\langle f\rangle + t^2/8n}.
\end{equation}
Because $\mu$ is a product measure, this can be proved by the method of bounded martingale differences, using the filtration $(\F_i)_{i=1}^n$ of $K^n$ where $\F_i$ consists of the Borel sets that depend on only the first $i$ coordinates.  One can also phrase this proof as an induction on $n$, using a conditional version of~\eqref{eq:BobGot-moment} with $n=1$ at each step.  The result when $n=1$ is a classical lemma of Hoeffding~\cite{Hoe63}, and its extension to the setting of $1$-Lipschitz functions on product spaces is due to McDiarmid: see~\cite[Lemma 5.8]{McD89} and the arguments that follow it.  By a special case of the equivalence in Proposition~\ref{prop:BobGot}, that lemma of Hoeffding is equivalent to Pinsker's inequality~\eqref{eq:Pinsker} with the optimal constant.

This use of the inequality~\eqref{eq:BobGot-moment} fits into a more general formalism of bounding the `Laplace functional' of a metric probability space: see~\cite[Section 1.6]{Led--book}.

The rest of this subsection gives three stability results for T-inequalities under different kinds of perturbation to a measure.  These lemmas can all be proved directly from Definition~\ref{dfn:T}, but the proofs become simpler and quicker using condition (b) from Proposition~\ref{prop:BobGot}.

\begin{lem}\label{lem:RN-flat}
	Suppose that $(K,d,\mu)$ satisfies T$(\k,r)$, and let $\nu \in \Pr(K)$ be absolutely continuous with respect to $\mu$ and satisfy
	\[\frac{\d\nu}{\d\mu} \leq M \quad \mu\hbox{-a.s.}\]
	for some $M \in [1,\infty)$.  Then $(K,d,\nu)$ satisfies
	\[\rm{T}\Big(\k,2\frac{\log M}{\k} + 2r \Big).\]
	\end{lem}

\begin{proof}
We verify the relevant version of condition (b) from Proposition~\ref{prop:BobGot}.  Let $\langle\cdot\rangle_\mu$ and $\langle\cdot\rangle_\nu$ denote integration with respect to $\mu$ and $\nu$, respectively.  Let $f \in \rm{Lip}_1(K)$. Then
\[\int \rme^{\k f}\,\d\nu = \int \rme^{\k f}\frac{\d\nu}{\d\mu}\,\d\mu \leq M\int \rme^{\k f}\,\d\mu = \rme^{\k(\log M/\k)}\int \rme^{\k f}\,\d\mu.\]
By assumption, this is at most $\exp(\k\langle f\rangle_\mu + \k r + \k(\log M/\k))$. On the other hand, we have
\[\rmD(\nu\,\|\,\mu) = \int \log \frac{\d\nu}{\d\mu}\,\d\nu \leq \log M,\]
and so the original definition of T$(\k,r)$ gives
\[\langle f\rangle_\mu \leq \langle f\rangle_\nu + \ol{d}(\nu,\mu) \leq \langle f\rangle_\nu + \frac{1}{\k}\log M + r.\]
Combining these inequalities completes the proof.
\end{proof}

\begin{lem}\label{lem:infty-transport}
Suppose that $(K,d,\mu)$ satisfies T$(\k,r)$, let $\nu \in \Pr(K)$, and assume there exists a coupling $\l$ of $\mu$ and $\nu$ which is supported on the set
\[\{(x,y) \in K\times K:\ d(x,y) \leq \delta\}.\]
Then $\nu$ satisfies T$(\k,r+2\delta)$.
\end{lem}

\begin{proof}
We verify the relevant version of condition (b) from Proposition~\ref{prop:BobGot}. Let $\langle\cdot\rangle_\mu$ and $\langle\cdot\rangle_\nu$ denote integration with respect to $\mu$ and $\nu$, respectively.  Then
\[\int \rme^{\k f}\,\d\nu = \int_{K\times K}\rme^{\k f(y)}\,\l(\d x,\d y).\]
Since $d(x,y) \leq \delta$ for $\l$-almost every $(x,y)$, and $f$ is $1$-Lipschitz, the last integral is bounded above by
\[\int_{K\times K}\rme^{\k f(x) + \k\delta}\,\l(\d x,\d y) = \rme^{\k\delta}\int \rme^{\k f}\,\d\mu.\]
By assumption, this is at most $\exp(\k\langle f\rangle_\mu + \k r + \k\delta)$. Finally, a repeat of the same reasoning about $\l$ and $f$ gives $\langle f\rangle_\mu \leq \langle f\rangle_\nu + \delta$.  Combining these inequalities completes the proof.
\end{proof}

\begin{prop}\label{prop:T-under-dbar}
Suppose that $(K,d,\mu)$ satisfies T$(\k,r)$, let $\nu \in \Pr(K)$, and assume that $\ol{d}(\nu,\mu) \leq \delta^2$ for some $\delta \in [0,1/8)$.  Then there exists a Borel set $U \subseteq K$ satisfying $\nu(U) \geq 1-4\delta$ and such that $\nu_{|U}$ satisfies
\[\rm{T}\Big(\k,\frac{8\delta + 2\log 4}{\k} + 4r + 4\delta\Big).\]
\end{prop}

\begin{proof}
Let $\l$ be a coupling of $\mu$ and $\nu$ such that $\int d\,\d\l \le \delta^2$, and let
\[W := \{(x,y):\ d(x,y)\leq \delta\}.\]
Then Markov's inequality gives $\l(W) \geq 1-\delta$. Let $\l_1 := \l_{|W}$, and let $\mu_1$ and $\nu_1$ be the first and second marginals of $\l_1$, respectively.

For these new measures, we have
\[\frac{\d\l_1}{\d\l} = \frac{1}{\l(W)}\cdot 1_W \leq (1-\delta)^{-1} \leq 1 + 2\delta,\]
and hence also $\d\mu_1/\d\mu \leq 1+2\delta$. Therefore, by Lemma~\ref{lem:RN-flat}, the measure $\mu_1$ still satisfies T$(\k,r_1)$, where
\[r_1 := 2\frac{\log (1+2\delta)}{\k} + 2r \leq 4\delta/\k + 2r.\]

Then, since $\l_1$ is supported on $W$, Lemma~\ref{lem:infty-transport} promises that $\nu_1$ still satisfies T$(\k,r_2)$, where $r_2 = r_1 + 2\delta$.

Let $\rho := \d\nu_1/\d\nu$, and let $f:= 1 + 2\delta - \rho$. Since $\rho \leq 1 + 2\delta$, $f$ is non-negative.  On the other hand,
\[\int f\,\d\nu = 1 + 2\delta - 1 = 2\delta.\]
Therefore another appeal to Markov's inequality gives that the set
\[U := \{\rho \geq 1/2+2\delta\} = \{f \leq 1/2\}\]
satisfies $\nu(U) \geq 1 - 4\delta$.

Finally, observe that
\[\nu_{|U} = \frac{1}{\nu(U)}\cdot 1_U\cdot \nu \leq \frac{1}{1 - 4\delta}\cdot (2\rho)\cdot\nu = \frac{2}{1 - 4\delta}\cdot \nu_1 \leq 4\nu_1.\]
By a second appeal to Lemma~\ref{lem:RN-flat}, this implies that $\nu_{|U}$ satisfies T$(\k,r_3)$, where
\[r_3 = 2r_2 + 2\frac{\log 4}{\k} = 2(r_1+2\delta) + 2\frac{\log 4}{\k} \leq \frac{8\delta + 2\log 4}{\k} + 4r + 4\delta.\]
\end{proof}

\begin{rmks}$\phantom{i}$
\begin{enumerate}
\item Different estimates can be obtained by choosing different thresholds for $\rho$ when defining the set $U$ in the proof above.  The choice we have made here is simple, but not canonical.  This version suits our later purposes quite well, because in those cases $\k$ is very large while $r$ is only fairly small, so the expression with $\k$ in the denominator is also very small.

\item The qualitative conclusion of Proposition~\ref{prop:T-under-dbar} cannot be substantially improved: in particular, we cannot conclude that $\nu$ itself satisfies T$(\k,r)$ for some controlled value of $r$.  To see this, suppose that $K$ has diameter $1$, and that $\nu$ equals $(1-\delta^2)\mu + \delta^2 \mu'$, where $\mu$ satisfies T$(\k,\delta)$ for some huge $\k$, but $\mu'$ is another measure whose support is at distance $1$ from the support of $\mu$.  Then $\ol{d}(\nu,\mu) = \delta^2$, but that small multiple of $\mu'$ which appears in $\nu$ is hard to transport to the rest of $\nu$.  This prevents $\nu$ from satisfying T$(\k',r')$ with $r' \ll 1$ and $\k' \gg 1/\delta^2$, no matter how large $\k$ is.
\end{enumerate}
\end{rmks}

We finish this subsection with another stability property.  This one applies to measures on Hamming spaces.  It shows that T-inequalities survive when we lift those measures to spaces of slightly higher dimension.

\begin{lem}[Stability under lifting]\label{lem:stab-lift}
Let $\mu$ be a probability measure on a Hamming metric space $(A^n,d_n)$.  Let $S$ be a nonempty subset of $[n]$, let $\mu_S$ be the projection of $\mu$ to $A^S$, and define $a \in [0,1)$ by $|S| = (1-a)n$.  If $\mu_S$ satisfies T$(\k,r)$, then $\mu$ itself satisfies
\[\rm{T}\Big(\frac{\k}{1-a},(1-a)r + a\Big).\]
\end{lem}

\begin{proof}
Let $d_S$ be the normalized Hamming metric on $A^S$. Suppose that $\nu \in \Pr(A^n)$ has $\rmD(\nu\,\|\,\mu) < \infty$, and let $\nu_S$ be the projection of $\nu$ to $A^S$.  Then $\rmD(\nu_S\,\|\,\mu_S) \leq \rmD(\nu\,\|\,\mu)$, because $\d\nu_S/\d\mu_S$ is the conditional $\mu$-expectation of $\d\nu/\d\mu$ onto the $\s$-algebra generated by the coordinates in $S$, and so we may apply the conditional Jensen inequality in the definition~\eqref{eq:dfn-D}.  Therefore our assumption on $\mu_S$ gives
\[\ol{d_S}(\nu_S,\mu_S) \leq \frac{1}{\k}\rmD(\nu\,\|\,\mu) + r.\]
Let $\l_S$ be a coupling of $\nu_S$ and $\mu_S$ satisfying $\int d_S\,\d\l_S = \ol{d_S}(\nu_S,\mu_S)$, and let $\l$ be any lift of $\l_S$ to a coupling of $\nu$ and $\mu$.  Then
\begin{align*}
\int d_n\ \d\l &= \frac{|S|}{n}\int d_S\ \d\l_S + \frac{n-|S|}{n}\int d_{[n]\setminus S}(\bf{x}_{[n]\setminus S},\bf{y}_{[n]\setminus S})\ \l(\d\bf{x},\d\bf{y}) \\ &\leq (1-a)\Big(\frac{1}{\k}\rmD(\nu_S\,\|\,\mu_S) + r\Big) + a\\ &\leq (1-a)\Big(\frac{1}{\k}\rmD(\nu\,\|\,\mu) + r\Big) + a.
\end{align*}
\end{proof}

\subsection{A related functional inequality}\label{subs:logSob}

\begin{dfn}\label{dfn:L}
Fix $(K,d)$ as before, and let $\a > 0$ and $0 \leq \k_0 \leq \k \leq \infty$.  Then $\mu \in \Pr(K)$ satisfies the \textbf{L-inequality with parameters $\k_0$, $\k$ and $\a$}, or \textbf{L$([\k_0,\k],\a)$}, if
\begin{equation}\label{eq:LSL}
\rmD(\mu_{|\rme^{tf}}\,\|\,\mu) \leq \a t^2
\end{equation}
for any $1$-Lipschitz function $f$ on $K$ and any real value $t \in [\k_0,\k]$.
\end{dfn}

This definition is closely related to another important branch of measure concentration theory: logarithmic Sobolev inequalities.  Put roughly, given a metric probability space $(K,d,\mu)$ and a suitable class of functions on $K$, a logarithmic Sobolev inequality asserts that
\begin{equation}\label{eq:log-Sob}
\rmD(\mu_{|\rme^f}\,\|\,\mu) \leq \a \int |\nabla f|^2\,\d\mu_{|\rme^f}
\end{equation}
for all $f$ in that class, where $|\nabla f|(x)$ is some notion of the `local gradient' of $f$ at the point $x \in K$.  If $(K,d)$ is a Riemannian manifold, then $|\nabla f|$ is often the norm of the true gradient function $\nabla f$ on $K$.  For various other metric spaces $(K,d)$, including many discrete spaces, alternative notions are available.

Logarithmic Sobolev inequalities form a large area of study in functional analysis in their own right.  A good introduction to these inequalities and their relationship to measure concentration can be found at the beginning of~\cite{BobGot99}, which also gives many further references.  Several more recent developments are covered in~\cite[Section 8]{GozLeo10}.

Product spaces exhibit certain logarithmic Sobolev inequalities that improve with dimension, and these offer a third alternative route to Proposition~\ref{prop:Marton}: see~\cite[Section 4]{Led9597}.  This approach, more than Marton's original one, played an important role in the discovery of the proof of Theorem B. This is discussed further in Subsection~\ref{subs:whyDTC}.

In general, if $f$ is Lipschitz, then any sensible choice for $|\nabla f|$ should be bounded by the Lipschitz constant of $f$.  Thus, although we do not introduce any quantity that plays the role of $|\nabla f|$ in this paper, we can regard~\eqref{eq:LSL} as a stunted logarithmic Sobolev inequality. This is the reason for choosing the letter L.

L-inequalities imply T-inequalities. This link is one of the key tools in our proof of Theorem B below.  To explain it, we begin with the following elementary estimate, which is also needed again later by itself.

\begin{lem}\label{lem:D-est}
Let $(X,\mu)$ be a probability space, let $a < b$ be real numbers, and let $f:X\to\bbR$ be a measurable function satisfying $a \leq f \leq b$ almost surely. Then
\[\rmD(\mu_{|\rme^f}\,\|\,\mu) \leq (b-a)^2.\]
\end{lem}

\begin{proof}
Both sides of the desired inequality are unchanged if we add a constant to $f$, so we may assume that $a = 0$.  Let $\langle \cdot \rangle$ denote integration with respect to $\mu$.  Since we are now assuming that $f \geq 0$ almost surely, we must have $\langle \rme^f\rangle \geq 1$, and therefore
\begin{align*}
\int f\rme^f\,\d\mu &= \int f(\rme^f - \langle \rme^f\rangle)\,\d\mu + \langle f\rangle\langle \rme^f\rangle\\
&\leq \int f(\rme^f - 1)\,\d\mu + \langle f\rangle\langle \rme^f\rangle.
\end{align*}
Since $\rme^f - 1 \leq f\rme^f$ (for instance, by the convexity of $\exp$), the above implies that
\[\int f\rme^f\,\d\mu \leq \int f^2\rme^f\,\d\mu + \langle f\rangle \langle \rme^f\rangle \leq b^2\langle \rme^f\rangle + \langle f\rangle \langle \rme^f\rangle.\]
Dividing by $\langle \rme^f\rangle$ and re-arranging, this inequality completes the proof.
\end{proof}

Lemma~\ref{lem:D-est} and its proof are well known, but I do not know their origins, and have included them for completeness.  The proof actually gives the stronger inequality
\[\rmD(\mu_{|\rme^f}\,\|\,\mu) \leq \int(f - \min f)^2\,\d\mu_{|\rme^f}.\]
This version may be regarded as a logarithmic Sobolev inequality of the kind in~\eqref{eq:log-Sob}.  For the present paper we need only the simpler statement in Lemma~\ref{lem:D-est}.

Now we show how to turn an L-inequality into a T-inequality.  Intuitively, we find that an L-inequality is a `differential version' of a T-inequality.  This is made precise by the following proof, which follows the classic Herbst argument from the study of logarithmic Sobolev inequalities, with some slight adjustments to the present setting. See~\cite{AidMasShi94,Led95} for early expositions of Herbst's original (unpublished) argument.

\begin{prop}\label{prop:LSL-T}
If $(K,d)$ has diameter at most $1$ and $0 < r < \k < \infty$, then L$([r,\k],r/\k)$ implies $T(\k,2r)$.
\end{prop}

\begin{proof}
Let $f$ be $1$-Lipschitz, and define $\phi:[0,\k]\to \bbR$ by
\[\phi(t) := \left\{\begin{array}{ll} \langle f\rangle &\quad \hbox{if}\ t=0\\ \frac{1}{t}C_\mu(tf) & \quad \hbox{if}\ 0 < t \leq \k.\end{array}\right.\]
A simple argument shows that $\phi$ is continuous at $0$, and a little calculus gives
\[\phi'(t) = \frac{1}{t^2}\Big(t\int f\,\d\mu_{|\rme^{tf}} - C_\mu(tf)\Big) = \frac{1}{t^2}\rmD(\mu_{|\rme^{tf}}\,\|\,\mu) \quad \hbox{for}\ t > 0,\]
where the second equality comes from Lemma~\ref{lem:maxent}. Since $(K,d)$ has diameter at most $1$ and $f$ is $1$-Lipschitz, Lemma~\ref{lem:D-est} gives $\phi'(t) \leq 1$ for any $t$.  On the other hand, L$([r,\k],r/\k)$ gives $\phi'(t) \leq r/\k$ whenever $r \leq t \leq \k$.  Combining these bounds, we obtain
\begin{align*}
	C_\mu(\k f) = \k\phi(\k) &= \k\phi(0) + \k \int_0^\k\phi'(t)\,\d t\\
	&= \k\phi(0) + \k \int_0^r\phi'(t)\,\d t + \k \int_r^\k\phi'(t)\,\d t\\
	&\leq \k\phi(0) + 2\k r = \k\langle f\rangle + 2\k r.
\end{align*}
This completes the proof via part (b) of Proposition~\ref{prop:BobGot}.
\end{proof}

\subsection{The role of the inequalities in the rest of this paper}

T-inequalities already appear in the formulation of Theorem B.  However, the proof of that theorem involves L-inequalities as well.  Put very roughly, in part of the proof of that theorem, we assume that a given measure $\mu$ does not already satisfy T$(r n/1200,r)$, deduce from Proposition~\ref{prop:LSL-T} that a related L-inequality also fails, and then use the function $f$ and parameter $t$ from \emph{that} failure to write $\mu$ as a mixture of two `better' measures.  This part of the proof of Theorem B is the subject of the next section.  I do not know an approach that works directly with the violation of the T-inequality, and avoids the L-inequality.

\section{A relative of Theorem B}\label{sec:aux-decomp}

The proof of Theorem B follows two other decomposition results, which come progressively closer to the conclusion we really want.  The biggest part of the proof goes towards the first of those auxiliary decompositions.

In the rest of this section, we always use $\langle\cdot\rangle$ to denote expectation with respect to a measure $\mu$ on $A^n$.  Here is the first auxiliary decomposition:

\begin{thm}\label{thm:bigdecomp1}
Let $\eps,r > 0$.  For any $\mu \in \Pr(A^n)$ there exists a fuzzy partition $(\rho_1,\dots,\rho_k)$ such that
\begin{enumerate}
	\item[(a)] $\rmI_\mu(\rho_1,\dots,\rho_k) \leq 2\cdot \DTC(\mu)$ (recall~\eqref{eq:Imu} for the notation on the left here),
	\item[(b)] $\langle \rho_1\rangle < \eps$, and
	\item[(c)] $\langle \rho_j\rangle > 0$ and the measure $\mu_{|\rho_j}$ satisfies T$(r n/200,r)$ for every $j=2,3,\dots,k$.
	\end{enumerate}
\end{thm}

This differs from Theorem B in two respects.  Firstly, it gives control only over the mutual information $\rmI_\mu(\rho_1,\dots,\rho_k)$, not the actual number of terms $k$.  Secondly, that control is in terms of $\DTC(\mu)$, not $\TC(\mu)$ as in the statement of Theorem B.  The use of $\DTC$ here is crucial. Constructing the fuzzy partition in Theorem~\ref{thm:bigdecomp1} requires careful estimates on the behaviour of $\DTC$, and no suitable estimates seem to be available for $\TC$.

In the next section we improve Theorem~\ref{thm:bigdecomp1} to Theorem~\ref{thm:bigdecomp2}, which does control the number of functions in the fuzzy partition, but still using $\DTC(\mu)$.  Finally, we complete the proof of Theorem B using Lemma~\ref{lem:trimming}, which finds a slightly lower-dimensional projection of $\mu$ whose $\DTC$ is controlled in terms of $\TC(\mu)$.  I do not know a more direct way to use $\TC(\mu)$ for the proof of Theorem B.

\subsection{The DTC-decrement argument}\label{subs:DTC-dec-intro}

The key idea for the proof of Theorem~\ref{thm:bigdecomp1} is this: if $\mu$ itself does not satisfy T$(r n/200,r)$, then we can turn that failure of concentration into a representation of $\mu$ as a mixture whose terms have substantially smaller $\DTC$ on average.  In this mixture, if much of the mass is on pieces that still fail T$(r n/200,r)$, then those can be decomposed again.  This process cannot continue indefinitely, because the average $\DTC$ of the terms in these mixtures must remain non-negative.  Thus, after a finite number of steps, we reach a mixture whose terms mostly do satisfy T$(r n/200,r)$.  Moreover, our specific estimate on the decrement in average $\DTC$ at each step turns into a bound on the mutual information of this mixture: up to a constant, it is at most the total reduction we achieved in the average $\DTC$, which in turn is at most the initial value $\DTC(\mu)$.

In Subsection~\ref{subs:whyDTC}, we try to give an intuitive reason for expecting that such an argument can be carried out using $\DTC$.  I do not know of any reason to hope for a similar argument using $\TC$.  The reader may wish to consult Subsection~\ref{subs:whyDTC} before finishing the rest of this section.

Here is the proposition that provides this decrement in average $\DTC$:

\begin{prop}\label{prop:dec}
If $\mu \in \Pr(A^n)$ does not satisfy T$(r n/200,r)$, then there is a fuzzy partition $(\rho_1,\rho_2)$ such that
\[\rmI_\mu(\rho_1,\rho_2) \geq r^2n^{-1}\rme^{-n}\]
and
\begin{equation}\label{eq:decrement}
\langle \rho_1\rangle \cdot \DTC(\mu_{|\rho_1}) + \langle\rho_2\rangle \cdot \DTC(\mu_{|\rho_2}) \leq \DTC(\mu) - \frac{1}{2}\rmI_\mu(\rho_1,\rho_2).
\end{equation}
\end{prop}

It is worth noting an immediate corollary: if a measure $\mu$ has $\DTC(\mu) < r^2n^{-1}\rme^{-n}/2$, then $\mu$ must already satisfy T$(rn/200,r)$, simply because the left-hand side of~\eqref{eq:decrement} cannot be negative.  This corollary can be seen as a robust version of concentration for product measures. However, since the value $r^2n^{-1}\rme^{-n}/2$ is so small as $n$ grows, I doubt that this version has any advantages over more standard results. For instance, the combination of Propositions~\ref{prop:Marton} and~\ref{prop:T-under-dbar} shows that if $\TC(\mu)$ is $o(n)$ then we can trim off a small piece of $\mu$ and be left with strong measure concentration.  By another of Han's inequalities from~\cite{Han78}, $\TC(\mu)$ is always at most $(n-1)\cdot \DTC(\mu)$, so we obtain the same conclusion if $\DTC(\mu)$ is $o(1)$.  Some recent stronger results are cited in Subsection~\ref{subs:EFKY} below. Qualitatively, this is the best one can do: the mixture of two product measures in Example~\ref{ex:prod-mix} has $\DTC$ roughly $\log 2$ and does not exhibit any non-trivial concentration. The particular inequality $\DTC(\mu) < r^2n^{-1}\rme^{-n}/2$ does not play any role in the rest of this paper.

Most of this section is spent proving Proposition~\ref{prop:dec}.  Before doing so, let us explain how it implies Theorem~\ref{thm:bigdecomp1}.  This implication starts with the following corollary.

\begin{cor}\label{cor:dec}
Let $\mu \in \Pr(A^n)$, and let $(\rho_i)_{i=1}^k$ be a fuzzy partition. Let $S$ be the set of all $i \in [k]$ for which $\mu_{|\rho_i}$ does not satisfy T$(r n/200,r)$.  Then there is another fuzzy partition $(\rho'_j)_{j=1}^\ell$ with the following properties:
\begin{enumerate}
\item[(a)] (growth in mutual information)
\[\rmI_\mu(\rho_1',\dots,\rho_\ell') \geq  \rmI_\mu(\rho_1,\dots,\rho_k) + r^2n^{-1}\rme^{-n}\sum_{i \in S}\langle \rho_i\rangle;\]
\item[(b)] (proportional decrease in average $\DTC$)
\begin{multline*}
\sum_{j=1}^\ell\langle \rho'_j\rangle \cdot \DTC(\mu_{|\rho'_j}) \\ \leq \sum_{i=1}^k \langle \rho_i\rangle \cdot \DTC(\mu_{|\rho_i}) - \frac{1}{2}\big[\rmI_\mu(\rho_1',\dots,\rho_\ell') - \rmI_\mu(\rho_1,\dots,\rho_k)\big].
\end{multline*}
\end{enumerate}
\end{cor}

\begin{proof}
	For each $i \in S$, let $(\rho_{i,1},\rho_{i,2})$ be the fuzzy partition obtained by applying Proposition~\ref{prop:dec} to the measure $\mu_{|\rho_i}$.  Now let $(\rho_j')_{j=1}^\ell$ be an enumeration of the following functions:
	\[\hbox{all}\ \rho_i \ \hbox{for}\ i \in [k]\setminus S \ \hbox{and all} \ \rho_i\cdot \rho_{i,1} \ \hbox{and} \ \rho_i\cdot \rho_{i,2} \ \hbox{for}\ i \in S.\]
	These functions satisfy
	\[\sum_{i\in [k]\setminus S}\rho_i + \sum_{i\in S}\rho_i\cdot \rho_{i,1} + \sum_{i\in S}\rho_i\cdot \rho_{i,2} = \sum_{i\in [k]\setminus S}\rho_i + \sum_{i\in S}\rho_i = 1,\]
	so $(\rho'_j)_{j=1}^\ell$ is a new fuzzy partition.
	
	Let us show that this new fuzzy partition has properties (a) and (b). Both of these follow from the version of the chain rule for mutual information and fuzzy partitions presented in equation~\eqref{eq:fuzzy-chain}.  In the present setting, that rule gives
	\begin{equation}\label{eq:fuzzy-chain-again}
\rmI_\mu(\rho'_1,\dots,\rho'_\ell) - \rmI_\mu(\rho_1,\dots,\rho_k) = \sum_{i\in S}\langle \rho_i\rangle\cdot  \rmI_{\mu_{|\rho_i}}(\rho_{i,1},\rho_{i,2}).
\end{equation}
	By the first conclusion of Proposition~\ref{prop:dec}, this is at least
	\[r^2 n^{-1}\rme^{-n}\sum_{i\in S}\langle \rho_i\rangle.\]
	This proves property (a).
	
	On the other hand, the second conclusion of Proposition~\ref{prop:dec} gives
	\begin{multline*}
\Big(\int \rho_{i,1}\,\d\mu_{|\rho_i}\Big)\cdot \DTC(\mu_{|\rho_i\rho_{i,1}}) + \Big(\int \rho_{i,2}\,\d\mu_{|\rho_i}\Big)\cdot \DTC(\mu_{|\rho_i\rho_{i,2}}) \\ \leq \DTC(\mu_{|\rho_i}) - \frac{1}{2}\rmI_{\mu_{|\rho_i}}(\rho_{i,1},\rho_{i,2}).
\end{multline*}
We multiply this inequality by $\langle \rho_i\rangle$, observe that
\[\langle \rho_i\rangle\cdot \Big(\int \rho_{i,a}\,\d\mu_{|\rho_i}\Big) = \langle \rho_i\rho_{i,a}\rangle \quad \hbox{for}\ a \in \{1,2\},\]
and add the results over $i \in S$.  This gives
\begin{align*}
\sum_{j=1}^\ell\langle \rho'_j\rangle \cdot \DTC(\mu_{|\rho'_j}) &= \sum_{i\in [k]\setminus S}\langle \rho_i\rangle \cdot \DTC(\mu_{|\rho_i}) \\
&\quad + \sum_{i\in S}\Big(\langle \rho_i\rho_{i,1}\rangle\cdot \DTC(\mu_{|\rho_i\rho_{i,1}}) + \langle \rho_i\rho_{i,2}\rangle\cdot \DTC(\mu_{|\rho_i\rho_{i,2}})\Big)\\
&\leq \sum_{i\in [k]\setminus S}\langle \rho_i\rangle \cdot \DTC(\mu_{|\rho_i})\\
&\quad + \sum_{i\in S}\langle \rho_i\rangle \cdot \DTC(\mu_{|\rho_i}) - \frac{1}{2}\sum_{i\in S}\langle \rho_i\rangle \cdot \rmI_{\mu_{|\rho_i}}(\rho_{i,1},\rho_{i,2}).
\end{align*}
Substituting from~\eqref{eq:fuzzy-chain-again}, the right-hand side here is equal to
\[\sum_{i=1}^k\langle \rho_i\rangle \cdot \DTC(\mu_{|\rho_i}) - \frac{1}{2}\big[\rmI_\mu(\rho_1',\dots,\rho_\ell') - \rmI_\mu(\rho_1,\dots,\rho_k)\big].\]
This proves property (b).
\end{proof}

\begin{proof}[Proof of Theorem~\ref{thm:bigdecomp1} from Corollary~\ref{cor:dec}]
We produce a finite sequence of fuzzy partitions $(\rho_{t,j})_{j=1}^{k_t}$ for $t=0,1,2,\dots,t_0$ by the following recursion.

To start the recursion, let $k_0 = 1$ and let $\rho_{0,1}$ be the constant function $1$. If $\mu$ already satisfies T$(rn/200,r)$, then we let $t_0 = 0$ and stop the recursion here.

Now suppose that $\mu$ does not satisfy T$(rn/200,r)$ and that we have already constructed the fuzzy partition $(\rho_{t,j})_{j=1}^{k_t}$ for some $t \geq 0$.  Let $S$ be the set of all $j \in [k_t]$ for which $\mu_{|\rho_{t,j}}$ does not satisfy T$(r n/200,r)$. If
\begin{equation}\label{eq:I-change}
\sum_{j\in S}\langle \rho_{t,j}\rangle < \eps,
\end{equation}
then set $t_0 := t$ and stop the recursion.  Otherwise, let $(\rho_{t+1,j})_{j=1}^{k_{t+1}}$ be a new fuzzy partition produced from $(\rho_{t,j})_{j=1}^{k_t}$ by Corollary~\ref{cor:dec}.

If this recursion does not stop at stage $t$, then the negation of~\eqref{eq:I-change} must hold at that stage.  Combining this fact with both conclusions of Corollary~\ref{cor:dec}, it follows that
\[\sum_{j=1}^{k_{t+1}}\langle \rho_{t+1,j}\rangle \cdot \DTC(\mu_{|\rho_{t+1,j}}) \leq \sum_{j=1}^{k_t}\langle \rho_{t,j}\rangle \cdot \DTC(\mu_{|\rho_{t,j}}) - \frac{1}{2}\eps r^2 n^{-1} \rme^{-n}.\]
Since the average $\DTC$ on the left cannot be negative, and the average $\DTC$ decreases here by at least the fixed amount $\eps r^2 n^{-1} \rme^{-n}/2$, the recursion must stop at some finite stage.

For each $t = 0,1,\dots,t_0-1$, conclusion (b) of Corollary~\ref{cor:dec} gives that
\begin{multline*}
\rmI_\mu(\rho_{t+1,1},\dots,\rho_{t+1,k_{t+1}}) - \rmI_\mu(\rho_{t,1},\dots,\rho_{t,k_t}) \\\leq 2\Big[\sum_{j=1}^{k_t}\langle \rho_{t,j}\rangle \cdot \DTC(\mu_{|\rho_{t,j}}) - \sum_{j=1}^{k_{t+1}}\langle \rho_{t+1,j}\rangle \cdot \DTC(\mu_{|\rho_{t+1,j}})\Big].
\end{multline*}
Summing these inequalities over those values of $t$, we obtain
\begin{equation}\label{eq:bd-I-by-DTC}
\rmI_\mu(\rho_{t_0,1},\dots,\rho_{t_0,k_{t_0}}) \leq 2\cdot \DTC(\mu) - 2\sum_{j=1}^{k_{t_0}}\langle \rho_{t_0,j}\rangle \cdot \DTC(\mu_{|\rho_{t_0,j}}) \leq 2\cdot \DTC(\mu).
\end{equation}

To finish, define $(\rho_1,\dots,\rho_k)$ by letting $\rho_1$ be the sum of all those functions $\rho_{t_0,j}$ with $1 \leq j \leq k_{t_0}$ for which $\mu_{|\rho_{t_0,j}}$ does not satisfy T$(r n/200,r)$, and letting $\rho_2$, \dots, $\rho_k$ be an enumeration of the remaining entries in $(\rho_{t_0,1},\dots,\rho_{t_0,k_{t_0}})$.  This new fuzzy partition $(\rho_1,\dots,\rho_k)$ satisfies all three of the required conclusions.  Conclusion (a) follows from~\eqref{eq:bd-I-by-DTC} and an application of inequality~\eqref{eq:refine-more-inf}.  Conclusions (b) and (c) are written into the conditions for stopping the recursion.
\end{proof}

\begin{rmk}
	The above proof gives a bound on the number of entries $k$ in the resulting fuzzy partition, but the bound is too poor to be useful beyond guaranteeing that the recursion terminates.  This is because the quantity $r^2 n^{-1}\rme^{-n}$ appearing in Proposition~\ref{prop:dec} is so small.
\end{rmk}

\begin{rmk}
The recursion above terminates at the first stage $t_0$ when most terms in the mixture exhibit sufficient concentration.  Then conclusion (a) is proved using the fact that the average $\DTC$ must remain non-negative.  However, there is no guarantee that most of the resulting measures $\mu_{|\rho_j}$ have $\DTC$ close to zero. It may be that we have obtained all the measure concentration we want, but many of the values $\DTC(\mu_{|\rho_j})$ are still large.  Indeed, if $\mu$ is one of the measures constructed using an $\eps$-uniform set in Example~\ref{ex:prod-img-meas}, then $\mu$ already satisfies a better T-inequality than those obtained in Theorem~\ref{thm:bigdecomp1}, and it cannot be decomposed into nontrivial summands using Proposition~\ref{prop:dec}.  However, as discussed in that example, its $\TC$ and $\DTC$ are both of order $n$, and it cannot be decomposed efficiently into approximate product measures.
\end{rmk}

\begin{rmk}
	`Increment' and `decrement' arguments have other famous applications, particularly in extremal combinatorics.  Perhaps the best known is the proof of Szemer\'edi's regularity lemma~\cite{Sze75,Sze76}, which uses an argument that is now often called the `energy increment'.  See~\cite[Section IV.5]{Bol98} for a modern textbook treatment of Szemer\'edi's lemma, and see~\cite{Tao06--hyperreg} and~\cite[Chapters 10 and 11]{TaoVu06} for a broader discussion of increment arguments in additive combinatorics.
	
	However, having made this connection, we should also stress the following difference.  In Szemer\'edi's regularity lemma, the vertex set of a large finite graph is partitioned into a controlled number of cells so that most \emph{pairs} among those cells have a property called `quasirandomness'.  This pairwise requirement on the cells leads to the extremely large tower-type bounds that necessarily appear in Szemer\'edi's lemma: see~\cite{Gow97}. By contrast, Theorem~\ref{thm:bigdecomp1} produces summands that mostly have a desired property --- the T-inequality --- individually, but are not required to interact with each other in any particular way.  This makes for less dramatic bounds: in particular, for the simple relationship in part (a) of that theorem.
	
	Another precedent for our `decrement' argument is the work of Linial, Samorodnitsky and Wigderson~\cite{LinSamWig00} giving a strongly polynomial time algorithm for permanent estimation.  Their algorithm relies on obtaining a substantial decrement in the permanent of a nonnegative matrix under a procedure called matrix scaling.
	\end{rmk}

\subsection{Proof of the DTC decrement}\label{subs:DTC-dec}

The rest of this section is spent proving Proposition~\ref{prop:dec}.

Given $\mu$ and a fuzzy partition $(\rho_1,\rho_2)$, let $(\zeta,\xi)$ be a randomization of the resulting mixture
\[\mu = \langle \rho_1\rangle\cdot \mu_{|\rho_1} + \langle \rho_2\rangle\cdot \mu_{|\rho_2}.\]
Thus, $(\zeta,\xi)$ takes values in $\{1,2\} \times X$.  For this pair of random variables, the left-hand side of~\eqref{eq:new-I-top} is
\[\DTC(\mu) - \langle\rho_1\rangle \cdot \DTC(\mu_{|\rho_1}) - \langle \rho_2\rangle \cdot \DTC(\mu_{|\rho_2}).\]
The proof of Proposition~\ref{prop:dec} rests on a careful analysis of the right-hand side of~\eqref{eq:new-I-top} for this pair $(\zeta,\xi)$, which reads
\begin{equation}\label{eq:new-I-top2}
\rmI(\xi\,;\,\zeta) - \sum_{i=1}^n\rmI(\xi_i\,;\,\zeta\,|\,\xi_{[n]\setminus i}).
\end{equation}
We next re-write this expression in terms of $\mu$, $\rho_1$ and $\rho_2$.

By Corollary~\ref{cor:I-and-KL}, the first term in~\eqref{eq:new-I-top2} is equal to
\begin{equation}\label{eq:first-term-form}
\langle\rho_1\rangle \cdot \rmD(\mu_{|\rho_1}\,\|\,\mu) + \langle\rho_2\rangle \cdot \rmD(\mu_{|\rho_2}\,\|\,\mu).
\end{equation}

The remaining terms in~\eqref{eq:new-I-top2} can be expressed similarly.  To this end, for each $i$, let $(\theta_{i,\bf{z}}:\ \bf{z} \in A^{[n]\setminus i})$ be a conditional distribution for $\xi$ given $\xi_{[n]\setminus i}$ according to $\mu$.  Thus, the only remaining randomness under $\theta_{i,\bf{z}}$ is in the coordinate $\xi_i$.  This conditional distribution represents $\mu$ as the following mixture:
\begin{equation}\label{eq:hookup}
\mu = \int \theta_{i,\bf{z}}\ \mu_{[n]\setminus i}(\d\bf{z}).
\end{equation}
For each $i$ and $\bf{z} \in A^{[n]\setminus i}$, let $\langle \cdot\rangle_{i,\bf{z}}$ denote integration with respect to $\theta_{i,\bf{z}}$. If we condition on the event $\{\xi_{[n]\setminus i} = \bf{z}\}$, then the pair $(\zeta,\xi)$ becomes a randomization of the mixture
\[\theta_{i,\bf{z}} = \langle\rho_1\rangle_{i,\bf{z}}\cdot (\theta_{i,\bf{z}})_{|\rho_1} + \langle \rho_2\rangle_{i,\bf{z}} \cdot (\theta_{i,\bf{z}})_{|\rho_2}.\]
This is because the conditional probability of the event $\{(\zeta,\xi) = (j,\bf{x})\}$ given the event $\{\xi_{[n]\setminus i} = \bf{z}\}$ equals
\[\frac{\rho_j(\bf{x})\cdot \mu(\bf{x})}{\mu_{[n]\setminus i}(\bf{z})} = \rho_j(\bf{x})\cdot \theta_{i,\bf{z}}(\bf{x})\]
whenever $j \in \{1,2\}$, $\bf{x} \in A^n$ and $\bf{z} = \bf{x}_{[n]\setminus i}$.  Now another appeal to Corollary~\ref{cor:I-and-KL} gives
\begin{align}\label{eq:second-term-form}
\rmI(\xi_i\,;\,\zeta\,|\,\xi_{[n]\setminus i} = \bf{z}) &= \rmI(\xi\,;\,\zeta\,|\,\xi_{[n]\setminus i} = \bf{z}) \nonumber \\
&= \langle \rho_1\rangle_{i,\bf{z}}\cdot  \rmD\big((\theta_{i,\bf{z}})_{|\rho_1}\,\big\|\,\theta_{i,\bf{z}}\big) + \langle \rho_2\rangle_{i,\bf{z}}\cdot  \rmD\big((\theta_{i,\bf{z}})_{|\rho_2}\,\big\|\,\theta_{i,\bf{z}}\big)
\end{align}
for $\mu_{[n]\setminus i}$-almost every $\bf{z}$. Averaging over $\bf{z}$, this becomes
\begin{align}\label{eq:2nd-terms}
\rmI(\xi_i\,;\,\zeta\,|\,\xi_{[n]\setminus i}) &= \int \langle \rho_1\rangle_{i,\bf{z}}\cdot  \rmD\big((\theta_{i,\bf{z}})_{|\rho_1}\,\big\|\,\theta_{i,\bf{z}}\big)\,\mu_{[n]\setminus i}(\d \bf{z}) \nonumber \\ &\quad + \int \langle \rho_2\rangle_{i,\bf{z}}\cdot  \rmD\big((\theta_{i,\bf{z}})_{|\rho_2}\,\big\|\,\theta_{i,\bf{z}}\big)\,\mu_{[n]\setminus i}(\d \bf{z}).
\end{align}

For the proof of Proposition~\ref{prop:dec}, we use these calculations in the following special case: suppose that $f:A^n\to [0,1]$ is $1$-Lipschitz and that $0 \leq t \leq n/200$, and let
\begin{equation}\label{eq:rhotf}
\rho_1 := \frac{1}{2}\rme^{-tf} \quad \hbox{and}\quad \rho_2 := 1 - \frac{1}{2}\rme^{-tf}.
\end{equation}
Note that
\[0 < \rho_1 \leq \frac{1}{2} \quad \hbox{and} \quad \frac{1}{2} \leq \rho_2 < 1.\]
These bounds simplify the proof below, and are the reason for the factor of $\frac{1}{2}$ in the definition of $\rho_1$.

For this choice of fuzzy partition, we need an upper bound for the right-hand side in~\eqref{eq:2nd-terms}.  It has two terms, which we estimate separately.  The key to both estimates is the following geometric feature of the present setting: the measure $\theta_{i,\bf{z}}$ is supported on the set
\[S_{i,\bf{z}} := \{\bf{x} \in A^n:\ \bf{x}_{[n]\setminus i} = \bf{z}\},\]
which has diameter $1/n$ in the normalized Hamming metric.

\begin{lem}\label{lem:1st-term}
For the choice of $\rho_1$ and $\rho_2$ in~\eqref{eq:rhotf}, we have
\[\int \langle \rho_1\rangle_{i,\bf{z}}\cdot  \rmD\big((\theta_{i,\bf{z}})_{|\rho_1}\,\big\|\,\theta_{i,\bf{z}}\big)\,\mu_{[n]\setminus i}(\d \bf{z}) \leq \frac{t^2}{n^2}\langle \rho_1\rangle.\]
\end{lem}

\begin{proof}
	The function $tf$ is $t$-Lipschitz, so for any $i$ and $\bf{z}$ we have
\[\max\big\{\rho_1(\bf{y}):\ \bf{y} \in S_{i,\bf{z}}\big\} \leq \rme^{t/n}\min\big\{\rho_1(\bf{y}):\ \bf{y} \in S_{i,\bf{z}}\big\}.\]
Therefore Lemma~\ref{lem:D-est} gives
\[\rmD\big((\theta_{i,\bf{z}})_{|\rho_1}\,\big\|\,\theta_{i,\bf{z}}\big) \leq t^2/n^2.\]
Substituting into the desired integral, we obtain
\[\int \langle \rho_1\rangle_{i,\bf{z}}\cdot  \rmD\big((\theta_{i,\bf{z}})_{|\rho_1}\,\big\|\,\theta_{i,\bf{z}}\big)\,\mu_{[n]\setminus i}(\d \bf{z}) \leq \frac{t^2}{n^2}\int \langle \rho_1\rangle_{i,\bf{z}}\ \mu_{[n]\setminus i}(\d \bf{z}),\]
and this right-hand integral is equal to $\langle \rho_1\rangle$ by~\eqref{eq:hookup}.
\end{proof}

\begin{lem}\label{lem:2nd-term}
For the choice of $\rho_1$ and $\rho_2$ in~\eqref{eq:rhotf}, we have
\[\int \langle \rho_2\rangle_{i,\bf{z}}\cdot  \rmD\big((\theta_{i,\bf{z}})_{|\rho_2}\,\big\|\,\theta_{i,\bf{z}}\big)\,\mu_{[n]\setminus i}(\d \bf{z}) \leq 32\frac{t^2}{n^2}\langle \rho_1\rangle.\]
\end{lem}

Beware that the average on the right here is $\langle\rho_1\rangle$, not $\langle\rho_2\rangle$. Also, the factor of $32$ is chosen to be simple, not optimal.

\begin{proof}
We certainly have $\langle \rho_2\rangle_{i,\bf{z}} \leq 1$ for all $i$ and $\bf{z}$, so it suffices to show that
\[\int \rmD\big((\theta_{i,\bf{z}})_{|\rho_2}\,\big\|\,\theta_{i,\bf{z}}\big)\,\mu_{[n]\setminus i}(\d \bf{z}) \leq 32\frac{t^2}{n^2}\langle \rho_1\rangle.\]
This is the work of the rest of the proof.  It turns out to be slightly easier to work with an integral over all of $A^n$, so let us re-write the last integral as
\[\int \rmD\big((\theta_{i,\bf{x}_{[n]\setminus i}})_{|\rho_2}\,\big\|\,\theta_{i,\bf{x}_{[n]\setminus i}}\big)\,\mu(\d \bf{x}).\]

Now suppose that $\bf{x} \in A^n$ and that $\bf{y},\bf{y}'\in S_{i,\bf{x}_{[n]\setminus i}}$.  Then $\bf{x}$, $\bf{y}$ and $\bf{y}'$ agree in all but possibly the $i^{\rm{th}}$ coordinate.  Since $f$ is $1$-Lipschitz, it follows that
\[\rho_2(\bf{y}') - \rho_2(\bf{y}) = \frac{1}{2}(\rme^{-tf(\bf{y})} - \rme^{-tf(\bf{y}')}) \leq \frac{1}{2}\rme^{-tf(\bf{x})}(\rme^{t/n} - \rme^{-t/n}).\]
Since $0 < t/n \leq 1/200$, we have $\rme^{t/n} - \rme^{-t/n} < 4t/n$.  Also, we have arranged that $\rho_2$ is always at least $1/2$.  Therefore the last estimate implies that
\[\rho_2(\bf{y}') \leq \rho_2(\bf{y}) + \frac{2t}{n}\rme^{-tf(\bf{x})} \leq \rho_2(\bf{y})\Big(1 + \frac{4t}{n}\rme^{-tf(\bf{x})}\Big) \leq \rho_2(\bf{y})\exp\Big(\frac{4t}{n}\rme^{-tf(\bf{x})}\Big)\]
for all $\bf{y},\bf{y}' \in S_{i,\bf{x}_{[n]\setminus i}}$.  Since $\theta_{i,\bf{x}_{[n]\setminus i}}$ is supported on $S_{i,\bf{x}_{[n]\setminus i}}$, we may combine this estimate with Lemma~\ref{lem:D-est} to obtain
\[\rmD\big((\theta_{i,\bf{x}_{[n]\setminus i}})_{|\rho_2}\,\big\|\,\theta_{i,\bf{x}_{[n]\setminus i}}\big) \leq \frac{16t^2}{n^2}\rme^{-2tf(\bf{x})} \leq \frac{16t^2}{n^2}\rme^{-tf(\bf{x})} = 32\frac{t^2}{n^2}\rho_1(\bf{x}).\]
For the second inequality here, notice that we have bounded $\rme^{-2tf(\bf{x})}$ by $\rme^{-tf(\bf{x})}$. This is extremely crude, but it suffices for the present argument.

Substituting into the desired integral, we obtain
\[\int \rmD\big((\theta_{i,\bf{x}_{[n]\setminus i}})_{|\rho_2}\,\big\|\,\theta_{i,\bf{x}_{[n]\setminus i}}\big)\,\mu(\d \bf{x}) \leq 32\frac{t^2}{n^2}\int \rho_1(\bf{x})\ \mu(\d\bf{x}).\]
\end{proof}

Combining the preceding lemmas with~\eqref{eq:new-I-top2},~\eqref{eq:first-term-form} and~\eqref{eq:2nd-terms}, we have shown the following.

\begin{cor}\label{cor:dec-lower-bd}
	For the choice of $(\rho_1,\rho_2)$ in~\eqref{eq:rhotf} we have
	\begin{align*}
&\DTC(\mu) - \langle\rho_1\rangle \cdot \DTC(\mu_{|\rho_1}) - \langle\rho_2\rangle \cdot \DTC(\mu_{|\rho_2})  \\ &\geq \rmI_\mu(\rho_1,\rho_2) - n\Big(\frac{t^2}{n^2}\langle \rho_1\rangle + 32\frac{t^2}{n^2}\langle \rho_1\rangle\Big)\\
&=  \langle\rho_1\rangle \cdot \rmD(\mu_{|\rho_1}\,\|\,\mu) + \langle\rho_2\rangle \cdot \rmD(\mu_{|\rho_2}\,\|\,\mu) - 33\frac{t^2}{n}\langle \rho_1\rangle.
\end{align*}
\qed
\end{cor}

\begin{proof}[Proof of Proposition~\ref{prop:dec}]
Since $d_n$ has diameter at most $1$, any measure on $(A^n,d_n)$ satisfies T$(\k,1)$ for all $\k > 0$.  We may therefore assume that $r < 1$. If $\mu$ does not satisfy T$(r n/200,r)$, then by Proposition~\ref{prop:LSL-T} it also does not satisfy L$([r/2,r n/200],100/n)$.  This means there are a $1$-Lipschitz function $f:A^n\to \bbR$ and a value $t \in [r/2,r n/200] \subseteq [0,n/200]$ such that
\[\rmD(\mu_{|\rme^{tf}}\,\|\,\mu) > 100\frac{t^2}{n}.\]
Replacing $f$ with $-f$ and then adding a constant if necessary, we may assume that $f$ takes values in $[0,1]$ and satisfies
\[\rmD(\mu_{|\rme^{-tf}}\,\|\,\mu) = \rmD(\mu_{|(\rme^{-tf}/2)}\,\|\,\mu) > 100\frac{t^2}{n}.\]

Now construct $\rho_1$ and $\rho_2$ from this function $f$ as in~\eqref{eq:rhotf}.  Combining the above lower bound on $\rmD(\mu_{|(\rme^{-tf}/2)}\,\|\,\mu) = \rmD(\mu_{|\rho_1}\,\|\,\mu)$ with Corollary~\ref{cor:dec-lower-bd}, we obtain
	\begin{align*}
&\DTC(\mu) - \langle\rho_1\rangle \cdot \DTC(\mu_{|\rho_1}) - \langle\rho_2\rangle \cdot \DTC(\mu_{|\rho_2})\\
&\geq \langle\rho_1\rangle \cdot \rmD(\mu_{|\rho_1}\,\|\,\mu) + \langle\rho_2\rangle \cdot \rmD(\mu_{|\rho_2}\,\|\,\mu) - 33\frac{t^2}{n}\langle \rho_1\rangle\\
&> \langle\rho_1\rangle \cdot \rmD(\mu_{|\rho_1}\,\|\,\mu) + \langle\rho_2\rangle \cdot \rmD(\mu_{|\rho_2}\,\|\,\mu) - \frac{1}{2}\langle\rho_1\rangle \cdot \rmD(\mu_{|\rho_1}\,\|\,\mu)\\
&\geq \frac{1}{2}\Big(\langle\rho_1\rangle \cdot \rmD(\mu_{|\rho_1}\,\|\,\mu) + \langle\rho_2\rangle \cdot \rmD(\mu_{|\rho_2}\,\|\,\mu)\Big)\\
&= \frac{1}{2}\rmI_\mu(\rho_1,\rho_2).
\end{align*}
The last line here follows from another appeal to Corollary~\ref{cor:I-and-KL}.

Finally, since $f\leq 1$ and $r/2 \leq t \leq r n/200 \leq n$, we also obtain
\[\rmI_\mu(\rho_1,\rho_2) \geq \langle\rho_1\rangle \cdot \rmD(\mu_{|\rho_1}\,\|\,\mu) \geq \int  \frac{1}{2}\rme^{-tf}\,\d\mu\cdot \Big(100\frac{r^2}{4n}\Big) \geq \frac{r^2 \rme^{-t}}{n} \geq r^2n^{-1}\rme^{-n}.\]
\end{proof}

\begin{rmk}
The proof of Proposition~\ref{prop:dec} actually exploits the failure of $\mu$ to satisfy the stronger L-inequality L$([r/2,rn/200],100/n)$, rather than the T-inequality in the statement of that proposition.  Our proof therefore gives the conclusion of Theorem~\ref{thm:bigdecomp1} with that L-inequality in place of the T-inequality for the measures $\mu_{|\rho_j}$, $2 \leq j \leq k$.  However, later steps in the proof of Theorem B use some properties that we know only for T-inequalities, such as those from Subsection~\ref{subs:BobGot}.  We do not refer to L-inequalities again after the present section.
\end{rmk}

\subsection{The use of DTC in this section}\label{subs:whyDTC}

Both $\TC$ and $\DTC$ are notions of multi-variate mutual information for a measure $\mu$ on $A^n$.  In searching for a `decrement' proof of Theorem B, it is natural that we try such a quantity.  This is because we regard product measures, which exhibit very strong concentration, as an extreme case, and they are precisely the measures for which any good notion of multi-variate mutual information should be zero.

It is more subtle to describe why $\DTC$, rather than $\TC$ or any other kind of multi-variate mutual information, is the right quantity for the decrement.  A valuable hint in this direction comes from inspecting the terms that define $\DTC$ and comparing them with older proofs of logarithmic Sobolev inequalities on Hamming cubes.  Let us discuss this with reference to the L-inequalities of Subsection~\ref{subs:logSob}, which are really a special case of logarithmic Sobolev inequalities.

Let $\mu$ be the uniform distribution on $\{0,1\}^n$, and let $\xi_i:\{0,1\}^n\to\{0,1\}$ be the $i^{\rm{th}}$ coordinate projection for $1 \leq i \leq n$.  Following~\cite[Section 4]{Led9597} (where the argument is credited to Bobkov), one obtains an L-inequality for the space $(\{0,1\}^n,d_n)$ and measure $\mu$ from the following pair of estimates:
\begin{enumerate}
	\item[(a)] For any other $\nu \in \Pr(\{0,1\}^n)$, we have
	\[\rmD(\nu\,\|\,\mu) \leq \sum_{i=1}^n \rmD(\nu\,\|\,\mu\,|\,\xi_{[n]\setminus i}),\]
	where $\rmD(\nu\,\|\,\mu\,|\,\xi_{[n]\setminus i})$ denotes a conditional KL divergence (see~\cite[Section 2.5]{CovTho06}; the summands on the right in~\cite[Proposition 4.1]{Led9597} have this form).
	\item[(b)] If $\nu = \mu_{|\rme^{-tf}}$ for some $1$-Lipschitz function $f$ on $\{0,1\}^n$, then
	\[\rmD(\nu\,\|\,\mu\,|\,\xi_{[n]\setminus i}) \leq \frac{t^2}{n^2} \quad \hbox{for each}\ i.\]
	\end{enumerate}
Combining (a) and (b), it follows that $\rmD(\mu_{|\rme^{-tf}}\,\|\,\mu) \leq t^2/n$ for any $1$-Lipschitz function $f$.

Now consider instead an arbitrary measure $\mu$ on $\{0,1\}^n$, and suppose that $\rho_1 = \rme^{-tf}$ and $\rho_2 = 1-\rho_1$ for a $1$-Lipschitz function $f$ and some $t$ (ignoring the technical factor of $1/2$ in~\eqref{eq:rhotf}).  Of the two estimates above, (b) still holds simply by Lemma~\ref{lem:D-est}, but (a) is often false unless $\mu$ is a product measure.

Observe that the quantity $\rmD(\mu_{|\rme^{-tf}}\,\|\,\mu)$ appears in the first term of~\eqref{eq:first-term-form}, the formula for $\rmI(\xi\,;\,\zeta)$ in the previous subsection.  Likewise, a simple re-write of the quantity $\rmD(\mu_{|\rme^{-tf}}\,\|\,\mu\,|\,\xi_{[n]\setminus i} = \bf{z})$ shows that it is equal to $\rmD((\theta_{i,\bf{z}})_{|\rho_1}\,\|\,\theta_{i,\bf{z}})$, again using the notation of the previous subsection.  This quantity appears in the first term of~\eqref{eq:second-term-form}, the formula for $\rmI(\xi_i\,;\,\zeta\,|\,\xi_{[n]\setminus i} = \bf{z})$.

So for a general measure $\mu$ the difference
\begin{equation}\label{eq:D-diffs}
\rmD(\mu_{|\rme^{-tf}}\,\|\,\mu) - \sum_{i=1}^n\rmD(\mu_{|\rme^{-f}}\,\|\,\mu\,|\,\xi_{[n]\setminus i})
\end{equation}
resembles part of the difference
\begin{equation}\label{eq:DTC-diffs}
\rmI(\xi\,;\,\zeta) - \sum_{i=1}^n\rmI(\xi_i\,;\,\zeta\,|\,\xi_{[n]\setminus i}) = \DTC(\mu) - \langle\rho_1\rangle\cdot \DTC(\mu_{|\rho_1}) - \langle \rho_2\rangle\cdot \DTC(\mu_{|\rho_2}).
\end{equation}
The difference in~\eqref{eq:D-diffs} does not quite appear in~\eqref{eq:DTC-diffs} because the coefficients of the various KL divergences here do not match.  But the resemblance is enough to suggest the following.  If a suitable L-inequality fails, so one can find a function $f$ and parameter $t$ for which $\rmD(\mu_{|\rme^{-tf}}\,\|\,\mu)$ is large, then the difference in~\eqref{eq:D-diffs} must also be large, and then one might hope to show that the DTC-decrement in~\eqref{eq:DTC-diffs} is also large.  The work of Subsection~\ref{subs:DTC-dec} consists of the tweaks and technicalities that are needed to turn this hope into rigorous estimates.

Ledoux uses the logarithmic Sobolev inequalities in~\cite[Section 4]{Led9597} to give a new proof of an exponential moment bound for Lipschitz functions on product spaces. This bound is in turn equivalent to Marton's transportation inequalities (Proposition~\ref{prop:Marton}) via the Bobkov--G\"{o}tze equivalence (see the discussion after Proposition~\ref{prop:BobGot}).  Marton's original proof of Proposition~\ref{prop:Marton} is quite different, and arguably more elementary.  She uses an induction on the dimension $n$, and no functional inequalities such as logarithmic Sobolev inequalities appear.  As remarked following Proposition~\ref{prop:BobGot}, yet another proof can be given using the Bobkov--G\"{o}tze equivalence and the method of bounded martingale differences.  This last proof does involve bounding an exponential moment, but it still seems more elementary than the logarithmic Sobolev approach.

\begin{ques}
	Is there a proof of Theorem~\ref{thm:bigdecomp1}, or more generally of Theorem B, which uses an induction on $n$ and either (i) some variants of Marton's ideas, or (ii) the Bobkov--G\"{o}tze equivalence and some more elementary way of controlling exponential moments of Lipschitz functions?
	\end{ques}

I expect that such an alternative proof would offer valuable additional insight into the phenomena behind Theorem B.

\section{Completed proof of Theorem B}

We still have to turn Theorem~\ref{thm:bigdecomp1} into Theorem B. This is the work of the present section, which has two stages.  The first, in Subsection~\ref{subs:another-aux-decomp}, continues to work solely with $\DTC$.  The second, in Subsection~\ref{subs:B}, takes us from $\DTC$ back to $\TC$.

\subsection{Another auxiliary decomposition}\label{subs:another-aux-decomp}

\begin{thm}\label{thm:bigdecomp2}
For any $\eps,r > 0$ there exists $c > 0$ such that the following holds.  Any $\mu \in \Pr(A^n)$ can be written as a mixture
\[\mu = p_1\mu_1 + \dots + p_m\mu_m\]
so that
\begin{enumerate}
	\item[(a)] $m \leq c\exp(c\cdot \DTC(\mu))$,
	\item[(b)] $p_1 < \eps$, and
	\item[(c)] the measure $\mu_j$ satisfies T$(r n/600,r)$ for every $j=2,3,\dots,m$.
	\end{enumerate}
\end{thm}

Letting $\rho_j := \d(p_j\mu_j)/\d \mu$, we always have 
\[\rmI_\mu(\rho_1,\dots,\rho_m) \leq \rmH(p_1,\dots,p_m) \leq \log m.\]
Therefore Theorem~\ref{thm:bigdecomp2} is a strengthening of Theorem~\ref{thm:bigdecomp1}.  We deduce Theorem~\ref{thm:bigdecomp2} from Theorem~\ref{thm:bigdecomp1} by using a simple sampling argument to `coarsen' the representation of $\mu$ given by Theorem~\ref{thm:bigdecomp1} and allowing a slight degradation in the error estimates.

\begin{prop}[Sampling from a low-information mixture]\label{prop:sampling}
Let $(X,\mu)$ be a standard probability space, and let $\mu$ be written as a mixture $\int \mu_\bullet \,\d P$ using some other probability space $(\O,P)$ and a kernel $\mu_\bullet$ from $\O$ to $X$.  Let $\eps \in (0,1/2)$, let $\O_1 \subseteq \O$ be measurable with $P(\O_1) > 1 - \eps/2$, and assume that
\[I := \int \rmD(\mu_\omega\,\|\,\mu)\,P(\d\omega) < \infty.\]
Finally, let $m := \lceil 16\eps^{-2}\rme^{16(I+1)/\eps}\rceil$.  Then there are elements $\omega_1$, \dots, $\omega_m \in \O_1$, not necessarily distinct, such that
\[\Big\|\frac{1}{m}\sum_{j=1}^m \mu_{\omega_j} - \mu\Big\| < 3\eps.\]
\end{prop}

This is proved using the probabilistic method, together with a simple truncation argument.  This combination is reminiscent of classical proofs of the weak law of large numbers for random variables without bounded second moments~\cite[Section X.2]{FellerVolI}.

\begin{proof}
\emph{Step 1: setup and truncation.}\quad Let $Q := P\ltimes \mu_\bullet$, the hookup introduced in Subsection~\ref{subs:basic}, and let $F$ be the Radon--Nikodym derivative $\d Q/ \d(P\times \mu)$.  Then we have
\[\mu_\omega = F(\omega,\cdot)\cdot \mu \quad \hbox{for}\ P\hbox{-a.e.}\ \omega\]
and
\[I = \int F\log F\,\d(P\times \mu).\]
The function $t\log t$ on $[0,\infty)$ has a global minimum at $t = \rme^{-1}$ and its value there is $-\rme^{-1}$.  Therefore
\begin{equation}\label{eq:log-log+}
\int  F\log^+ F\ \d(P\times \mu) \leq I + \rme^{-1} < I+1.
\end{equation}

Now define
\[F' := \min\{F,\rme^{8(I+1)/\eps}\} \quad \hbox{and} \quad \mu_\omega' := F'(\omega,\cdot)\cdot \mu \quad \hbox{for each}\ \omega.\]
Each $\mu_\omega'$ is a positive measure bounded by $\mu_\omega$, so Markov's inequality and~\eqref{eq:log-log+} give
\begin{align*}
\int \|\mu_\omega - \mu_\omega'\|\,P(\d\omega) &= \int (F-F')\ \d(P\times \mu)\\ 
&\leq \int_{\{F > \rme^{8(I+1)/\eps}\}} F\ \d(P\times \mu)\\
&\leq \frac{\eps}{8(I+1)}\int F\log^+ F\ \d(P\times \mu)\\
&< \eps/8.
\end{align*}

Define
\[\mu' := \int \mu_\omega'\,P(\d\omega) \quad \hbox{and} \quad f'(x) := \int F'(\omega,x)\,P(\d \omega).\]
Then $\mu'$ is a positive measure bounded by $\mu$ which satisfies $\|\mu - \mu'\| \leq \eps/8$, and $f'$ is a version of the Radon--Nikodym derivative $\d\mu'/\d\mu$.

\vspace{7pt}

\emph{Step 2: the probabilistic method.}\quad Let $Y_1$, \dots, $Y_m$ be i.i.d. random elements of $\O$, each with distribution $P$, and let $\sfP$ denote the underlying probability measure for those random elements. Let $\sfE$ and $\rm{Var}$ denote expectation and variance with respect to $\sfP$.

For each $x \in X$, consider the random variable
\[X_x := \frac{1}{m}\sum_{j=1}^m F'(Y_j,x).\]
This is an average of i.i.d. random variables. Each of them satisfies
\[\sfE F'(Y_j,x) = \int F'(\omega,x)\,P(\d\omega) = f'(x).\]
Also, each is bounded by $\rme^{8(I+1)/\eps}$, and hence satisfies $\rm{Var}(F'(Y_j,x)) \leq \rme^{16(I+1)/\eps}$.

Now consider the empirical average of the measures $\mu_{Y_j}'$: it is given by
\[\frac{1}{m}\sum_{j=1}^m \mu_{Y_j}' = \frac{1}{m}\sum_{j=1}^m (F'(Y_j,\,\cdot\,)\cdot \mu) = \Big(\frac{1}{m}\sum_{j=1}^m F'(Y_j,\,\cdot\,)\Big)\cdot \mu.\]
Therefore, by Fubini's theorem,
\begin{align*}
\sfE\Big\|\frac{1}{m}\sum_{j=1}^m \mu_{Y_j}' - \mu'\Big\| &= \sfE\int \Big|\frac{1}{m}\sum_{j=1}^m F'(Y_j,\,\cdot\,) - f'\Big|\,\d\mu\\
&= \int \sfE|X_x - \sfE X_x|\ \mu(\d x)\\
&\leq \int \sqrt{\rm{Var}(X_x)}\ \mu(\d x)\\
&\leq m^{-1/2}\rme^{8(I+1)/\eps}.
\end{align*}

Combining the bounds above, we obtain
\begin{align*}
\sfE\Big\|\frac{1}{m}\sum_{j=1}^m \mu_{Y_j} - \mu\Big\| &\leq \frac{1}{m}\sum_{j=1}^m \sfE\|\mu_{Y_j} - \mu_{Y_j}'\| + \sfE\Big\|\frac{1}{m}\sum_{j=1}^m \mu_{Y_j}' - \mu'\Big\| + \|\mu' - \mu\|\\
&= \int \|\mu_\omega - \mu_\omega'\|\,P(\d\omega) + \sfE\Big\|\frac{1}{m}\sum_{j=1}^m \mu_{Y_j}' - \mu'\Big\| + \|\mu' - \mu\|\\
&< \eps/4 + m^{-1/2}\rme^{8(I+1)/\eps} \leq \eps/2.
\end{align*}
Therefore Markov's inequality gives
\[\sfP\Big\{\Big\|\frac{1}{m}\sum_{j=1}^m \mu_{Y_j} - \mu\Big\| < \eps\Big\} > \frac{1}{2}.\]

On the other hand, we have
\[\sfE|\{j \in [m]:\ Y_j \in \O\setminus \O_1\}| = mP(\O\setminus \O_1) < m\eps /2,\]
so another appeal to Markov's inequality gives
\[\sfP\big\{|\{j \in [m]:\ Y_j \in \O_1\}| > (1-\eps) m\big\} > \frac{1}{2}.\]
Combining this probability lower bound with the previous one, it follows that the intersection of these events has positive probability.  Therefore some possible values $\o_1$, \dots, $\o_m$ of $Y_1$, \dots, $Y_m$ satisfy both
\[\Big\|\mu - \frac{1}{m}\sum_{j=1}^m\mu_{\o_j}\Big\| < \eps\]
and
\[|\{j \in [m]:\ \o_j \in \O_1\}| > (1-\eps) m.\]
To complete the proof, we simply discard any values $\o_j$ that lie in $\O\setminus \O_1$ and replace them with arbitrary members of $\O_1$.  This incurs an additional error of at most $2\eps$ in the total-variation approximation to $\mu$.
\end{proof}

\begin{proof}[Proof of Theorem~\ref{thm:bigdecomp2}]
By shrinking $\eps$ if necessary, we may assume that
\begin{equation}\label{eq:choose-eps}
\eps < \frac{1}{6} \quad \hbox{and also} \quad 400\log((1-3\eps/2)^{-1}) < r^2.
\end{equation}

Now apply Theorem~\ref{thm:bigdecomp1} with $\eps^2/2$ in place of $\eps$ and with the present value of $r$.  Let $(\rho_1,\dots,\rho_k)$ be the resulting fuzzy partition. 
Let $\O := [k]$, and let $P$ be the probability distribution on this set defined by $P(j) := p_j := \langle\rho_j\rangle$.  The formula from Corollary~\ref{cor:I-and-KL} and conclusion (a) of Theorem~\ref{thm:bigdecomp1} give
\[I := \rmI_\mu(\rho_1,\dots,\rho_k) = \sum_{j \in \O} p_j\cdot\rmD(\mu_{|\rho_j}\,\|\,\mu) \leq 2\cdot \DTC(\mu).\]
Also, let $\O_1 = \{2,3,\dots,k\}$, so conclusion (b) of Theorem~\ref{thm:bigdecomp1} gives $P(\O_1) > 1 - \eps^2/2$.

We now apply Proposition~\ref{prop:sampling} to $\mu$ and its representation as a mixture of the measures $\mu_{|\rho_j}$, which we abbreviate to $\mu_j$.  We apply that proposition with $\eps^2$ in place of $\eps$.  It provides a constant $c$ which depends on $\eps$ (and hence also on $r$, because of~\eqref{eq:choose-eps}), an integer $m \leq c\rme^{cI}$, and elements $i_1,\dots,i_m \in \O_1$ such that
\begin{equation}\label{eq:smalldiff}
\|\mu - \mu'\| < 3\eps^2 \quad \hbox{where} \quad \mu' := \frac{1}{m}\sum_{j=1}^m \mu_{i_j}.
\end{equation}

Using the Jordan decomposition of $\mu - \mu'$, we may now write
\[\mu = \g + \nu \quad \hbox{and} \quad \mu' = \g + \nu'\]
for some measures $\g$, $\nu$ and $\nu'$ such that $\nu$ and $\nu'$ are mutually singular.  These measures satisfy
\[\|\nu\| = \|\nu'\| = \frac{1}{2}\|\mu-\mu'\| < 3\eps^2/2 \quad \hbox{and hence} \quad \|\g\| > 1- 3\eps^2/2.\]
Let $f$ be the Radon--Nikodym derivative $\d\g/\d \mu'$.  Then $0\leq f \leq 1$, and we have
\[\frac{1}{m}\sum_{j=1}^m\int f\,\d\mu_{i_j} = \int f\,\d\mu' = \|\g\| > 1 - 3\eps^2/2.\]
Therefore, by Markov's inequality, the set $J$ of all $j \in [m]$ which satisfy
\begin{equation}\label{eq:f-int-large}
\int f\,\d\mu_{i_j} > 1 - 3\eps/2
\end{equation}
has cardinality at least $(1-\eps)m$.

If $j \in J$, then~\eqref{eq:f-int-large} implies that
\[\frac{\d(\mu_{i_j})_{|f}}{\d\mu_{i_j}} = \frac{f}{\int f\,\d\mu_{i_j}} < (1 - 3\eps/2)^{-1}.\]
For these $j$, Lemma~\ref{lem:RN-flat} gives that the measure $(\mu_{i_j})_{|f}$ still satisfies
\[\rm{T}\Big(r n/200,400\log((1-3\eps/2)^{-1})/r n + 2r\Big).\]
By the second upper bound in~\eqref{eq:choose-eps}, this implies T$(r n/200,3r)$.

So now we can write
\begin{align}\label{eq:mu-nu-g}
\mu &= \nu + \g \nonumber\\
&= \nu + f\cdot \mu' \nonumber \\
&= \nu + \frac{1}{m}\sum_{j\in [m]\setminus J} f\cdot \mu_{i_j} + \frac{1}{m}\sum_{j\in J} f\cdot \mu_{i_j}.
\end{align}
In this last sum, the first few terms satisfy
\[\Big\|\nu + \frac{1}{m}\sum_{j\in [m]\setminus J}f\cdot \mu_{i_j}\Big\| < 3\eps/2 + \frac{m - |J|}{m} < 4\eps.\]
On the other hand, the remainder
\[\frac{1}{m}\sum_{j\in J} f\cdot \mu_{i_j}\]
is a non-negative linear combination of probability measures that all satisfy T$(r n/200,3r)$.  So now let us combine the first few terms in~\eqref{eq:mu-nu-g} into a single `bad' term.  This gives a mixture of at most $m$ measures which has the desired properties, except that the bad term has total variation bounded by $4\eps$ rather than $\eps$, and the parameter $r$ has been replaced by $3r$ throughout.  Since $\eps > 0$ and $r > 0$ are both arbitrary, this completes the proof.
\end{proof}

Our route to Theorem~\ref{thm:bigdecomp2} is quite indirect.  This is because the proof of Theorem~\ref{thm:bigdecomp1} relates the decrement in the $\DTC$ to a change in the mutual information of the fuzzy partition $(\rho_j)_j$, not a change in the entropy of the probability vector $(\langle \rho_j\rangle)_j$.  As a result, Theorem~\ref{thm:bigdecomp1} might give a representation of $\mu$ as a mixture with far too many summands, and we must then go back and find a more efficient representation by sampling as in Proposition~\ref{prop:sampling}.

\begin{ques}
Can one give a more direct proof of Theorem~\ref{thm:bigdecomp2} by improving some of the estimates in the proof of Theorem~\ref{thm:bigdecomp1}?
\end{ques}

\subsection{Completed proof of Theorem B}\label{subs:B}

Let $\mu$ be as in the statement of Theorem B.  To deduce Theorem B from Theorem~\ref{thm:bigdecomp2}, the last step is to combine that theorem with Lemmas~\ref{lem:trimming} and~\ref{lem:stab-lift}.

\begin{proof}[Proof of Theorem B]
As remarked previously, $(A^n,d_n,\mu)$ must satisfy T$(\k,1)$ for all $\k > 0$, so we may assume that $r < 1$.
	
In the statement of Theorem B we have $E = \TC(\mu)$. Let $S \subseteq [n]$ be the subset provided by Lemma~\ref{lem:trimming}, so $|S| \geq (1-r)n$ and $\DTC(\mu_S) \leq E/r$.  Let $c_1$ be the constant given by Theorem~\ref{thm:bigdecomp2} for the current values of $\eps$ and $r$. Applying Theorem~\ref{thm:bigdecomp2} to $\mu_S$, we can write it as a mixture
\[\mu_S = \rho_1\cdot \mu_S + \cdots + \rho_m\cdot \mu_S\]
for some fuzzy partition $(\rho_j)_{j=1}^m$ on $A^S$ such that
\begin{enumerate}
\item[(a)] $m\leq c_1\exp(c_1\cdot \DTC(\mu_S)) \leq c\rme^{cE}$ where $c := \max\{c_1/r,1\}$,
\item[(b)] $\int \rho_1\,\d\mu_S < \eps$, and
\item[(c)] the measure $(\mu_S)_{|\rho_j}$ is defined and satisfies T$(r |S|/600,r)$ for every $j=2,3,\dots,m$.
\end{enumerate}
Let $\rho'_j(\bf{x}) = \rho_j(\bf{x}_S)$ for every $\bf{x} \in A^n$.  Then the resulting mixture
\[\mu = \rho'_1\cdot \mu + \cdots + \rho_m'\cdot \mu\]
still satisfies the analogues of properties (a) and (b) above.  Finally, property (c) above combines with Lemma~\ref{lem:stab-lift} to give that $\mu_{|\rho_j'}$ satisfies
\[\rm{T}\Big(\frac{n}{|S|}\cdot \frac{r|S|}{600},\frac{|S|}{n}r + r\Big)\]
for every $j=2,3,\dots,m$.  This inequality implies T$(rn/600,2r)$, and they coincide in case $S = [n]$. Since $r > 0$ is arbitrary, this completes the proof.
\end{proof}

\subsection{Aside: another possible connection}\label{subs:EFKY}

Theorem~\ref{thm:bigdecomp2} is worth comparing with recent results of Ellis, Friedgut, Kindler and Yehudayoff in~\cite{EllFriKinYeh16}.  In our terminology, they prove that if $\DTC(\mu)$ is very small, then $\mu$ must itself be close to a product measure in a rather strong sense (certainly strong enough to imply a good concentration inequality).  However, they prove this only when $\DTC(\mu)$ is bounded by a fixed and sufficiently small tolerance $\eps$, independently of the dimension $n$.  By contrast, in Part III we need to apply Theorem B when $\DTC(\mu)$ is of order $n$.  This is far outside the domain covered by the results of~\cite{EllFriKinYeh16}.

It would be interesting to understand whether our present work and the proofs in~\cite{EllFriKinYeh16} have some underlying structure in common.

\section{Proof of Theorem C}\label{sec:C}

To prove Theorem C, we carve out the desired partition of $A^n$ one set at a time.

\begin{prop}\label{prop:near-C}
	For any $r > 0$ there exist $c,\k > 0$ such that, for any alphabet $A$, the following holds for all sufficiently large $n$.  If $\mu \in \Pr(A^n)$, then there is a subset $V \subseteq A^n$ such that
	\[\mu(V) \geq \exp(-c\cdot \TC(\mu))\]
	and such that $\mu_{|V}$ satisfies T$(\k n,r)$.
\end{prop}

\begin{proof}[Proof of Theorem C from Proposition~\ref{prop:near-C}]
	Consider $\eps,r > 0$ as in the statement of Theorem C.  Clearly we may assume that $\eps < 1$. For this value of $r$ and for the alphabet $A$, let $n$ be large enough to apply Proposition~\ref{prop:near-C}.  Let $c_1$ and $\k$ be the new constants given by that proposition.

We now construct a finite disjoint sequence of subsets $V_1$, $V_2$, \dots, $V_{m-1}$ of $A^n$ by the following recursion.

To start, let $V_1$ be a subset such that 
\[\mu(V_1) \geq  \exp(-c_1\cdot \TC(\mu))\]
and such that $\mu_{|V_1}$ satisfies T$(\k n,r)$, as provided by Proposition~\ref{prop:near-C}.

Now suppose we have already constructed $V_1$, \dots, $V_\ell$ for some $\ell\geq 1$, and let $W:= A^n\setminus (V_1\cup\cdots \cup V_\ell)$. If $\mu(W) < \eps$, then stop the recursion and set $m := \ell+1$.  Otherwise, let $\mu':= \mu_{|W}$, and apply Proposition~\ref{prop:near-C} again to this new measure.  By Corollary~\ref{cor:cond-and-TC} we have
\[\TC(\mu') \leq \frac{1}{\mu(W)}(\TC(\mu) +\log 2) \leq \eps^{-1} (\TC(\mu) + \log 2),\]
so Proposition~\ref{prop:near-C} gives a new subset $V_{\ell+1}$ of $A^n$ such that
\[\mu(V_{\ell+1}) \geq \mu(W)\cdot \mu'(V_{\ell+1}) \geq \eps\exp(-c_1\cdot \TC(\mu')) \geq \eps 2^{-c_1/\eps}\exp(-c_1\cdot \TC(\mu)/\eps)\]
and such that $\mu'_{|V_{\ell+1}}$ satisfies T$(\k n,r)$.  Since $\mu'$ is supported on $W$, we may intersect $V_{\ell+1}$ with $W$ without disrupting either of these conclusions, and so assume that $V_{\ell+1}\subseteq W$.  Having done so, we have $\mu'_{|V_{\ell+1}} = \mu_{|V_{\ell+1}}$.  This continues the recursion.

The sets $V_j$ are pairwise disjoint, and they all have measure at least
\[\eps 2^{-c_1/\eps} \exp(-c_1\cdot \TC(\mu)/\eps).\]
Therefore the recursion must terminate at some finite value $\ell$ which satisfies
\[m = \ell+1 \leq \eps^{-1}2^{c_1/\eps}\exp(c_1\cdot\TC(\mu)/\eps) + 1 \leq (\eps^{-1}2^{c_1/\eps}+1)\exp(c_1\cdot\TC(\mu)/\eps).\]
Once it has terminated, let $U_j:= V_{j-1}$ for $j=2,3,\dots,m$, and let $U_1$ be the complement of all these sets.

If we set $c:= \max\{(2^{-c_1/\eps}/\eps +1 ),c_1/\eps\}$, then this partition $U_1$, \dots, $U_m$ has all three of the desired properties.
	\end{proof}

It remains to prove Proposition~\ref{prop:near-C}.  Fix $r > 0$ for the rest of the section, and consider $\mu \in \Pr(A^n)$.  Let $\k := r/1200$, and let $c_\rm{B}$ be the constant provided by Theorem B with the input parameters $r$ and $\eps := 1/2$.  Cearly we may assume that $c_\rm{B} \geq 1$ without disrupting the conclusions of Theorem B.  We prove Proposition~\ref{prop:near-C} with this value for $\k$ but with $9r$ in place of $r$ and with a new constant $c$ derived from $r$ and $c_\rm{B}$. Since $r > 0$ was arbitrary this still completes the proof.  Since any measure on $A^n$ satisfies T$(\k n,9r)$ if $r \geq 1/9$, we may also assume that $r < 1/9$.

To lighten notation, let us define
\begin{equation}\label{eq:choice-of-eta}
	\delta:= \min\{r^2/42,1/18\}.
\end{equation}
We make frequent references to this small auxiliary quantity below.   Its definition results in a correct dependence on $r$ for certain estimates near the end of the proof.

At a few points in this section it is necessary that $n$ be sufficiently large in terms of $r$ and $|A|$.  These points are explained as they arise.

There are three different cases in the proof of Proposition~\ref{prop:near-C}, in which the desired set $V$ exists for different reasons.  Two of those cases are very simple. All of the real work goes into the third case, including an appeal to Theorem B.  The analysis of that third case could be applied to any measure $\mu$ on $A^n$, but I do not see how to control the necessary estimates unless we assume the negations of the first two cases --- this is why we separate the proof into cases at all.

The first and simplest case is when $\mu$ has a single atom of measure at least
\begin{equation}\label{eq:atom-bd}
\exp\Big(-\frac{161 c_\rm{B}}{\delta^2} \cdot \TC(\mu)\Big).
\end{equation}
  Then we just take $V$ to be that singleton.  So in the rest of our analysis we may assume that all atoms of $\mu$ weigh less than this.  Of course, the constant that appears in front of $\TC(\mu)$ in~\eqref{eq:atom-bd} is chosen to meet our needs when we come to apply this assumption later.

The next case is when $\TC(\mu)$ is sufficiently small.  This is dispatched by the following lemma.

\begin{lem}\label{lem:small-TC}
	If $\TC(\mu) \leq r^4n$, and if $n$ is sufficiently large in terms of $r$, then there is a subset $V$ of $A^n$ satisfying $\mu(V) \geq 1/2$ such that $\mu_{|V}$ satisfies T$(8rn,9r)$.
	\end{lem}

\begin{proof}
Let $\mu_{\{i\}}$ be the $i^{\rm{th}}$ marginal of $\mu$ for $i=1,2,\dots,n$.  Recall from~\eqref{eq:TC2} that
\[\TC(\mu) = \rmD(\mu\,\|\,\mu_{\{1\}}\times \cdots \times \mu_{\{n\}}).\]
Therefore Marton's inequality from the first part of Proposition~\ref{prop:Marton} gives
\[\ol{d_n}(\mu,\mu_{\{1\}}\times \cdots \times \mu_{\{n\}}) \leq \sqrt{\frac{\TC(\mu)}{2n}} \leq r^2,\]
where in the end we ignore a factor of $\sqrt{1/2}$. Since our assumptions imply that $r < 1/8$, we may now apply the second part of Proposition~\ref{prop:Marton} and then Proposition~\ref{prop:T-under-dbar}. These give a subset $V$ of $A^n$ satisfying $\mu(V) \geq 1 - 4r \geq 1/2$ and such that $\mu_{|V}$ satisfies
	\[\rm{T}\Big(8rn,\frac{8r + 2\log 4}{8rn} + 4r + 4r\Big).\]
	If $n$ is sufficiently large, this implies T$(8rn,9r)$.
\end{proof}

So now we may assume that $\mu$ has no atoms of measure at least~\eqref{eq:atom-bd}, and also that $\TC(\mu)$ is greater than $r^4 n$.  We keep these assumptions in place until the proof of Proposition~\ref{prop:near-C} is completed near the end of this section.

This leaves the case in which we must use Theorem B.  This argument is more complicated.  It begins by deriving some useful structure for the measure $\mu$ from the two assumptions above.  This is done in two steps which we formulate as Lemma~\ref{lem:remaining-case} and Proposition~\ref{prop:remaining-case}.  Lemma~\ref{lem:remaining-case} can be seen as a quantitative form of the asymptotic equipartition property for a big piece of the measure $\mu$.

\begin{lem}\label{lem:remaining-case}
Consider a measure $\mu$ satisfying the assumptions above, and let $M := \lceil n\log|A|\rceil$.  If $n$ is sufficiently large, then there are a subset $P\subseteq A^n$ and a constant $h \geq 160c_\rm{B}\TC(\mu)/\delta^2$ such that $\mu(P) \geq 1/4M$,
	\[\TC(\mu_{|P}) \leq 4\TC(\mu),\]
	and
	\[\rme^{-h} \geq \mu_{|P}(\bf{x}) > \rme^{-h - 1} \quad \forall \bf{x} \in P.\]
	\end{lem}

\begin{proof}
	Let $E := \TC(\mu)$ and let $h_0 := 161c_\rm{B}E/\delta^2$ (the exponent from~\eqref{eq:atom-bd}). Let
	\[P_j := \{\bf{x}:\ \rme^{-h_0-j} \geq \mu(\bf{x}) > \rme^{-h_0-j -1}\}\quad \hbox{for}\ j=0,1,\dots,M-1\]
	and let
	\[P_M := \{\bf{x}:\ \rme^{-h_0- M} \geq \mu(\bf{x})\}.\]
	Since, by assumption, all atoms of $\mu$ weigh less than $\rme^{-h_0}$, the sets $P_0,P_1,\dots,P_M$ constitute a partition of $A^n$.  Also, from its definition, the last of these sets must satisfy
	\begin{equation}\label{eq:M+1-small}
	\mu(P_M) \leq \rme^{-h_0 -M}|P_M| \leq \rme^{- h_0}|A|^{-n}|A|^n = \rme^{-161c_\rm{B}E/\delta^2}.
	\end{equation}
Since we are also assuming that $E > r^4 n$, this bound is less than $1/4$ provided $n$ is large enough in terms of $r$.
	
	Using this partition, we may write
	\[\mu = \sum_{j=0}^M \mu(P_j)\cdot \mu_{|P_j}.\]
	By Lemma~\ref{lem:TC-approx-concave}, this leads to
	\begin{equation}\label{eq:TC-approx-concave-appn}
		\rm{TC}(\mu) + \rmH\big(\mu(P_0),\dots,\mu(P_M)\big) \geq \sum_{j=0}^M \mu(P_j)\cdot \rm{TC}(\mu_{|P_j}).
	\end{equation}
	The left-hand side here is at most
	\[\rm{TC}(\mu) + \log (M+1) \leq \rm{TC}(\mu) + \log\log|A| + \log n + 2.\]
	Since we are assuming that $\rm{TC}(\mu) > r^4 n$, this upper bound is less than $2\rm{TC}(\mu)$ provided $n$ is sufficiently large in terms of $r$ and $|A|$.
	
	Therefore, provided $n$ is sufficiently large, the right-hand side of~\eqref{eq:TC-approx-concave-appn} is at most $2\TC(\mu)$.  By Markov's inequality, it follows that more than half of $\mu$ is supported on cells $P_j$ which satisfy
	\begin{equation}\label{eq:TC-still-bdd}
		\rm{TC}(\mu_{|P_j}) \leq 4\rm{TC}(\mu).
	\end{equation}
	By~\eqref{eq:M+1-small}, it follows that more than a quarter of $\mu$ is supported on such cells $P_j$ with $j \leq M-1$.
	
	Now choose $j \leq M-1$ for which~\eqref{eq:TC-still-bdd} holds and such that $\mu(P_j)$ is maximal, and let $P := P_j$.  This maximal value must be at least $1/4M$.  By the definition of $P$, each $\bf{x} \in P$ satisfies
	\[\rme^{a - h_0 - j} \geq \mu_{|P}(\bf{x}) > \rme^{a - h_0 - j -1}, \quad \hbox{where}\ \rme^a := 1/\mu(P).\]
	Finally, let $h := h_0 + j - a$, and observe that
	\[h \geq h_0 - a \geq 161c_\rm{B}E/\delta^2 - \log \big(4\lceil n\log |A|\rceil\big).\]
	Since we are assuming that $E > r^4n$, this is at least $160c_\rm{B}E/\delta^2$ provided $n$ is sufficiently large in terms of $r$ and $|A|$.
	\end{proof}

\begin{prop}\label{prop:remaining-case}
Consider a measure $\mu$ as in the preceding lemma. Let $P$ be the set and $h$ the constant obtained there.  Let $m:= \lceil (1-\delta)n\rceil$.  Then there are subsets
\[S\subseteq [n], \quad R \subseteq A^n \quad \hbox{and} \quad Q\subseteq A^S\]
such that
\begin{itemize}
	\item[(a)] $|S| = m$,
	\item[(b)] $\mu(R) \geq \delta/32M$,
\item[(c)] $\bf{x}_S\in Q$ for all $\bf{x} \in R$,
\item[(d)] every $\bf{z} \in Q$ has the property that
\[\max_{\bf{x} \in R:\ \bf{x}_S = \bf{z}}\mu_{|R}(\bf{x}) \leq \rme^{-\delta h/4}\mu_{|R}\{\bf{x}:\ \bf{x}_S = \bf{z}\},\]
and
\item[(e)] $\TC(\mu_{|R}) \leq 33\TC(\mu)/\delta$.
\end{itemize}
	\end{prop}

\begin{proof}
	Let $\mu':= \mu_{|P}$, and let $\xi_i:A^n\to A$ be the coordinate projection for $i=1,\dots,n$.
	
	By re-ordering the coordinates of $A^n$ if necessary, let us assume that
\begin{equation}\label{eq:order-coords}
\rmH_{\mu'}(\xi_1) \leq \cdots \leq \rmH_{\mu'}(\xi_n).
\end{equation}
Having done so, let $S := \{1,2,\dots,m\}$.  Let $\mu'_{[m]}$ be the projection of $\mu'$ to $A^m$.

From~\eqref{eq:order-coords} and the subadditivity of entropy, we obtain
\[\rmH_{\mu'}(\xi_{[m]}) \leq \sum_{i=1}^m \rmH_{\mu'}(\xi_i) \leq \frac{m}{n}\sum_{i=1}^n \rmH_{\mu'}(\xi_i) = \frac{m}{n}\big(\rmH(\mu') + \TC(\mu')\big).\]
Provided $n$ is sufficiently large in terms of $r$, this is at most
\[(1-3\delta/4)\rmH(\mu') + (1-3\delta/4)\TC(\mu') \leq (1-3\delta/4)\rmH(\mu') + \TC(\mu').\]
The first conclusion of Lemma~\ref{lem:remaining-case} gives that $\TC(\mu') \leq 4\TC(\mu) = 4E$. On the other hand, the second conclusion gives that
\begin{equation}\label{eq:bds-H-mu'}
h + 1 > \rmH(\mu') \geq h,
\end{equation}
and $h$ is greater than $16E/\delta$ because of our earlier assumption that $c_\rm{B} \geq 1$. Combining these inequalities, we deduce that
\begin{multline}\label{eq:proj-H-small}
\rmH_{\mu'}(\xi_{[m]}) \leq (1-3\delta/4)\rmH(\mu') + 4E
< (1-3\delta/4)\rmH(\mu') + \frac{\delta}{4}\rmH(\mu') \\ \leq (1-\delta/2)\rmH(\mu') < (1-\delta/2)(h+1).
\end{multline}

Let
\begin{align*}
Q &:= \big\{\bf{z}\in A^m:\ \mu'_{[m]}(\bf{z}) \geq \rme^{-(1-\delta/4)h}\big\}\\
&= \big\{\bf{z}\in A^m:\ -\log \mu'_{[m]}(\bf{z}) \leq (1-\delta/4)h\big\}.
\end{align*}
Since
\[\rmH_{\mu'}(\xi_{[m]}) = \sum_{\bf{z} \in A^m}\mu'_{[m]}(\bf{z})\cdot \big(-\log \mu'_{[m]}(\bf{z})\big),\]
we deduce from~\eqref{eq:proj-H-small} and Markov's inequality that
\[\mu_{[m]}'(Q) \geq 1 - \frac{(1-\delta/2)(h+1)}{(1-\delta/4)h} = \frac{\delta/4}{(1-\delta/4)} - \frac{(1-\delta/2)}{(1-\delta/4)h}.\]
Since our assumptions on $\mu$ give
\[h \geq 160c_\rm{B}E/\delta^2 \geq 160c_\rm{B}r^4n/\delta^2,\]
it follows that $\mu'_{[m]}(Q) \geq \delta/8$ provided $n$ is large enough in terms of $r$.

Let $R := \{\bf{x} \in P:\ \bf{x}_{[m]} \in Q\}$.  Since $R \subseteq P$, we have $\mu'_{|R} = \mu_{|R}$.

This completes the construction of $S$, $R$ and $Q$.  Among the desired properties, (a) and (c) are clear from the construction.  Property (b) holds because
\[\mu(R) = \mu_{|P}(R)\cdot \mu(P) = \mu'_{[m]}(Q)\cdot \mu(P) \geq \frac{\delta}{8}\cdot \frac{1}{4M}.\]

To prove property (d), observe that every $\bf{x}\in R\subseteq P$ satisfies
\[\mu_{|R}(\bf{x}) = \frac{\mu'(\bf{x})}{\mu'(R)} = \frac{\mu'(\bf{x})}{\mu'_{[m]}(Q)} \leq \frac{\rme^{-h}}{\mu'_{[m]}(Q)},\]
whereas every $\bf{z} \in Q$ satisfies
\[\mu_{|R}\{\bf{x}:\ \bf{x}_{[m]} = \bf{z}\} = \frac{\mu'\{\bf{x}:\ \bf{x}_{[m]} = \bf{z}\}}{\mu'(R)} = \frac{\mu'_{[m]}(\bf{z})}{\mu'_{[m]}(Q)} \geq \frac{\rme^{-(1-\delta/4)h}}{\mu'_{[m]}(Q)}.\]

Finally, towards property (e), Corollary~\ref{cor:cond-and-TC} and Lemma~\ref{lem:remaining-case} give
\[\TC(\mu_{|R}) \leq \frac{1}{\mu'(R)}(\TC(\mu') + \log 2) \leq \frac{8}{\delta}(4\TC(\mu) + \log 2).\]
Since we assume $\TC(\mu) > r^4 n$, this last upper bound is less than $33\TC(\mu)/\delta$ provided $n$ is sufficiently large in terms of $r$.
\end{proof}

Next we explain how the sets $S$, $R$ and $Q$ obtained by Proposition~\ref{prop:remaining-case} can be turned into a set $V$ as required by Proposition~\ref{prop:near-C}.

Let $\nu := \mu_{|R}$, and let $\langle \cdot \rangle_\nu$ denote integration with respect to $\nu$. We apply Theorem B to $\nu$ with input parameters $r$ and $\eps := 1/2$.  The number of terms in the resulting mixture is at most $c_\rm{B}\exp(c_\rm{B}\cdot \TC(\nu))$, and by conclusion (e) of Proposition~\ref{prop:remaining-case} this is at most $c_\rm{B}\exp(33c_\rm{B}E/\delta)$.  At least half the weight in the mixture must belong to terms that satisfy T$(\k n,r)$. Therefore there is at least one term, which we may write in the form $\rho\cdot \nu$, which satisfies T$(\k n,r)$ and also
\begin{equation}\label{eq:rho-int-big}
\langle \rho\rangle_\nu \geq  \frac{1}{2c_\rm{B}}\rme^{-33c_\rm{B}E/\delta}.
\end{equation}

For each $\bf{z} \in Q$, let
\[B_{\bf{z}} := \{\bf{x} \in R:\ \bf{x}_S = \bf{z}\},\]
and now let
\[Q_1 := \Big\{\bf{z} \in Q:\ \int_{B_{\bf{z}}}\rho\,\d\nu \geq \delta\cdot \nu(B_{\bf{z}})\cdot \langle \rho\rangle_\nu\Big\}.\]

The next step in the construction of $V$ is to choose a subset $U\subseteq R$ at random so that the events $\{U\ni \bf{x}\}$ have independent probabilities $\rho(\bf{x})$ for $\bf{x} \in R$.  In the final step, $V$ is obtained from $U$ by trimming it slightly.  Let $\bbE$ and $\rm{Var}$ denote expectation and variance with respect to the random choice of $U$.

\begin{lem}\label{lem:close-to-right-size}
	Provided $n$ is sufficiently large in terms of $r$, we have
	\[	\bbE\Big|\nu(U\cap B_{\bf{z}}) - \int_{B_{\bf{z}}} \rho\,\d\nu\Big| \leq \delta \int_{B_{\bf{z}}} \rho\,\d\nu\]
	for every $\bf{z} \in Q_1$.
\end{lem}

\begin{proof}
	In this proof we need the inequality
	\begin{equation}\label{eq:n-lg-enough}
	\rme^{-\delta h/4} \leq \frac{\delta^3}{2c_\rm{B}}\exp(-33c_\rm{B}E/\delta),
	\end{equation}
	where $h$ is the quantity from Lemma~\ref{lem:remaining-case}.  Since $h \geq 160c_\rm{B}E/\delta^2$, this holds provided
	\[\exp(-40c_\rm{B}E/\delta) \leq \frac{\delta^3}{2c_\rm{B}}\exp(-33c_\rm{B}E/\delta).\]
	Given our current assumption that $E > r^4n$, this last requirement holds for all $n$ that are sufficiently large in terms of $r$.

So now assume that~\eqref{eq:n-lg-enough} holds. Then for each $\bf{z} \in Q_1$ we may combine~\eqref{eq:rho-int-big},~\eqref{eq:n-lg-enough}, conclusion (d) from Proposition~\ref{prop:remaining-case}, and the definition of $Q_1$ to obtain
	\begin{equation}\label{eq:diffuse}
	\max_{\bf{x} \in B_{\bf{z}}}\nu(\bf{x}) \leq \frac{\delta^3}{2c_\rm{B}}\rme^{-33c_\rm{B}E/\delta}\nu(B_{\bf{z}}) \leq \delta^3\cdot \nu(B_{\bf{z}})\cdot \langle \rho\rangle_\nu \leq \delta^2\int_{B_{\bf{z}}}\rho\,\d\nu.
	\end{equation}

	Now consider the random variable
\[X := \nu(U\cap B_{\bf{z}}) - \int_{B_{\bf{z}}} \rho\,\d\nu = \sum_{\bf{x} \in B_{\bf{z}}}\nu(\bf{x})\cdot (1_{\{U\ni \bf{x}\}} - \rho(\bf{x})).\]
From the distribution of the random set $U$, it follows that $\bbE X = 0$ and
\[\rm{Var}(X) = \sum_{\bf{x}\in B_{\bf{z}}}(\nu(\bf{x}))^2\cdot \rm{Var}(1_{\{U\ni \bf{x}\}}) \leq \sum_{\bf{x}\in B_{\bf{z}}}(\nu(\bf{x}))^2\cdot \rho(\bf{x}).\]
By~\eqref{eq:diffuse}, this is at most
\[\delta^2 \int_{B_{\bf{z}}} \rho\,\d\nu\cdot \sum_{\bf{x}\in B_{\bf{z}}}\nu(\bf{x})\rho(\bf{x}) = \delta^2 \Big(\int_{B_{\bf{z}}} \rho\,\d\nu\Big)^2.\]
Therefore, by the monotonicity of Lebesgue norms,
\[\bbE|X| \leq \sqrt{\bbE (X^2)} = \sqrt{\rm{Var}(X)} \leq \delta\int_{B_{\bf{z}}} \rho\,\d\nu.\]
\end{proof}

\begin{cor}\label{cor:close-to-right-size}
	Provided $n$ is sufficiently large in terms of $r$,  we have
	\[	\bbE\big|\nu(U) - \langle\rho\rangle_\nu\big| \leq \sum_{\bf{z} \in Q}\bbE\Big|\nu(U\cap B_{\bf{z}}) - \int_{B_{\bf{z}}} \rho\,\d\nu\Big| \leq 3\delta \langle \rho \rangle_\nu.\]
\end{cor}

\begin{proof}
	The left-hand inequality follows from the triangle inequality and the fact that $(B_{\bf{z}}:\ \bf{z} \in Q)$ is a partition of $R$. So consider the right-hand inequality.  By separating $Q$ into $Q_1$ and $Q\setminus Q_1$, the sum over $Q$ is at most
	\begin{align}\label{eq:close-to-right-size}
		&\sum_{\bf{z} \in Q_1}\bbE\Big|\nu(U\cap B_{\bf{z}}) - \int_{B_{\bf{z}}} \rho\,\d\nu\Big| + \sum_{\bf{z} \in Q\setminus Q_1}\bbE \nu(U\cap B_{\bf{z}}) + \sum_{\bf{z} \in Q\setminus Q_1}\int_{B_{\bf{z}}} \rho\,\d\nu \nonumber \\
		&= 		\sum_{\bf{z} \in Q_1}\bbE\Big|\nu(U\cap B_{\bf{z}}) - \int_{B_{\bf{z}}} \rho\,\d\nu\Big| + 2\sum_{\bf{z} \in Q \setminus Q_1}\int_{B_{\bf{z}}} \rho\,\d\nu.
	\end{align}
	Using the previous lemma, the first sum in~\eqref{eq:close-to-right-size} is at most
	\[\delta\sum_{\bf{z} \in Q_1}\int_{B_{\bf{z}}}\rho\,\d\nu \leq \delta \langle \rho\rangle_\nu.\]
	By the definition of $Q_1$, the second sum in~\eqref{eq:close-to-right-size} is at most
	\[2\delta\sum_{\bf{z} \in Q \setminus Q_1}\nu(B_{\bf{z}})\cdot \langle \rho\rangle_\nu \leq 2\delta \langle \rho\rangle_\nu.\]
\end{proof}

\begin{lem}\label{lem:dbar-small}
		Provided $n$ is sufficiently large in terms of $r$, we have
	\[\bbE\big[ \nu(U)\cdot \ol{d_n}(\nu_{|U},\nu_{|\rho})\big] \leq 7\delta\langle \rho\rangle_\nu.\]
\end{lem}

\begin{proof}
	In the event that $\nu(U) > 0$, we can bound the distance $\ol{d_n}(\nu_{|U},\nu_{|\rho})$ using Corollary~\ref{cor:dbar-and-TV} with the partition $(B_{\bf{z}}:\ \bf{z} \in Q)$ of $R$. Each $B_{\bf{z}}$ has diameter at most $\delta$ according to $d_n$, and the diameter of the whole of $A^n$ is $1$, so that corollary gives
\[		\ol{d_n}(\nu_{|U},\nu_{|\rho}) \leq \delta\sum_{\bf{z} \in Q}\nu_{|\rho}(B_{\bf{z}}) + \frac{1}{2}\sum_{\bf{z} \in Q}|\nu_{|U}(B_{\bf{z}}) - \nu_{|\rho}(B_{\bf{z}})|.\]
Since the values $\nu_{|\rho}(B_{\bf{z}})$ for $\bf{z} \in Q$ must sum to $1$, the first term on the right is just $\delta$.  Let us also remove the factor $1/2$ from the second right-hand term for simplicity.
	
	Multiplying by $\nu(U)$, we obtain
	\begin{align*}
		\nu(U)\cdot \ol{d_n}(\nu_{|U},\nu_{|\rho}) &\leq \delta\nu(U) + \nu(U)\sum_{\bf{z}\in Q}|\nu_{|U}(B_{\bf{z}}) - \nu_{|\rho}(B_{\bf{z}})| \nonumber \\
		&= \delta\nu(U) + \sum_{\bf{z} \in Q}\big|\nu(U\cap B_{\bf{z}}) - \nu(U)\nu_{|\rho}(B_{\bf{z}})\big| \nonumber \\
		&\leq \delta\nu(U) + \sum_{\bf{z} \in Q}\Big|\nu(U\cap B_{\bf{z}}) - \int_{B_{\bf{z}}} \rho\,\d\nu \Big| \nonumber \\
		& \qquad \qquad \qquad \qquad  + \sum_{\bf{z} \in Q}|\nu(U) - \langle\rho\rangle_\nu|\cdot \nu_{|\rho}(B_{\bf{z}}),
	\end{align*}
	using that $\int_{B_{\bf{z}}}\rho\,\d\nu = \langle\rho\rangle_\nu\cdot \nu_{|\rho}(B_{\bf{z}})$ for the last step.  These inequalities also hold in case $\nu(U) = 0$.
	
	Taking expectations, we obtain
	\begin{align*}
		&\bbE\big[\nu(U)\cdot \ol{d_n}(\nu_{|U},\nu_{|\rho})\big] \nonumber \\ &\leq \delta\bbE \nu(U) + \sum_{\bf{z} \in Q}\bbE\Big|\nu(U\cap B_{\bf{z}}) - \int_{B_{\bf{z}}} \rho\,\d\nu \Big| + \bbE|\nu(U) - \langle\rho\rangle_\nu|\cdot \sum_{\bf{z}\in Q}\nu_{|\rho}(B_{\bf{z}}) \nonumber \\
		&= \delta\langle \rho\rangle_\nu + \sum_{\bf{z} \in Q}\bbE\Big|\nu(U\cap B_{\bf{z}}) - \int_{B_{\bf{z}}} \rho\,\d\nu \Big| + \bbE|\nu(U) - \langle\rho\rangle_\nu|.
	\end{align*}
Corollary~\ref{cor:close-to-right-size} bounds both the second and the third terms here: each is at most $3\delta\langle \rho\rangle_\nu$.  Adding these estimates completes the proof.
\end{proof}

\begin{proof}[Proof of Proposition~\ref{prop:near-C}]
	We may assume that $r < 1/8$ without loss of generality.
	
	As explained previously, there are three cases to consider.  In case $\mu$ has an atom $\bf{x}$ of measure at least~\eqref{eq:atom-bd}, we can let $V := \{\bf{x}\}$.  In case $\TC(\mu)\leq r^4n$, we use the set $V$ provided by Lemma~\ref{lem:small-TC}.
	
	So now suppose that neither of those cases holds.  This assumption provides the hypotheses for Proposition~\ref{prop:remaining-case}.  Let $R$ be the set provided by that proposition, and let $\nu:= \mu_{|R}$.  Choose $\rho$ and construct a random subset $U$ of $R$ as above.  By Corollary~\ref{cor:close-to-right-size}, Lemma~\ref{lem:dbar-small}, and two applications of Markov's inequality, the random set $U$ satisfies each of the inequalities
	\[\big|\nu(U) - \langle \rho\rangle_\nu\big| \leq 9\delta\langle \rho\rangle_\nu\]
	and
	\[\nu(U)\cdot \ol{d_n}(\nu_{|U},\nu_{|\rho}) \leq 21\delta\langle \rho\rangle_\nu\]
with probability at least $2/3$.
	
	Therefore there exists a subset $U \subseteq R$ which satisfies both of these inequalities.  For that set $U$, and recalling that we chose $\delta \leq 1/18$ (see~\eqref{eq:choice-of-eta}), the first inequality gives $\nu(U) \geq \langle \rho\rangle_\nu /2$.  Using this, the second inequality implies that
	\[\ol{d_n}(\nu_{|U},\nu_{|\rho}) \leq 42\delta.\]

Since we also chose $\delta \leq r^2/42$ (see~\eqref{eq:choice-of-eta} again), and since $\nu_{|\rho}$ satisfies T$(\k n,r)$, we may now repeat the proof of Lemma~\ref{lem:small-TC} (based on Proposition~\ref{prop:T-under-dbar}).  It gives another subset $V$ of $A^n$ with $\nu_{|U}(V) \geq 1 - 4r \geq 1/2$ and such that the measure $(\nu_{|U})_{|V}$ satisfies T$(\k n,9r)$, provided $n$ is sufficiently large in terms of $\k$ and $r$.  Since  $\nu_{|U}$ is supported on $U$, we may replace $V$ with $U\cap V$ if necessary and then assume that $V\subseteq U$.  Having done so, we have $(\nu_{|U})_{|V} = (\mu_{|U})_{|V} = \mu_{|V}$.

Finally we must prove the lower bound on $\mu(V)$ for a suitable constant $c$.  Recalling~\eqref{eq:rho-int-big}, it follows from the estimates above that
\begin{multline*}
\mu(V) \geq \mu(U)\cdot \mu(V\,|\,U) = \mu(R)\cdot \nu(U)\cdot \nu(V\,|\,U) \\ \geq \frac{\delta}{32M}\cdot \frac{\langle\rho\rangle_\nu}{2}\cdot \frac{1}{2}
	\geq \frac{\delta}{256c_\rm{B}M}\cdot \rme^{-33c_\rm{B}E/\delta}.
	\end{multline*}
Recalling the definition of $M$ in Lemma~\ref{lem:remaining-case}, this last expression is equal to
\[\frac{\delta}{256c_\rm{B}\lceil n\log|A|\rceil}\cdot \rme^{-33c_\rm{B}E/\delta}.\]
Since we have restricted attention to the last of our three cases, we know that $E > r^4 n$.  Therefore, for any fixed $c > 33c_\rm{B}/\delta$, the last expression must be greater than $\rme^{-cE}$ for all sufficiently large $n$, since the negative exponential is stronger than the prefactor as $n\to\infty$.
\end{proof}

\begin{rmk}
In Part III we apply Theorem C during the proof of Theorem A.  In that application, the measure $\mu$ is the conditional distribution of $n$ consecutive letters in a stationary ergodic process, given the corresponding letters in another process coupled to the first. In this case, the Shannon--McMillan theorem enables us to obtain conclusion (d) of Proposition~\ref{prop:remaining-case} after just trimming the measure $\mu$ slightly, without using Lemma~\ref{lem:remaining-case}.  This makes for a slightly easier proof of Theorem C in this case.
\end{rmk}

\begin{rmk}
	In the proof of Theorem C, we first obtain the new measure $\nu = \mu_{|R}$. Then we apply Theorem B to it, consider just one of the summands $\rho\cdot \nu$ in the resulting mixture, and use a random construction to produce a set $U$ such that $\ol{d_n}(\nu_{|U},\nu_{|\rho})$ is small.
	
	An alternative is to consider the whole mixture
	\begin{equation}\label{eq:mixture-from-B}
\nu = \rho_1\cdot \nu_1 + \cdots + \rho_m\cdot \nu_m
\end{equation}
	given by Theorem B, and then use a random partition $U_1$, \dots, $U_m$ so that (i) each $\bf{x} \in A^n$ chooses independently which cell to belong to and (ii) each event $\{U_j\ni\bf{x}\}$ has probability $\rho_j(\bf{x})$.  Using similar estimates to those above, one can show that, in expectation, most of the weight in the mixture is on terms for which $\ol{d_n}(\nu_{|U_j},\nu_{|\rho_j})$ is small.
	
	However, in this alternative approach, we must still begin by passing from $\mu$ to $\nu = \mu_{|R}$.  It does not give directly a partition as required by Theorem C, since we have $\nu_{|U_j} = \mu_{|R\cap U_j}$, and the sets $R\cap U_j$ do not cover the whole of $A^n$.  For this reason, it seems easier to use just one well-chosen term from~\eqref{eq:mixture-from-B}, as we have above.
\end{rmk}

\section{A reformulation: extremality}\label{sec:ext}

The work of Parts II and III is much easier if we base it on a slightly different notion of measure concentration than that studied so far.  This alternative notion comes from an idea of Thouvenot in ergodic theory, described in~\cite[Definition 6.3]{Tho02}.  It is implied by a suitable T-inequality, so we can still bring Theorem C to bear when we need it later, but this new notion is more robust and therefore easier to carry from one setting to another.

\subsection{Extremal measures}

Here is our new notion of measure concentration.

\begin{dfn}\label{dfn:pre-ext}
	Let $(K,d,\mu)$ be a metric probability space.  It is \textbf{$(\k,r)$-extremal} if any representation of $\mu$ as a mixture $\int \mu_\bullet\,\d P$ satisfies
	\[\int \ol{d}(\mu_\omega,\mu)\,P(\d\omega) \leq \frac{1}{\k}\int \rmD(\mu_\omega\,\|\,\mu)\,P(\d\omega) + r.\]
\end{dfn}

In the first place, notice that if $r$ is at least the diameter of $(K,d)$ then $(K,d,\mu)$ is $(\k,r)$-extremal for every $\k > 0$.

If $(K,d,\mu)$ satisfies T$(\k,r)$ then it is also $(\k,r)$-extremal: T$(\k,r)$ gives an inequality for each $\ol{d}(\mu_\o,\mu)$, and we can simply integrate this with respect to $P(\d \o)$.  However, extremality is generally slightly weaker than a T-inequality.  This becomes clearer from various stability properties of extremality.  In particular, the next lemma shows that extremality survives with roughly the same constants under sufficiently small perturbations in $\ol{d}$, whereas T-inequalities can degrade substantially unless we also condition the perturbed measure on a subset (Proposition~\ref{prop:T-under-dbar}).  This deficiency of T-inqualities may be seen in the examples sketched in the second remark after the proof of Proposition~\ref{prop:T-under-dbar}.

\begin{lem}[Stability under perturbation in transportation]\label{lem:ext-dbar-stab}
	Let $\mu,\mu' \in \Pr(K)$, assume that $\mu$ is $(\k,r)$-extremal, and let $\delta := \ol{d}(\mu,\mu')$.  Then $\mu'$ is $(\k,r+2\delta)$-extremal.
\end{lem}

\begin{proof}
	By the vague compactness of $\Pr(K)$, there is a coupling $\l$ of $\mu'$ and $\mu$ such that
	\[\int d(x',x)\,\l(\d x',\d x) = \delta.\]
	Let us disintegrate $\l$ over the first coordinate, thus:
	\begin{equation}\label{eq:kappa-to-mu}
	\l = \int_K \delta_{x'}\times \l_{x'}\ \mu'(\d x').
	\end{equation}
	
	Now suppose that $\mu'$ is equal to the mixture $\int \mu'_\bullet\,\d P$. Combining this with~\eqref{eq:kappa-to-mu}, we obtain
\[		\mu = \int_K \l_{x'}\ \mu'(\d x') = \int_\O\Big[\int_K\l_{x'}\ \mu'_\o(\d x')\Big]\ P(\d \o) = \int_\O\mu_\o\ P(\d \o),\]
	where
	\begin{equation}\label{eq:nu-from-g}
		\mu_\o := \int_K\l_{x'} \ \mu'_\o(\d x').
	\end{equation}
	Thus, we have turned our representation of $\mu'$ as a mixture into a corresponding representation of $\mu$. Therefore the assumed extremality gives
	\[\int \ol{d}(\mu_\o,\mu)\ P(\d\o) \leq \frac{1}{\k}\int \rmD(\mu_\o\,\|\,\mu)\ P(\d\o) + r.\]
	
	The triangle inequality for $\ol{d}$ gives
\[		\int \ol{d}(\mu'_\o,\mu')\ P(\d\o) \leq \int \ol{d}(\mu'_\o,\mu_\o)\ P(\d\o) + \int \ol{d}(\mu_\o,\mu)\ P(\d\o) + \ol{d}(\mu,\mu')\]
	The last right-hand term here is $\delta$.  We have just found a bound for the middle term.  Finally, by the definition~\eqref{eq:nu-from-g}, the first term is at most
\begin{multline*}
\int \Big[\iint d(x',x)\ \l_{x'}(\d x)\ \mu'_\o(\d x')\Big]\ P(\d\o)
		= \iint d(x',x)\ \l_{x'}(\d x)\ \mu'(\d x')\\
		= \int d\ \d\l = \delta.
	\end{multline*}
	Combining these estimates gives
	\[\int \ol{d}(\mu'_\o,\mu')\ P(\d\o) \leq \frac{1}{\k}\int \rmD(\mu_\o\,\|\,\mu)\ P(\d\o) + r + 2\delta.\]
	Finally, for the remaining integral on the right-hand side, we apply the fact that KL divergence is non-increasing when one forms the compounds of two probability measures with the same kernel. See, for instance,~\cite[Section 4.4, part 1]{CovTho06} for this standard consequence of the chain rule. This gives
	\begin{align*}
		\int \rmD(\mu_\o\,\|\,\mu)\ P(\d\o)
		&= \int \rmD\Big(\int_K\l_{x'}\ \mu'_\o(\d x')\,\Big\|\,\int_K\l_{x'}\ \mu'(\d x') \Big)\ P(\d\o)\\
		&\leq \int \rmD(\mu'_\o\,\|\,\mu')\ P(\d\o).
	\end{align*}
\end{proof}

Next we show how extremality is inherited by product measures.

\begin{lem}[Products of extremal measures]\label{lem:pre-ext-prod}
	Let $(K,d_K)$ and $(L,d_L)$ be compact metric spaces, let $0 < \a < 1$, and let $d$ be the following metric on $K\times L$:
	\begin{equation}\label{eq:prod-met-again}
	d((x,y),(x',y')) := \a d_K(x,x') + (1-\a)d_L(y,y').
	\end{equation}
	Let $\mu \in \Pr(K)$ and $\nu \in \Pr(L)$, and assume that they are $(\a\k,r_K)$-extremal and $((1-\a)\k,r_L)$-extremal, respectively.  Let $r := \a r_K + (1-\a)r_L$. Then $\mu\times \nu$ is $(\k,r)$-extremal on the metric space $(K\times L,d)$.
\end{lem}

\begin{proof}
	Consider a representation
	\[\mu\times \nu = \int_\O \l_\o\ P(\d\o).\]
	For each $\o$, let $\l_{K,\o}$ be the marginal of $\l_\o$ on $K$, and consider the disintegration of $\l_\o$ over the first coordinate:
	\[\l_\o = \int_K (\delta_x \times \l_{L,(\o,x)})\ \l_{K,\o}(\d x).\]
	
	We apply Lemma~\ref{lem:dist-from-prod-meas} to each $\o$ separately and then integrate.  This gives
	\begin{multline}\label{eq:int-dbar-bd}
		\int \ol{d}(\l_\o,\mu\times \nu)\ P(\d \o) \\ \leq \a\int \ol{d_K}(\l_{K,\o},\mu)\ P(\d \o) + (1-\a)\iint \ol{d_L}(\l_{L,(\o,x)},\nu)\ \l_{K,\o}(\d x)\ P(\d \o).
	\end{multline}
	
	The measure $\mu$ is equal to the mixture $\int \l_{K,\o}\,\d P$.  Since $\mu$ is $(\a\k,r_K)$-extremal, it follows that the first right-hand term in~\eqref{eq:int-dbar-bd} is at most
	\[\frac{1}{\k}\int \rmD(\l_{K,\o}\,\|\,\mu)\ P(\d \o) + \a r_K.\]
	
	On the other hand, let $\O' := \O \times K$, and let $Q := P\ltimes \l_{K,\bullet}$. Then $\nu$ is equal to the mixture
	\[\iint \l_{L,(\o,x)}\,\l_{K,\o}(\d x)\,P(\d \o) = \int \l_{L,(\o,x)}\,Q(\d\o,\d x).\]
	Applying the fact that $\nu$ is $((1-\a)\k,r_L)$-extremal to this mixture, it follows that the second right-hand term in~\eqref{eq:int-dbar-bd} is at most
	\[\frac{1}{\k}\iint \rmD(\l_{L,(\o,x)}\,\|\,\nu)\ \l_{K,\o}(\d x)\ P(\d \o) + (1-\a)r_L.\]
	
	Adding these two estimates, we obtain
	\begin{multline*}
		\int \ol{d}(\l_\o,\mu\times \nu)\ P(\d \o) \\ \leq \frac{1}{\k}\int \rmD(\l_{K,\o}\,\|\,\mu)\ P(\d \o) +\frac{1}{\k}\iint \rmD(\l_{L,(\o,x)}\,\|\,\nu)\ \l_{K,\o}(\d x)\ P(\d \o) + r.
	\end{multline*}
	By the chain rule for KL divergence, this is equal to
	\[\frac{1}{\k}\int \rmD(\l_\o\,\|\,\mu\times \nu)\ P(\d \o) + r.\]
\end{proof}

\subsection{Extremal kernels}

In Parts II and III, we use extremality not just for one measure on a compact metric space $(K,d)$, but for a $\Pr(K)$-valued kernel $\mu_\bullet$ on another probability space $(Y,\nu)$.  In those applications, the important property is that the individual measures $\mu_y$ are highly extremal for most $y \in Y$ according to $\nu$. An exceptional set of small probability is allowed.

With this in mind, let us extend Definition~\ref{dfn:pre-ext} to kernels.

\begin{dfn}\label{dfn:ext}
	Let $(Y,\nu)$ be a probability space, let $(K,d)$ be a compact metric space, and let $\mu_\bullet$ be a kernel from $Y$ to $K$.  The pair $(\nu,\mu_\bullet)$ is \textbf{$(\k,r)$-extremal} if there is a measurable function
	\[Y\to (0,\infty):y\mapsto r_y\]
	such that $\mu_y$ is $(\k,r_y)$-extremal for every $y \in Y$ and such that
	\begin{equation}\label{eq:r-ave}
	\int r_y\,\nu(\d y) \leq r.
	\end{equation}
\end{dfn}

Thus, we control the extremality of the individual measures $\mu_y$ by the average of their parameters in~\eqref{eq:r-ave}, rather than by explicitly introducing a `bad set' of small probability.  This turns out to give much simpler estimates in many of the proofs below, starting with Lemma~\ref{lem:ext-ker-dbar-stab}.  Of course, Definition~\ref{dfn:ext} is easily related to such a `bad set' via Markov's inequality:

\begin{lem}\label{lem:ker-ext-bad-set}
If $(\nu,\mu_\bullet)$ is $(\k,r)$-extremal, then there is a subset $Y_1 \subseteq Y$ such that $\nu(Y_1) \geq 1-\sqrt{r}$ and such that $\mu_y$ is $(\k,\sqrt{r})$-extremal for every $y \in Y$.

On the other hand, if $(K,d)$ has diameter $r_0$, $Y_1 \subseteq Y$ is measurable, and $\mu_y$ is $(\k,r)$-extremal for all $y\in Y_1$, then $(\nu,\mu_\bullet)$ is extremal with parameters
\[\big(\k,r + r_0\nu(Y\setminus Y_1)\big).\]
\qed
	\end{lem}

\begin{ex}
Let $p_1\mu_1 + \dots + p_m\mu_m$ be a mixture given by Theorem B.  Regard $p = (p_1,\dots,p_m)$ as a measure on $[m]$ and $(\mu_j)_{j=1}^m$ as a kernel from $[m]$ to $A^n$. Then the pair $(p,\mu_\bullet)$ is $(rn/1200,r+\eps)$-extremal. \qed
\end{ex}

Definition~\ref{dfn:ext} clearly depends on the kernel $\mu_\bullet$ only up to agreement $\nu$-almost everywhere.  Therefore it may be regarded unambiguously as a property of the measure $\nu\ltimes \mu_\bullet$ on the product space $Y\times K$.  In the sequel we sometimes write $(\k,r)$-extremality as a property of a measure on $Y\times K$ rather than of a pair such as $(\nu,\mu_\bullet)$.

The next result is an easy extension of Lemma~\ref{lem:ext-dbar-stab} to this setting.

\begin{lem}[Stability under perturbation in transportation]\label{lem:ext-ker-dbar-stab}
	Let $(Y,\nu)$ be a probability space, let $(K,d)$ be a compact metric space, and let $\mu_\bullet$ and $\mu'_\bullet$ be two kernels from $Y$ to $K$.  Assume that $(\nu,\mu_\bullet)$ is $(\k,r)$-extremal, and let
	\[\delta := \int \ol{d}(\mu_y,\mu'_y)\,\nu(\d y).\]
	Then $(\nu, \mu'_\bullet)$ is $(\k,r+2\delta)$-extremal.
\end{lem}

\begin{proof}
	Let $y\mapsto r_y$ be the function promised by Definition~\ref{dfn:ext} for $(\nu,\mu_\bullet)$.  By applying Lemma~\ref{lem:ext-dbar-stab} for each $y$ separately, it follows that the function
	\[y\mapsto r_y + 2\ol{d}(\mu_y,\mu'_y)\]
	serves the same purpose for $(\nu,\mu'_\bullet)$.  This function has integral at most $r+2\delta$.
\end{proof}

Lemma~\ref{lem:pre-ext-prod} also has an easy generalization to kernels.

\begin{lem}[Pointwise products of extremal measures]\label{lem:ext-prod}
		Let $(K,d_K)$ and $(L,d_L)$ be compact metric spaces, let $0 < \a < 1$, and let $d$ be the metric~\eqref{eq:prod-met-again} on $K\times L$.
		
		Let $\mu_\bullet$ and $\mu'_\bullet$ be kernels from $(Y,\nu)$ to $K$ and $L$, respectively, and assume that $(\nu,\mu_\bullet)$ and $(\nu,\mu'_\bullet)$ are $(\a\k,r_K)$-extremal and $((1-\a)\k,r_L)$-extremal, respectively.  Let $r := \a r_K + (1-\a)r_L$.
		
		Then the measure $\nu$ and pointwise-product kernel $y\mapsto \mu_y\times \mu'_y$ are $(\k,r)$-extremal on the metric space $(K\times L,d)$.
	\end{lem}

\begin{proof}
Let $y\mapsto r_{K,y}$ and $y\mapsto r_{L,y}$ be functions as promised by Definition~\ref{dfn:ext} for $\mu_\bullet$ and $\mu'_\bullet$, respectively.  For each $y \in Y$, apply Lemma~\ref{lem:pre-ext-prod} to conclude that $\mu_y\times \mu'_y$ is $(\k,\a r_{K,y} + (1-\a)r_{L,y})$-extremal.  Finally, observe that
\[\int (\a r_{K,y} + (1-\a)r_{L,y})\,\nu(\d y) \leq \a r_K + (1-\a)r_L. \]
\end{proof}

We often use extremality when $\mu_\bullet$ is a conditional distribution of a $K$-valued random variable $F$ given $\G$, where $F$ is defined on a probability space $(\O,\F,P)$ and $\G$ is a $\s$-subalgebra of $\F$.  Then the following nomenclature is useful.

\begin{dfn}\label{dfn:RV-ext}
In the setting above, $F$ is \textbf{$(\k,r)$-extremal over $\G$} if $P\ltimes \mu_\bullet$ is $(\k,r)$-extremal, where $\mu_\bullet$ is a conditional distribution of $F$ given $\G$.
\end{dfn}

In this setting, the important consequence of extremality is the following.

\begin{lem}\label{lem:ext-for-finite}
Let $(\O,\F,P)$ be a probability space, let $\G\subseteq \scrH$ be $\s$-subalgebras of $\F$ that are countably generated modulo $P$, let $(K,d)$ be a finite metric space, and let $F:\O\to K$ be measurable and $(\k,r)$-extremal over $\G$.  Let $\mu_\bullet$ and $\nu_\bullet$ be conditional distributions for $F$ given $\G$ and $\scrH$, respectively.  Then
\[\int \ol{d}(\nu_\o,\mu_\o)\,P(\d\o) \leq \frac{1}{\k}\big[\rmH(F\,|\,\G) - \rmH(F\,|\,\scrH)\big] + r.\]
\end{lem}

One can drop the assumption of countable generation, but we omit the details.

\begin{proof}
	Let $\o \mapsto r_\o$ be the function promised by Definition~\ref{dfn:ext} for $(P,\mu_\bullet)$.
	
Since $\scrH$ is countably generated modulo $P$, there is a measurable map $\psi:\O \to Y$ to a standard measurable space such that $\psi$ generates $\scrH$ modulo $P$. Up to agreement $P$-a.e., we may write $\nu_\o$ as $\nu'_{\psi(\o)}$ for some kernel $\nu'_\bullet$ from $Y$ to $K$.

In addition, let $Q_\bullet:\O\to \Pr(Y)$ be a conditional distribution for $\psi$ given $\G$.  Such a $Q_\bullet$ exists because $Y$  is standard. Since $\scrH \supseteq \G$, the tower property of conditional expectation gives
\begin{equation}\label{eq:decomp-each-omega}
\mu_\o = \int_Y \nu'_y\ \ Q_\o(\d y) \quad \hbox{for}\ P\hbox{-a.e.}\ \o.
\end{equation}
Therefore, for $P$-a.e. $\o$, we may apply the $(\k,r_\o)$-extremality of $\mu_\o$ to obtain
\[\int \ol{d}(\nu'_y,\mu_\o)\,Q_\o(\d y) \leq \frac{1}{\k}\int \rmD(\nu'_y\,\|\,\mu_\o)\ Q_\o(\d y) + r_\o.\]
Integrating with respect to $P(\d \o)$, this gives
\begin{align*}
\int \ol{d}(\nu_\o,\mu_\o)\,P(\d\o) &= \iint \ol{d}(\nu'_y,\mu_\o)\,Q_\o(\d y)\,P(\d \o) \\ &\leq \frac{1}{\k}\iint \rmD(\nu'_y\,\|\,\mu_\o)\ Q_\o(\d y)\ P(\d \o) + \int r_\o\,P(\d\o)\\
&\leq \frac{1}{\k}\int \rmD(\nu_\o\,\|\,\mu_\o)\ P(\d \o) + r.
\end{align*}

On the other hand, using again that $\scrH\supseteq \G$, Lemma~\ref{lem:I-and-KL} gives
\[\int \rmD(\nu_\o\,\|\,\mu_\o)\ P(\d\o) = \rmH(F\,|\,\G) - \rmH(F\,|\,\scrH).\]
\end{proof}

Combining Lemmas~\ref{lem:ext-for-finite} and~\ref{lem:ext-ker-dbar-stab} gives the following useful consequence.

\begin{cor}[Inheritance of extremality]\label{cor:refinement-still-ext}
	In the setting of Lemma~\ref{lem:ext-for-finite}, let
	\[a := \rmH(F\,|\,\G) - \rmH(F\,|\,\scrH).\]
Then $F$ is $(\k,2a/\k + 3r)$-extremal over $\scrH$. \qed
\end{cor}

Finally, here is an analog of Lemma~\ref{lem:stab-lift} for extremal kernels.

\begin{lem}[Stability under lifting]\label{lem:ext-stab-lift}
	Let $d_n$ be the normalized Hamming metric on $A^n$ for some finite set $A$, and let $\nu\ltimes \mu_\bullet$ be a probability measure on $Y\times A^n$.  Let $S\subseteq [n]$ satisfy $|S| \geq (1-a)n$, and let $\mu_{S,y}$ be the projection of $\mu_y$ to $A^S$ for each $y \in Y$.  If $\nu\ltimes \mu_{S,\bullet}$ is $(\k,r)$-extremal, then $\nu\ltimes \mu_\bullet$ is $(n\k/|S|,r + a)$-extremal.
\end{lem}

\begin{proof}
First consider the case when $Y$ is a single point.  Effectively this means we have a single measure $\mu$ on $A^n$ for which $\mu_S$ is $(\k,r)$-extremal.  Suppose $\mu$ is represented as the mixture $\int \mu_\o\,P(\d\o)$. Projecting to $A^S$, this becomes $\mu_S = \int \mu_{S,\o}\,P(\d\o)$, and we have $\rmD(\mu_{S,\o}\,\|\,\mu_S) \leq \rmD(\mu_\o\,\|\,\mu)$ for each $\o$ as in the proof of Lemma~\ref{lem:stab-lift}. So our assumption on $\mu_S$ gives
\[\int \ol{d_S}(\mu_{S,\o},\mu_S)\,P(\d \o) \leq \frac{1}{\k}\int \rmD(\mu_{S,\o}\,\|\,\mu_S)\,P(\d \o) + r \leq \frac{1}{\k}\int \rmD(\mu_\o\,\|\,\mu)\,P(\d \o) + r .\]
Now, just as in the proof of Lemma~\ref{lem:stab-lift}, we may lift a family of couplings which realize the distances $\ol{d_S}(\mu_{S,\o},\mu_S)$, and thereby turn the above inequalities into
\[\int \ol{d_n}(\mu_\o,\mu)\,P(\d \o) \leq \frac{|S|}{n}\Big(\frac{1}{\k}\int \rmD(\mu_\o\,\|\,\mu)\,P(\d \o) + r\Big) + a.\]
Casually taking $r$ outside the parentheses, this shows ${(n\k/|S|,r+a)}$-extremality.

In the case of general $Y$ and a kernel $\mu_\bullet$, simply apply the special case above to each $\mu_y$ separately.
\end{proof}

\part{RELATIVE BERNOULLICITY}

To prove Theorem A, we construct a factor of $(X,\mu,T)$ using a sequence of applications of Theorem C, and then show that the measure concentration promised by that theorem implies relative Bernoullicity.  The second of these steps rests on Thouvenot's relative version of Ornstein theory.  In this part of the paper we recall the main results of that theory, and provide some small additions which assist in its application later.

At the heart of Ornstein's original work on Bernoullicity are two necessary and sufficient conditions for an ergodic finite-state process to be isomorphic to a Bernoulli shift: the finitely determined property, and the very weak Bernoulli property.  More recently, other necessary and sufficient conditions have been added to this list, such as extremality~\cite[Definition 6.3]{Tho02} or a version of measure concentration~\cite{MarShi94} (historically referred to as the `almost blowing-up property' in ergodic theory).

This theory can be generalized to characterize relative Bernoullicity of an ergodic measure-preserving system over a distinguished factor map.  This was realized by Thouvenot, who proved the first key extensions of Ornstein's results in~\cite{Tho75,Tho75b} (see also~\cite{Kie84} for another treatment).  Thouvenot's more recent survey~\cite{Tho02} describes the broader context.  In the first instance, the theory is based on a relative version of the finitely determined property.  Thouvenot's theorem that relative finite determination implies relative Bernoullicity is recalled as Theorem~\ref{thm:FD-Bern} below.

In general, a proof that a given system is relatively Bernoulli over some factor map has two stages: (i) a proof of relative finite determination, and then (ii) an appeal to Theorem~\ref{thm:FD-Bern}.  To prove Theorem A, we construct the new factor map $\pi_1$ that it promises, and then show relative Bernoullicity over $\pi_1$ following those two stages.

Stage (ii) simply amounts to citing Theorem~\ref{thm:FD-Bern} from Thouvenot's work.  That theorem is the most substantial fact that we bring into play from outside this paper.

There is more variety among instances of stage (i) in the literature: that is, proofs of relative finite determination itself.  For the construction in our proof of Theorem A, we obtain it from a relative version of extremality.  Relative extremality is a much more concrete property than finite determination, and is built on the ideas of Section~\ref{sec:ext}.

Regarding these two stages, let us make a curious observation. The assumption of ergodicity in Theorem A is needed only as a necessary hypothesis for Theorem~\ref{thm:FD-Bern}. We do not use it anywhere in stage (i) of our proof of Theorem A.  In particular, because we have chosen to prove Theorem C in general, rather than just for the examples coming from ergodic theory, we make no direct use of the Shannon--McMillan theorem in this paper (see the first remark at the end of Section~\ref{sec:C}). Ergodicity does play a crucial role within Thouvenot's proof of Theorem~\ref{thm:FD-Bern} via the ergodic and Shannon--McMillan theorems.

This part of the paper has the following outline. Section~\ref{sec:rel-ent} recalls some standard theory of Kolmogorov--Sinai entropy relative to a factor.  Sections~\ref{sec:version-of-Orn} and~\ref{sec:rel-Orn} introduce relative finite determination and relative extremality, respectively.  Then Section~\ref{sec:rel-Orn-proofs} proves the crucial implication from the latter to the former.

The ideas in these sections have all been known in this branch of ergodic theory for many years, at least as folklore. For the sake of completeness, I include full proofs of all but the most basic results that we subsequently use, with the exception of Theorem~\ref{thm:FD-Bern}.  Nevertheless, this part of the paper covers only those pieces of relative Ornstein theory that are needed in Part III, and omits many other topics.  For example, a knowledgeable reader will see a version of relative very weak Bernoullicity at work in the proof of Proposition~\ref{prop:ext-ABI} below, and will recognize a relative version of `almost block independence' in its conclusion: see the remark following that proof.  But we do not introduce these properties formally, and leave aside the implications between them.  Relative very weak Bernoullicity already appears in~\cite{Tho75}, and is known to be equivalent to relative finite determination by results of Thouvenot~\cite{Tho75}, Rahe~\cite{Rah78} and Kieffer~\cite{Kie82}.

\section{Background on relative entropy in ergodic theory}\label{sec:rel-ent}

This section contains the most classical facts that we need. We recount some of them carefully for ease of reference later and to establish notation, but we omit most of the proofs.

\subsection{Factors, observables, and relative entropy}

For any set $A$ we write $T_A$ for the leftward coordinate shift on $A^\bbZ$.  In the sequel the letters $A$ and $B$ always denote finite alphabets.

Most traditional accounts of entropy in ergodic theory are explained in terms of finite measurable partitions.  Here we adopt the slightly different language of observables, which is closer to the spirit of coding and information theory (see, for instance,~\cite[Chapter 5]{Bil65}).

Let $(X,\mu,T)$ be a measure-preserving system.  An $A$-valued \textbf{observable} on $X$ is simply a measurable map $X\to A$.  An $A$-valued \textbf{process} is a bi-infinite sequence
\[\xi = (\dots ,\xi_{-1},\xi_0,\xi_1,\dots)\]
of observables which together define a factor map from $(X,\mu,T)$ to $(A^\bbZ,\xi_\ast\mu,T_A)$.  We often adopt the point of view that a process is a special kind of factor map.  By equivariance, it follows that these maps must satisfy
\begin{equation}\label{eq:obs-proc}
\xi_n(x) = \xi_0(T^nx) \quad \mu\hbox{-a.s.}
\end{equation}
Thus, up to agreement modulo $\mu$, $\xi$ is completely specified by the time-zero coordinate $\xi_0$.  On the other hand, any $A$-valued observable $\xi_0$ gives rise to a unique $A$-valued process: simply use~\eqref{eq:obs-proc} as the definition of all the other $\xi_n$s.

Given a sequence $\xi$ as above and any subset $F \subseteq \bbZ$, we let
\[\xi_F := (\xi_n)_{n\in F}.\]
We use this notation mostly when $F$ is a discrete interval.  For any integers $a$ and $b$, we write
\[[a;b] = (a-1;b] = [a;b+1) = (a-1;b+1) := \{a,a+1,\dots,b\},\] interpreting this as $\emptyset$ if $b < a$.  We extend this notation to allow $a = -\infty$ or $b = \infty$ in the obvious way.

Given a system $(X,\mu,T)$ and a process $\xi$, the \textbf{finite-dimensional distributions} of $\xi$ are the distributions of the observables $\xi_F$ corresponding to finite subsets $F \subseteq \bbZ$.  In case $X = A^\bbZ$, $T = T_A$, and $\xi$ is the canonical process, we denote the distribution of $\xi_F$ by $\mu_F$.

The following conditional version of this notation is less standard, but also useful later in the paper.

\begin{dfn}[Block kernels]\label{dfn:block-kernel}
	Given a system $(X,\mu,T)$, processes $\xi:X\to A^\bbZ$ and $\pi:X\to B^\bbZ$, and $n \in \bbN$, the \textbf{$(\xi,\pi,n)$-block kernel} is the conditional distribution
	\[\mu(\xi_{[0;n)} = \bf{a}\,|\,\pi_{[0;n)} = \bf{b}) \quad (\bf{a} \in A^n,\ \bf{b} \in B^n),\]
	regarded as a kernel from $B^n$ to $A^n$. We define it arbitrarily for any $\bf{b}$ for which $\mu(\pi_{[0;n)} = \bf{b}) = 0$.
	
	In case $X = B^\bbZ\times A^\bbZ$, $T = T_{B\times A}$, and $\xi$ and $\pi$ are the canonical $A$- and $B$-valued processes on $X$, we may refer simply to the \textbf{$n$-block kernel}, and denote it by $\mu^\bl_{[0;n)}(\bf{a}\,|\,\bf{b})$.
\end{dfn}

Now let $\xi:X \to A^\bbZ$ be a process and let $\pi:(X,\mu,T)\to (Y,\nu,S)$ be a general factor map.  The \textbf{relative Kolmogorov--Sinai} (`\textbf{KS}') \textbf{entropy of $(X,\mu,T,\xi)$ over $\pi$} is the quantity
\[\rmh(\xi,\mu,T\,|\,\pi) = \lim_{n\to\infty}\frac{1}{n}\rmH_\mu(\xi_{[0;n)}\,|\,\pi).\]
The limit exists by the usual subadditivity argument.  In case $\pi$ is trivial it simplifies to
\[\rmh(\xi,\mu,T) = \rmh(\xi_\ast\mu,T_A),\] the KS entropy of the factor system $(A^\bbZ,\xi_\ast\mu,T_A)$, by the Kolmogorov--Sinai theorem.

Henceforth we frequently omit the subscript $\mu$ from $\rmH_\mu$, provided this measure is clear from the context.

In general, if $\pi$ and $\phi$ are measurable maps on $(X,\mu)$, then we say that $\phi$ is \textbf{$\pi$-measurable} if $\phi$ is measurable with respect to the $\s$-algebra generated by $\pi$ up to $\mu$-negligible sets.

The next lemma describes a standard alternative approach to relative KS entropy in case the factor map $\pi$ is also a process.

\begin{lem}\label{lem:rel-ent-formula}
If $\pi:X \to B^\bbZ$ is a process, then 
\[\rmH(\xi_{[0;n)}\,|\,\pi_{[0;n)}) \geq \rmH(\xi_{[0;n)}\,|\,\pi)\]
for every $n$, both of these sequences are subadditive in $n$, and both have the form
\[\rmh(\xi,\mu,T\,|\,\pi)\cdot n + o(n) \quad \hbox{as} \ n\to\infty.\]
If $\pi$ is an arbitrary factor map, then
\[\rmh(\xi,\mu,T\,|\,\pi) = \inf\big\{\rmh(\xi,\mu,T\,|\,\pi'):\ \pi'\ \hbox{is a $\pi$-measurable process}\big\}.\]
\qed
\end{lem}

If $\phi$ is another general factor map of $(X,\mu,T)$, then we define
\[\rmh(\phi,\mu,T\,|\,\pi) := \sup\big\{\rmh(\xi,\mu,T\,|\,\pi):\ \xi\ \hbox{is a $\phi$-measurable process}\big\}.\]
In particular,
\[\rmh(\mu,T\,|\,\pi) := \rmh(\rm{id}_X,\mu,T\,|\,\pi) = \sup\big\{\rmh(\xi,\mu,T\,|\,\pi):\ \xi\ \hbox{a process on $(X,\mu,T)$}\big\}.\]

Once again, an important special case arises when $X = B^\bbZ\times A^\bbZ$, $T = T_{B\times A}$, and $\pi$ is the canonical $B$-valued process on $X$. In this case we frequently abbreviate
\[\rmh(\mu,T\,|\,\pi) =: \rmh(\mu\,|\,\pi).\]

If $\pi:X \to Y$ and $\phi:X\to Z$ are maps between measurable spaces, then we write $\pi\vee \phi$ for the map
\[X\to Y\times Z:x\mapsto (\pi(x),\phi(x)).\]
In the case of observables, this notation deliberately suggests the common refinement of a pair of partitions.  We extend it to larger collections of maps or observables in the obvious way.

Relative KS entropy is a dynamical cousin of conditional Shannon entropy, and it has many analogous properties. The next result is an analog of the chain rule for Shannon entropy.  It follows at once from that rule, Lemma~\ref{lem:rel-ent-formula}, and the definitions above.

\begin{lem}[Chain rule]\label{lem:chain}
For any factor maps $\pi$, $\phi$ and $\psi$ on $(X,\mu,T)$, we have
\begin{equation}\label{eq:chain}
\rmh(\phi\vee \psi,\mu,T\,|\,\pi) = \rmh(\phi,\mu,T\,|\,\pi) + \rmh(\psi,\mu,T\,|\,\phi\vee\pi).
\end{equation}
\qed
\end{lem}

\begin{cor}[Chain inequality]\label{cor:chain-ineq}
For any factor maps  $\pi$, $\phi$ and $\psi$, we have
\[\rmh(\psi,\mu,T\,|\,\phi\vee \pi) \geq \rmh(\psi,\mu,T\,|\,\pi) - \rmh(\phi,\mu,T\,|\,\pi).\]
\end{cor}

\begin{proof}
	This follows by re-arranging~\eqref{eq:chain} and using the monotonicity
	\[\rmh(\phi\vee \psi,\mu,T\,|\,\pi) \geq \rmh(\psi,\mu,T\,|\,\pi).\]
\end{proof}

The next lemma generalizes the well-known equality between the KS entropy of a process and the Shannon entropy of its present conditioned on its past (see, for instance,~\cite[Equation (12.5)]{Bil65}).  The method of proof is the same. That method also appears in the proof of a further generalization in Lemma~\ref{lem:rel-ent-over-periodic} below.

\begin{lem}\label{lem:rel-past}
For any process $\xi$ and factor map $\pi$, we have
\[\rmh(\xi,\mu,T\,|\,\pi) = \rmH(\xi_0\,|\,\pi\vee \xi_{(-\infty;0)}).\]
\qed
\end{lem}

\subsection{The use of a periodic set}

At a few points later we need a slightly less standard variation on the calculations above.  To formulate it, we say that a measurable subset $F \subseteq X$ is \textbf{$m$-periodic under $T$ modulo $\mu$} for some $m \in \bbN$ if (i) the sets $F$, $T^{-1}F$, \dots, $T^{-(m-1)}F$ form a partition of $X$ modulo $\mu$-negligible sets and (ii) $T^{-m}F = F$ modulo $\mu$.

\begin{lem}\label{lem:rel-ent-over-periodic}
If $\pi$ is a factor map of $(X,\mu,T)$, $\xi$ is a process, and $F$ is $\pi$-measurable and $m$-periodic under $T$ modulo $\mu$, then
	\[\rmH(\xi_{[0;m)}\,|\,\pi;F) \geq \rmH(\xi_{[0;m)}\,|\,\pi\vee \xi_{(-\infty;0)};F)=  \rmh(\xi,\mu,T\,|\,\pi)\cdot m.\]
\end{lem}

(Recall~\eqref{eq:ent-on-event} for the notation $\rmH(\,\cdot\,|\,\cdot\,;F)$.)

\begin{proof}
The inequality on the left follows from the monotonicity of $\rmH$ under conditioning, so we focus on the equality on the right.  In case $m=1$, this is precisely the statement of Lemma~\ref{lem:rel-past}.  To handle the general case we simply mimic the usual proof of that special case.
	
	First, for any $k \in \bbN$, the chain rule for Shannon entropy gives
	\begin{equation}\label{eq:m-and-km}
		\rmH(\xi_{[0;km)}\,|\,\pi;F) = \sum_{i=0}^{k-1}\rmH(\xi_{[im;(i+1)m)}\,|\,\pi \vee \xi_{[0;im)};F).
	\end{equation}
For each $i$ in this sum we have $T^{-im}F = F$ modulo $\mu$, because $F$ is $m$-periodic.  Since the $\s$-algebra generated by $\pi$ is globally $T$-invariant, the conditional distribution of $\xi_{[im;(i+1)m)}$ given $\pi\vee \xi_{[0;im)}$ is the same as the conditional distribution of $\xi_{[0;m)}$ given $\pi\vee \xi_{[-im;0)}$, up to shifting by $T$.  Therefore the right-hand side of~\eqref{eq:m-and-km} is equal to
\[\sum_{i=0}^{k-1}\rmH(\xi_{[0;m)}\,|\,\pi\vee \xi_{[-im;0)};F).\]
By the usual appeal to the martingale convergence theorem (see, for instance,~\cite[Theorem 12.1]{Bil65}), this sum is equal to
\begin{equation}\label{eq:m-and-km-lim}
k\cdot \rmH(\xi_{[0;m)}\,|\,\pi\vee \xi_{(-\infty;0)};F) + o(k) \quad \hbox{as}\ k\to\infty.
\end{equation}
	
	On the other hand, consider some $t \in \{1,\dots,m-1\}$.  Then, reasoning again from the $T$-invariance of the $\s$-algebra generated by $\pi$, we have
	\begin{multline*}
		\rmH(\xi_{[0;km)}\,|\,\pi;T^{-t}F) = \rmH(\xi_{[t;t+km)}\,|\,\pi;F) \\ = \rmH(\xi_{[0;km)}\,|\,\pi;F) + O(t\log|A|) = \rmH(\xi_{[0;km)}\,|\,\pi;F) + O(m\log|A|),
	\end{multline*}
	since each of the observables $\xi_{[t;t+km)}$ and $\xi_{[0;km)}$ determines all but $t$ coordinates of the other.  Therefore
	\begin{align*}
		\rmH(\xi_{[0;km)}\,|\,\pi) &= \frac{1}{m}\sum_{t=0}^{m-1}\rmH(\xi_{[0;km)}\,|\,\pi;T^{-t}F)\\
		&= \rmH(\xi_{[0;km)}\,|\,\pi;F) + O(m\log |A|),
	\end{align*}
where the first equality holds because the partition $F$, $T^{-1}F$, \dots, $T^{-(m-1)}F$ is $\pi$-measurable and all these sets have measure $1/m$. Combining this with~\eqref{eq:m-and-km} and~\eqref{eq:m-and-km-lim}, dividing by $k$, and then letting $k\to\infty$, we obtain
\begin{multline*}
\rmH(\xi_{[0;m)}\,|\,\pi\vee \xi_{(-\infty;0)};F) = 
\lim_{k\to\infty} \frac{1}{k}\rmH(\xi_{[0;km)}\,|\,\pi;F) \\= \lim_{k\to\infty}\frac{1}{k} \rmH(\xi_{[0;km)}\,|\,\pi) = \rmh(\xi,\mu,T\,|\,\pi)\cdot m.
\end{multline*}
\nopagebreak\end{proof}

\begin{rmk}
Lemma~\ref{lem:rel-ent-over-periodic} is really a relative-entropy version of Abramov's formula for the entropy of an induced transformation~\cite{Abr59}, in the special case where we induce on a periodic set.
\end{rmk}

\section{Relative finite determination}\label{sec:version-of-Orn}

\subsection{Extensions, joinings, and relative $\ol{\d}$-distance}\label{subs:exts-etc}

Before we define relative finite determination, we need some more notation in the setting of finite-alphabet processes.  Let $A$ and $B$ be finite sets, and let $\b$ be the canonical process from $B^\bbZ\times A^\bbZ$ to $B^\bbZ$.  Let $\nu$ be a shift-invariant probability measure on $B^\bbZ$, and let $\rm{Ext}(\nu,A)$ be the set of all shift-invariant probability measures $\l$ on $B^\bbZ\times A^\bbZ$ which satisfy $\b_\ast\l = \nu$.

Given $\l,\theta\in\rm{Ext}(\nu,A)$, let
\[\rm{Join}(\l,\theta\,|\,\b)\]
be the set of all invariant probability measures $\g$ on $B^\bbZ \times A^\bbZ\times A^\bbZ$ which satisfy
\[\g\{(b,a,a'):\ (b,a) \in U\} = \l (U) \quad \hbox{and} \quad \g\{(b,a,a'):\ (b,a') \in U\} = \theta (U)\]
for all measurable $U \subseteq B^\bbZ \times A^\bbZ$. Members of $\rm{Join}(\l,\theta\,|\,\b)$ are called \textbf{relative joinings of $\l$ and $\theta$ over $\b$}.

The set $\rm{Join}(\l,\theta\,|\,\b)$ always contains the relative product of $\l$ and $\theta$ over the map $\b$, so this set is never empty. If $\l$ and $\theta$ are both ergodic, then $\rm{Join}(\l,\theta\,|\,\b)$ also contains all ergodic components of that relative product, so in particular it has some ergodic members.

\begin{dfn}\label{dfn:rel-dbar}
	In the setting above, the \textbf{relative $\ol{\d}$-distance between $\l$ and $\theta$ over $\b$} is the quantity
	\[\ol{\d}(\l,\theta\,|\,\b) = \inf\Big\{\g\{(b,a,a'):\ a_0 \neq a'_0\}:\ \g \in \rm{Join}(\l,\theta\,|\,\b)\Big\}.\]
\end{dfn}

A standard vague-compactness argument in the space $\rm{Join}(\l,\theta\,|\,\b)$ shows that this infimum is achieved by some $\g$.

In order to estimate $\ol{\d}(\l,\theta\,|\,\b)$, it is helpful to relate it to transportation distances and the finite-dimensional distributions $\l_{[0;n)}$ and $\theta_{[0;n)}$.  We may do this in terms of the associated block kernels from Definition~\ref{dfn:block-kernel}.

\begin{lem}\label{lem:dbar-and-dbar}
For any $\l,\theta \in \rm{Ext}(\nu,A)$ we have
\[\ol{\d}(\l,\theta\,|\,\b) = \lim_{n\to\infty} \int \ol{d_n}\Big(\l_{[0;n)}^\bl(\,\cdot\,|\,b_{[0;n)}),\ \theta_{[0;n)}^\bl(\,\cdot\,|\,b_{[0;n)})\Big)\ \nu(\d b),\]
\end{lem}

\begin{proof}
This is a relative version of a well-known formula for Ornstein's original $\ol{\d}$-metric: see, for instance,~\cite[Theorem I.9.7]{Shi96}.  The existence of the right-hand limit is part of the conclusion to be proved.

We see easily that $\ol{\d}(\l,\theta\,|\,\b)$ is at least the limit supremum of the right-hand side: if $\g$ is any element of $\rm{Join}(\l,\theta\,|\,\b)$ which achieves the value of $\ol{\d}(\l,\theta\,|\,\b)$, then its finite-dimensional distributions $\g_{[0;n)}$ give an upper bound for that limit supremum.  So now let $\delta$ be the limit infimum of the right-hand side, and let us show that $\ol{\d}(\l,\theta\,|\,\b) \leq \delta$. Suppose $\delta$ is the true limit along the subsequence indexed by $n_1 < n_2 < \dots$.

For each $n \in \bbN$ and $\bf{b} \in B^n$, select an optimal coupling of $\l^\bl_{[0;n)}(\,\cdot\,|\,\bf{b})$ and $\theta^\bl_{[0;n)}(\,\cdot\,|\,\bf{b})$. These couplings together define a kernel from $B^n$ to $A^n\times A^n$ for each $n$.  Hooking these kernels up to the measures $\nu_{[0;n)}$, we obtain a sequence of probability measures $\g_n$ on $B^n\times A^n\times A^n$ such that $\g_n$ couples $\l_{[0;n)}$ and $\theta_{[0;n)}$ over a common copy of $B^n$ and such that
\begin{multline}\label{eq:ave-prob-small}
\int_{B^{n_k}\times A^{n_k}\times A^{n_k}} d_{n_k}(\bf{a},\bf{a}')\ \g_{n_k}(\d\bf{b},\d\bf{a},\d\bf{a}') \\ = \frac{1}{n_k}\sum_{i=1}^{n_k} \g_{n_k}\{(\bf{b},\bf{a},\bf{a}'):\ a_i\neq a'_i\}\to \delta \quad \hbox{as}\ k \to\infty.
\end{multline}
For each $n$, let $\g'_n$ be any probability measure on $B^\bbZ \times A^\bbZ \times A^\bbZ$ whose projection to $B^{[0;n)}\times A^{[0;n)}\times A^{[0;n)}$ agrees with $\g_n$, and let $\g''_n$ be the following averaged measure on $B^\bbZ \times A^\bbZ \times A^\bbZ$:
\[\g''_n := \frac{1}{n}\sum_{t=0}^{n-1}(T_{B\times A\times A}^t)_\ast \g'_n.\]
Because of this averaging, the bound~\eqref{eq:ave-prob-small} implies that
\begin{equation}\label{eq:ave-prob-small2}
\g''_{n_k}\{(b,a,a'):\ a_0 \neq a'_0\}	= \frac{1}{n_k}\sum_{i=0}^{n_k-1}\g'_{n_k}\{(b,a,a'):\ a_i \neq a'_i\} \to \delta \quad \hbox{as}\ k \to\infty.
\end{equation}

Finally, let $\g''$ be any subsequential vague limit of the sequence $(\g''_{n_k})_{k=1}^\infty$. Then $\g''$ is invariant, because the averaged measures $\g''_n$ satisfy
\[\|\g''_n - (T_{B\times A\times A})_\ast\g''_n\| = O(1/n).\]
Taking the limit in~\eqref{eq:ave-prob-small2}, we have
\[\g''\{(b,a,a'):\ a_0 \neq a'_0\} = \delta.\]

It remains to show that $\g''$ is a coupling of $\l$ and $\theta$ over $\b$.  We show that the first projection of $\g''$ to $B^\bbZ\times A^\bbZ$ is $\l$; the argument for the second projection is the same.

Fix $\ell \in \bbN$ and two strings $\bf{b} \in B^\ell$ and $\bf{a} \in A^\ell$, and let
\[Y := \big\{(b,a) \in B^\bbZ\times A^\bbZ:\ b_{[0;\ell)} = \bf{b}\ \hbox{and}\ a_{[0;\ell)} = \bf{a}\big\}.\]
For each $n$, the definition of $\g_n''$ gives
\begin{equation}\label{eq:another-ave}
\g_n''(Y\times A^\bbZ) = \frac{1}{n}\sum_{t=0}^{n-1}\g_n'(T_{B\times A\times A}^{-t}(Y\times A^\bbZ)).
\end{equation}
For $0 \leq t\leq n - \ell$, the set
\[T_{B\times A\times A}^{-t}(Y\times A^\bbZ) = \big\{(b,a,a'):\ b_{[t;t+\ell)} = \bf{b}\ \hbox{and}\ a_{[t;t+\ell)} = \bf{a}\big\}\]
depends only on coordinates with indices in $[0;n)$, and so we have
\begin{align*}
	&\g_n'(T_{B\times A\times A}^{-t}(Y\times A^\bbZ))\\
	&= \l_{[0;n)}\big\{(b_0,\dots,b_{n-1},a_0,\dots,a_{n-1}):\\
	&\qquad \qquad \qquad \qquad (b_t,\dots,b_{t+\ell-1}) = \bf{b}\ \hbox{and}\ (a_t,\dots,a_{t+\ell-1}) = \bf{a}\big\}\\ &= \l(Y).
\end{align*}
Inserting this into~\eqref{eq:another-ave}, we obtain
\begin{align*}
	\frac{1}{n}\sum_{t=0}^{n-\ell}\l(Y) + O(\ell/n) &= \l(Y) + O(\ell/n).
\end{align*}
Taking $n = n_k$ here and letting $k \to\infty$, we conclude that $\g''(Y\times A^\bbZ) = \l(Y)$ for any such choice of $\ell$, $\bf{b}$ and $\bf{a}$.  Therefore the first projection of $\g''$ to $B^\bbZ\times A^\bbZ$ is equal to $\l$.
\end{proof}

It follows from Lemma~\ref{lem:dbar-and-dbar} that the function $\ol{\d}(\,\cdot\,,\,\cdot\,|\,\b)$ is indeed a metric on $\rm{Ext}(\nu,A)$, but we do not use this particular fact later in the paper.  See~\cite[Section I.9]{Shi96} for the analogous proof in the non-relative setting.

\subsection{Relative finite determination}

The relative version of the finitely determined property has been studied in various papers, particularly~\cite{Tho75},~\cite{Rah78} and~\cite{Kie84}.  Here we essentially follow the ideas of~\cite{Tho75}, except that we also define a quantitative version of this property, in which various small parameters are fixed rather than being subject to a universal or existential quantifier.  We do this because several arguments later in the paper require that we keep track of these parameters explicitly.

We define relative finite determination in two steps: for shift-invariant measures, and then for processes. Suppose first that $\nu \in \Pr(B^\bbZ)$ is shift-invariant, and let $\b$ be the canonical $B$-valued process on $B^\bbZ\times A^\bbZ$.

\begin{dfn}\label{dfn:rel-FD1}
	Let $\delta > 0$. A measure $\l \in \rm{Ext}(\nu,A)$ is \textbf{relatively $\delta$-finitely determined} (`\textbf{relatively $\delta$-FD}') \textbf{over $\b$} if there exist $\eps > 0$ and $n\in\bbN$ for which the following holds: if another measure $\theta\in \rm{Ext}(\nu,A)$ satisfies
	\begin{itemize}
		\item[(i)] $\rmh(\theta\,|\,\b) > \rmh(\l\,|\,\b) - \eps$ and
		\item[(ii)] $\|\theta_{[0;n)} - \l_{[0;n)}\| < \eps$,
	\end{itemize}
	then $\ol{\d}(\l,\theta\,|\,\b) < \delta$.
	
	The measure $\l$ is \textbf{relatively FD over} $\b$ if it is relatively $\delta$-FD over $\b$ for every $\delta > 0$.
\end{dfn}

Now consider a general system $(X,\mu,T)$ and a pair of processes $\pi:X \to B^\bbZ$ and $\xi:X \to A^\bbZ$.

\begin{dfn}\label{dfn:rel-FD2}
The process $\xi$ is \textbf{relatively $\delta$-FD} (resp. \textbf{relatively FD}) \textbf{over $\pi$} if the joint distribution
\[\l = (\pi \vee \xi)_\ast \mu\]
is relatively $\delta$-FD (resp. relatively FD) over $\b$.
\end{dfn}

The part of this definition which quantifies over all $\delta$ agrees with the definition in~\cite{Tho75,Kie84}.

The following result is the heart of Thouvenot's original work~\cite{Tho75}.

\begin{thm}\label{thm:FD-Bern}
If the system $(B^\bbZ\times A^\bbZ,\l,T_{B\times A})$ is ergodic and $\l$ is relatively FD over $\b$, then $(B^\bbZ\times A^\bbZ,\l,T_{B\times A})$ is relatively Bernoulli over $\b$. \qed
\end{thm}

We use this theorem as a `black box' in Part III.  It is the most significant result that we cite from outside this paper.

\begin{lem}\label{lem:FD-and-FD}
Let $(X,\mu,T)$ be a system, let
\[\xi:X \to A^\bbZ, \quad \pi:X \to B^\bbZ \quad \hbox{and} \quad \pi':X \to (B')^\bbZ\]
be three processes, and assume that $\pi$ and $\pi'$ generate the same factor of $(X,\mu,T)$ modulo $\mu$-neglible sets.  If $\xi$ is relatively $\delta$-FD over $\pi$, then it is also relatively $\delta$-FD over $\pi'$.
\end{lem}

\begin{proof}
	Since $\pi$ and $\pi'$ generate the same $\s$-subalgebra of $\B_X$ modulo negligible sets, and since our measurable spaces are all standard, there is a commutative diagram of factor maps
\begin{center}
	$\phantom{i}$\xymatrix{ & (X,\mu,T) \ar_\pi[dl]\ar^{\pi'}[dr] \\
		(B^\bbZ,\nu,T_B) & & ((B')^\bbZ,\nu',T_{B'}), \ar^\phi[ll]}
\end{center}
	where $\phi$ is an isomorphism of systems.  (See, for instance, the subsection on isomorphism vs. conjugacy in~\cite[Section 5]{Bil65}.)

Let
\[\l := (\pi\vee \xi)_\ast \mu \quad \hbox{and} \quad \l' := (\pi' \vee \xi)_\ast \mu.\]
Let $\b$ be the coordinate projection $B^\bbZ\times A^\bbZ \to B^\bbZ$, and define $\b'$ analogously.  Let $\a$ be the canonical process on $A^\bbZ$, so that $\a_F$ is the coordinate projection $A^\bbZ\to A^F$ for any $F\subseteq \bbZ$.

For the given value of $\delta$ and for the processes $\xi$ and $\pi$, let $\eps > 0$ and $n \in \bbN$ be the values promised by Definition~\ref{dfn:rel-FD2}.

Let $\eps' := \eps/3$.  Since $\phi_{[0;n)}:(B')^\bbZ\to B^{[0;n)}$ is measurable, there exist $m \in \bbN$, a map $\Phi:(B')^\bbZ \to B^{[0;n)}$ that depends only on coordinates in $[-m;n+m)$, and a measurable subset $U \subseteq (B')^\bbZ$ with $\nu'(U) > 1 - \eps/6$ such that
\begin{equation}\label{eq:phi-finitary-approx}
\phi_{[0;n)}|U = \Phi|U.
\end{equation}
Let $n' := n + 2m$, and let $U^\rm{c} := (B')^\bbZ\setminus U$.

Now suppose that $\theta' \in \rm{Ext}(\nu',A)$ satisfies
\begin{itemize}
	\item[(i$'$)] $\rmh(\theta'\,|\,\b') > \rmh(\l'\,|\,\b') - \eps'$ and
	\item[(ii$'$)] $\|\theta'_{[0;n')} - \l_{[0;n')}'\| < \eps'$, hence also
	\begin{equation}\label{eq:TV-xi}
	\|\theta'_{[-m;n+m)} - \l'_{[-m;n+m)}\| < \eps'
	\end{equation}
by shift-invariance. 
\end{itemize}
Let $\theta := (\phi \times \a)_\ast \theta' \in \Pr(B^\bbZ\times A^\bbZ)$.  We shall deduce from (i$'$) and (ii$'$) that $\theta$ and $\l$ satisfy the anologous conditions (i) and (ii) in Definition~\ref{dfn:rel-FD1}.

Towards (i), the chain rule and the Kolmogorov--Sinai theorem give that
\[\rmh(\l\,|\,\b) = \rmh(\l) - \rmh(\nu) = \rmh(\l') - \rmh(\nu') = \rmh(\l'\,|\,\b'),\]
and similarly $\rmh(\theta\,|\,\b) = \rmh(\theta'\,|\,\b')$.  Therefore, by (i$'$),
\[\rmh(\theta\,|\,\b) > \rmh(\l\,|\,\b) - \eps' > \rmh(\l\,|\,\b) - \eps.\]

To verify (ii), consider the estimates
\begin{align}\label{eq:theta-lambda1}
\|\theta_{[0;n)} - \l_{[0;n)}\| &= \big\|(\phi_{[0;n)}\times \a_{[0;n)})_\ast (\theta' - \l')\big\| \nonumber\\  &\leq \|(\phi_{[0;n)}\times \a_{[0;n)})_\ast(1_{U^{\rm{c}}\times A^\bbZ}\cdot \theta')\| \nonumber\\
&\quad + \big\|(\phi_{[0;n)}\times \a_{[0;n)})_\ast(1_{U\times A^\bbZ}\cdot (\theta' - \l'))\big\| \nonumber\\
&\quad + \|(\phi_{[0;n)}\times \a_{[0;n)})_\ast(1_{U^{\rm{c}}\times A^\bbZ}\cdot \l')\|.
\end{align}
Since $\theta'(U\times A^\bbZ) = \l'(U\times A^\bbZ) = \nu'(U)$, the first and last terms here are both less than $\eps/6$.  On the other hand, by~\eqref{eq:phi-finitary-approx}, the middle term is equal to
\[\big\|(\Phi\times \a_{[0;n)})_\ast(1_{U\times A^\bbZ}\cdot (\theta' - \l'))\big\|,\]
which is at most
\begin{equation}\label{eq:theta-lambda2}
\big\|(\Phi\times \a_{[0;n)})_\ast(\theta' - \l')\big\| + 2\cdot (\eps/6)
\end{equation}
by the same reasoning that gave~\eqref{eq:theta-lambda1}.  Since $\Phi$ depends only on coordinates in $[-m;n+m)$, the first term of~\eqref{eq:theta-lambda2} is bounded by the left-hand side of~\eqref{eq:TV-xi}, hence is at most $\eps/3$.  Adding up these estimates, we obtain condition (ii):
\[\|\theta_{[0;n)} - \l_{[0;n)}\| < \eps.\]

Having shown conditions (i) and (ii) for $\theta$ and $\l$, our choice of $\eps$ and $n$ gives $\ol{\d}(\l,\theta\,|\,\b) < \delta$.  Let $\g$ be a triple joining that witnesses this inequality, and let 
\[\g':= (\phi^{-1}\times \a\times \a)_\ast\g.\]
This new triple joining witnesses that $\ol{\d}(\l',\theta'\,|\,\b') < \delta$.
\end{proof}

\section{Relative extremality}\label{sec:rel-Orn}

Extremality is one of the characterizations of Bernoullicity for a finite-state ergodic process.  It was introduced by Thouvenot in the mid 1970s, and began to circulate in unpublished notes by Feldman and in Antoine Lamotte's PhD thesis, also unpublished.  The earliest published reference I know is~\cite{Rud78--twopt}, where Rudolph uses this characterization of Bernoullicity without proof.  Ornstein and Weiss introduce extremality and prove its equivalence to Bernoullicity in the general setting of amenable groups in~\cite[Section III.4]{OrnWei87}, and use it for their extension of Ornstein theory to that setting. Extremality is also defined and compared to other characterizations of Bernoullicity in~\cite[Definition 6.3]{Tho02}, and a complete account of its place in the theory for a single automorphism appears in~\cite[Chapter 5]{KalMcC10}.

In this paper we need a relative version of extremality for one process over another.  We define this as a natural sequel to Definition~\ref{dfn:RV-ext}.  The resulting definition is superficially a little different from Thouvenot's, but equivalent in all important respects.

Consider again a general system $(X,\mu,T)$ and a pair of processes $\xi: X \to A^\bbZ$ and $\pi: X \to B^\bbZ$.

\begin{dfn}\label{dfn:rel-ext-period}
	Let $r,\k > 0$.  The process $\xi$ is \textbf{relatively $(\k,r)$-extremal over $\pi$} if the map $\xi_{[0;n)}$ is $(\k n,r)$-extremal over $\pi_{[0;n)}$ on the probability space $(X,\mu)$ for all sufficiently large $n$.
\end{dfn}

This definition was the motivation for our work in Section~\ref{sec:ext}.  Unpacking Definition~\ref{dfn:RV-ext}, we can  write Definition~\ref{dfn:rel-ext-period} more explicitly as follows.  Let $\nu := \pi_\ast\mu$, let $\l:= (\pi\vee \xi)_\ast\mu$, and let $\l^\bl_{[0;n)}$ be the $(\xi,\pi,n)$-block kernel from Definition~\ref{dfn:block-kernel}.  The process $\xi$ is $(\k,r)$-extremal over $\pi$ if, for every sufficiently large $n$, there is a real-valued function $\bf{b}\mapsto r_{\bf{b}}$ on $B^n$ such that
\begin{itemize}
	\item the conditional measure $\l^\bl_{[0;n)}(\,\cdot\,|\,\bf{b})$ is a $(\k n,r_{\bf{b}})$-extremal measure on the metric space $(A^n,d_n)$ for each $\bf{b} \in B^n$, and
	\item we have
	\[\int r_{\bf{b}}\ \nu_{[0;n)}(\d\bf{b}) \leq r.\]
	\end{itemize}

Extremality is the notion that forms the key link between measure concentration as studied in Part I and relative Bernoullicity.  As such, it is the backbone of this whole paper.

The next lemma is a dynamical version of Corollary~\ref{cor:refinement-still-ext}.  It is used during the proof of Theorem A in Subsection~\ref{subs:AA}.  During that proof, we need to construct a new observable with respect to which another is fairly extremal, and then enlarge that new observable a little further without losing too much extremality.

\begin{lem}\label{lem:ext-robust}
Let $\xi$ and $\pi$ be processes on $(X,\mu,T)$ and suppose that $\xi$ is relatively $(\k,r)$-extremal over $\pi$.  Let $\pi'$ be another process such that $\pi'_0$ refines $\pi_0$, and assume that
	\begin{equation}\label{eq:ents-close}
	\rmh(\xi,\mu,T\,|\,\pi) - \rmh(\xi,\mu,T\,|\,\pi') \leq \k r.
	\end{equation}
	Then $\xi$ is relatively $(\k,6r)$-extremal over $\pi'$.
	\end{lem}

\begin{proof}
Let $n$ be large enough that
\begin{itemize}
\item $\xi_{[0;n)}$ is $(\k n,r)$-extremal over $\pi_{[0;n)}$, and
\item $\rmH(\xi_{[0;n)}\,|\,\pi_{[0;n)}) \leq \big(\rmh(\xi,\mu,T\,|\,\pi) + \k r/2\big)\cdot n$.
\end{itemize}
Then Lemma~\ref{lem:rel-ent-formula} and the assumed bound~\eqref{eq:ents-close} give
\begin{multline*}
\rmH(\xi_{[0;n)}\,|\,\pi_{[0;n)}) - \rmH(\xi_{[0;n)}\,|\,\pi'_{[0;n)}) \\ \leq \big(\rmh(\xi,\mu,T\,|\,\pi) - \rmh(\xi,\mu,T\,|\,\pi') + \k r/2\big)\cdot n \leq 3\k r n/2.
\end{multline*}
Now apply Corollary~\ref{cor:refinement-still-ext} to $\xi_{[0;n)}$, $\pi_{[0;n)}$ and $\pi'_{[0;n)}$. It gives that $\xi_{[0;n)}$ is extremal over $\pi'_{[0;n)}$ with parameters
\[\Big(\k n,\frac{2\cdot (3\k r n/2)}{\k n} + 3r\Big) = (\k n,6r).\]
\end{proof}

\section{From extremality to finite determination}\label{sec:rel-Orn-proofs}

Consider again a system $(X,\mu,T)$ and processes $\xi:X\to A^\bbZ$ and $\pi:X\to B^\bbZ$. The main result of this section is that extremality implies finite determination, with some explicit dependence between the quantitative versions of these properties.

\begin{prop}\label{prop:ABI-FD}
	Let $r > 0$. If $\xi$ is relatively $(\k,r)$-extremal over $\pi$ for some $\k > 0$, then $\xi$ is relatively $(7r)$-FD over $\pi$.
\end{prop}

The constant $7$ that appears here is convenient, but certainly not optimal.

\subsection{Comparison with a product of block kernels}

Let $\l$ be a measure on $B^\bbZ\times A^\bbZ$, and let $\nu$ be its marginal on $B^\bbZ$.  Let $\a$ and $\b$ be the canonical $A$- and $B$-valued processes on $B^\bbZ \times A^\bbZ$, respectively.  In this subsection we write $\rmH$ for $\rmH_\l$.

The key to Proposition~\ref{prop:ABI-FD} is a comparison between the block kernel $\l^\bl_{[0;kn)}$ and the product of $k$ copies of $\l_{[0;n)}^\bl$.  This is given in Proposition~\ref{prop:ext-ABI}. This comparison is also used for another purpose later in the paper: see the proof of Lemma~\ref{lem:more-ext1}.  For that later application, we need a little extra generality: we can assume that $\l$ is invariant only under the $n$-fold shift $T_{B\times A}^n$, not necessarily under the single shift $T_{B\times A}$.  We can still define the block kernels $\l^\bl_{[0,kn)}$ as in Definition~\ref{dfn:block-kernel}.  We return to the setting of true shift-invariance in Subsection~\ref{subs:ABI-FD}, where we apply Proposition~\ref{prop:ext-ABI} with $\l := (\pi\vee \xi)_\ast\mu$ to prove Proposition~\ref{prop:ABI-FD}.

So now assume that $\l$ is $T_{B\times A}^n$-invariant, but not necessarily $T_{B\times A}$-invariant.

\begin{lem}\label{lem:rel-n-past}
	We have
	\[\rmH(\a_{[0;n)}\,|\,\b\vee \a_{(-\infty;0)}) = \rmh(\a_{[0;n)},\l,T_{B\times A}^n\,|\,\b).\]
	\end{lem}

\begin{proof}
This just requires the right point of view.  The triple $(B^\bbZ\times A^\bbZ,\l,T_{B\times A}^n)$ is a measure-preserving system.  In this system, we may regard the map $\a_{[0;n)}$ as a single observable taking values in $A^n$.  Under the transformation $T_{B\times A}^n$, this observable generates the process
\[(\dots,\a_{[-2n;n)},\a_{[n;0)},\a_{[0;n)},\a_{[n;2n)},\dots).\]
This is just a copy of $\a$, divided into consecutive blocks of $n$ letters which we now regard as single letters in the enlarged alphabet $A^n$.  We may identify $\b$ with a process generated by $\b_{[0;n)}$ in the same way.

Now apply Lemma~\ref{lem:rel-past} to this system and pair of processes.
\end{proof}

\begin{prop}\label{prop:ext-ABI}
In the situation above, assume that
\begin{itemize}
	\item[(i)] (extremality) $\a_{[0;n)}$ is relatively $(\k n,r)$-extremal over $\b_{[0;n)}$ according to $\l$, and
	\item[(ii)] (conditional Shannon entropy is close to its Kolmogorov--Sinai limit)
	\[\rmH(\a_{[0;n)}\,|\,\b_{[0;n)}) < \rmh(\a_{[0;n)},\l,T_{B\times A}^n\,|\,\b) + \k r n.\]
\end{itemize}
	 Then
	\[\int \ol{d_{kn}}\Big(\l^\bl_{[0;kn)}(\,\cdot\,|\,b_{[0;kn)}),\ \l^\bl_{[0;n)}(\,\cdot\,|\,b_{[0;n)})\times \cdots \times \l^\bl_{[0;n)}(\,\cdot\,|\,b_{[(k-1)n;kn)})\Big)\ \nu(\d b) < 2r\]
	for every $k \in \bbN$.  Here $\ol{d_{kn}}$ is the transportation metric associated to $(A^{kn},d_{kn})$.
\end{prop}

If $\l$ is actually shift-invariant, then hypothesis (ii) of Proposition~\ref{prop:ext-ABI} may be rewritten as
\[\rmH(\a_{[0;n)}\,|\,\b_{[0;n)}) < \big(\rmh(\l,T_{B\times A}\,|\,\b) + \k r\big)\cdot n.\]
Therefore, in this special case, hypothesis (ii) holds for all sufficiently large $n$ by Lemma~\ref{lem:rel-ent-formula}.

\begin{proof}
We prove this by induction on $k$, but first we need to formulate a slightly more general inductive hypotheses.  Suppose that $k \in \bbN$ and that $S$ is any subset of $\bbZ$.  Then we write
\[B^S \to \Pr(A^{kn}):\bf{b} \mapsto \l^\bl_{[0;kn)\,|\,S}(\,\cdot\,|\,\bf{b})\]
for a version of the conditional distribution of $\a_{[0;kn)}$ given $\b_S$ under $\l$.  In this notation we have $\l^\bl_{[0;kn)} = \l^\bl_{[0;kn)\,|\,[0;kn)}$.  We show by induction on $k$ that
\begin{equation}\label{eq:extra-ext-ABI}
\int \ol{d_{kn}}\Big(\l^\bl_{[0;kn)\,|\,S}(\,\cdot\,|\,b_S),\ \l^\bl_{[0;n)}(\,\cdot\,|\,b_{[0;n)})\times \cdots \times \l^\bl_{[0;n)}(\,\cdot\,|\,b_{[(k-1)n;kn)})\Big)\ \nu(\d b) < 2r
\end{equation}
whenever $k \in \bbN$ and $S \supseteq [0;kn)$.  The case $S = [0;kn)$ gives Proposition~\ref{prop:ext-ABI}.

The proof of~\eqref{eq:extra-ext-ABI} when $k=1$ is a slightly degenerate version of the proof when $k > 1$, so we explain them together.

Thus, assume we already know~\eqref{eq:extra-ext-ABI} with $k-1$ in place of $k$ and whenever $S \supseteq [0;(k-1)n)$.  If $k=1$, then regard this inductive assumption as vacuous. To extend the desired conclusion to $k$ we use Lemma~\ref{lem:dist-from-prod-meas}.  Assume now that $S\supseteq [0;kn)$.

For each $\bf{b} \in B^S$, we have a probability measure $\l^\bl_{[0;kn)\,|\,S}(\,\cdot\,|\,\bf{b})$ on
	\[A^{kn} = \underbrace{A^n\times\cdots \times A^n}_k.\]
	The projection of this measure onto the first $k-1$ copies of $A^n$ is $\l^\bl_{[0;(k-1)n)\,|\,S}(\,\cdot\,|\,\bf{b})$. This is where we need the additional flexibility that comes from choosing $S$ separately: the projection of $\l^\bl_{[0;kn)}(\,\cdot\,|\,\bf{b})$ is $\l^\bl_{[0;(k-1)n)\,|\,[0;kn)}(\,\cdot\,|\,\bf{b})$, and in general this can be different from $\l^\bl_{[0;(k-1)n)}(\,\cdot\,|\,\bf{b}_{[0;(k-1)n)})$.

	We now disintegrate the measure $\l^\bl_{[0;kn)\,|\,S}(\,\cdot\,|\,\bf{b})$ further: for each $\bf{b} \in B^S$, let
	\[A^{(k - 1)n}\to \Pr(A^n):\bf{a}\mapsto \l'(\,\cdot\,|\,\bf{b},\bf{a})\]
	be a conditional distribution under $\l^\bl_{[0;kn)\,|\,S}(\,\cdot\,|\,\bf{b})$ of the last $n$ coordinates in $A^{k n}$, given that the first $(k -1)n$ coordinates agree with $\bf{a}$.  If $k = 1$, then simply ignore $\bf{a}$ and set $\l'(\,\cdot\,|\,\bf{b}) := \l^\bl_{[0;n)\,|\,S}(\,\cdot\,|\,\bf{b})$.
		
	In terms of this further disintegration, we may apply Lemma~\ref{lem:dist-from-prod-meas} to obtain
	\begin{align*}
		&\int \ol{d_{k n}}\Big(\l^\bl_{[0;k n)\,|\,S}(\,\cdot\,|\,b_S),\ \bigtimes_{j=0}^{k-1}\l^{\rm{block}}_{[0;n)}(\,\cdot\,|\,b_{[jn;(j+1)n)})\Big)\ \nu(\d b)\\
		&\leq \frac{k-1}{k}\int \ol{d_{(k-1)n}}\Big(\l^\bl_{[0;(k-1)n)\,|\,S}(\,\cdot\,|\,b_S),\ \bigtimes_{j=0}^{k-2}\l^{\rm{block}}_{[0;n)}(\,\cdot\,|\,b_{[jn;(j+1)n)})\Big)\ \nu(\d b)\\
		&\quad + \frac{1}{k}\iint \ol{d_n}\big(\l'(\,\cdot\,|\,b_S,\bf{a}),\ \l^{\rm{block}}_{[0;n)}(\,\cdot\,|\,b_{[(k-1)n;kn)})\big)\ \l^\bl_{[0;(k-1)n)\,|\,S}(\d\bf{a}\,|\,b_S)\ \nu(\d b).
	\end{align*}
	If $k=1$, then simply ignore the first term on the right of this inequality.
	
	By the inductive hypothesis, the first right-hand term here is less than ${2(k-1)r/k}$.  To finish the proof, we show that
	\[\iint \ol{d_n}\big(\l'(\,\cdot\,|\,b_S,\bf{a}),\ \l^{\rm{block}}_{[0;n)}(\,\cdot\,|\,b_{[(k-1)n;kn)})\big)\ \l^\bl_{[0;(k-1)n)\,|\,S}(\d\bf{a}\,|\,b_S)\ \nu(\d b) < 2r.\]
	
 The two kernels appearing inside the transportation distance here,
	\[\l'(\,\cdot\,|\,b_S,\bf{a}) \quad \hbox{and} \quad \l_{[0;n)}^{\rm{block}}(\,\cdot\,|\,b_{[(k-1)n;kn)}),\]
are conditional distributions for $\a_{[(k-1)n;kn)}$ under $\l$ given the $\s$-algebras
	\[\scrH \quad \hbox{generated by}\ \b_S \vee \a_{[0;(k-1)n)} \quad \hbox{and} \quad \G \quad \hbox{generated by}\ \b_{[(k-1)n;kn)},\]
	respectively.
	
	Clearly $\scrH \supseteq \G$. Therefore, by Lemma~\ref{lem:ext-for-finite}, the desired estimate follows from the $(\k n,r)$-extremality of $\a_{[0;n)}$ over $\b_{[0;n)}$, provided we show that
	\[\rmH(\a_{[(k-1)n;kn)}\,|\,\scrH) > \rmH(\a_{[(k-1)n;kn)}\,|\,\G) - \k rn.\]
	This, in turn, holds because of assumption (ii) in the statement of the current proposition:
	\begin{align*}
		\rmH(\a_{[(k-1)n;kn)}\,|\,\scrH) &= \rmH(\a_{[(k-1)n;kn)}\,|\,\b_S \vee \a_{[0;(k-1)n)})\\
		&\geq \rmH(\a_{[(k-1)n;kn)}\,|\,\b \vee \a_{(-\infty;(k-1)n)})\\
		&= \rmh(\a_{[0;n)},\l,T_{B\times A}^n\,|\,\b)\\
		&> \rmH(\a_{[(k-1)n;kn)}\,|\,\G) - \k rn,
	\end{align*}
	where the third line is obtained from Lemma~\ref{lem:rel-n-past} and invariance under $T_{B\times A}^n$, and the fourth is our appeal to assumption (ii).
	
	This continues the induction, and hence completes the proof.
\end{proof}

\begin{rmk}
Proposition~\ref{prop:ext-ABI} may be seen as a relative version of one of the standard implications of Ornstein theory: that extremality implies almost block independence.  We do not formulate a precise relative version of almost block independence in this paper, but this could easily be done along the same lines as the non-relative version: see, for instance,~\cite[Section IV.1]{Shi96} or~\cite[Definitions 473 and 476]{KalMcC10} (the latter reference calls this the `independent concatenations' property). Then the conclusion of Proposition~\ref{prop:ext-ABI} should imply that $\l$ is $(2r)$-almost block independent.  I expect this leads to another characterization of relative Bernoullicity for this $\l$, but we do not pursue this idea here.
	\end{rmk}

In the next subsection we use Proposition~\ref{prop:ext-ABI} to prove Proposition~\ref{prop:ABI-FD}.  At another point later we also need a different consequence of Proposition~\ref{prop:ext-ABI}:

\begin{cor}[Stretching property of extremality]\label{cor:stretch}
	Assume hypotheses (i) and (ii) from the statement of Proposition~\ref{prop:ext-ABI}.  Then the observable $\a_{[0;kn)}$ is relatively $(\k k n,5r)$-extremal over $\b_{[0;kn)}$ for all $k \in \bbN$. 
\end{cor}

\begin{proof}
Let us write a general element of $B^{kn}$ as $(\bf{b}_1,\dots,\bf{b}_k)$, where $\bf{b}_j \in B^n$ for each $j=1,2,\dots,k$.  A $k$-fold application of Lemma~\ref{lem:ext-prod} gives that the measure $\nu_{[0;kn)}$ and the pointwise-product kernel
	\[(\bf{b}_1,\dots,\bf{b}_k)\mapsto \l^\bl_{[0;n)}(\,\cdot\,|\,\bf{b}_1)\times \cdots \times \l^\bl_{[0;n)}(\,\cdot\,|\,\bf{b}_k)\]
	are $(\k kn,r)$-extremal.  Now combine this fact with Proposition~\ref{prop:ext-ABI} and Lemma~\ref{lem:ext-ker-dbar-stab}.
\end{proof}

\subsection{From extremality to finite determination}\label{subs:ABI-FD}

We return to the system $(X,\mu,T)$ and processes $\xi$ and $\pi$ that appear in Proposition~\ref{prop:ABI-FD}. Let $\nu := \pi_\ast\mu$ and $\l := (\pi\vee \xi)_\ast\mu$.  Let $\a$ and $\b$ be the canonical $A$- and $B$-valued processes on $B^\bbZ\times A^\bbZ$, respectively.

\begin{proof}[Proof of Proposition~\ref{prop:ABI-FD}]
	Choose $n \in \bbN$ so large that both of the following hold:
	\begin{itemize}
		\item[(a)] $\a_{[0;n)}$ is $(\k n,r)$-extremal over $\b_{[0;n)}$, and
		\item[(b)] $\rmH_\l(\a_{[0;n)}\,|\,\b_{[0;n)}) < (\rmh(\a,\l,T_{B\times A}\,|\,\b) + \k r)\cdot n = (\rmh(\l\,|\,\b) + \k r)\cdot n$.
	\end{itemize}
	The former holds for all sufficiently large $n$ by assumption. The latter holds for all sufficiently large $n$ by Lemma~\ref{lem:rel-ent-formula}.
	
By Proposition~\ref{prop:ext-ABI}, conditions (a) and (b) imply that
	\begin{equation}\label{eq:lambdas-close}
\int \ol{d_{kn}}\Big(\l^\bl_{[0;kn)}(\,\cdot\,|\,b_{[0;kn)}),\ \bigtimes_{j=0}^{k-1}\l^\bl_{[0;n)}(\,\cdot\,|\,b_{[jn;(j+1)n)})\Big)\ \nu(\d b) < 2r
	\end{equation}
for all $k\in\bbN$.
	
	Now suppose that $\theta \in \rm{Ext}(\nu,A)$ satisfies the conditions
	\begin{itemize}
		\item[(i)] $\rmh(\theta\,|\,\b) > \rmh(\l\,|\,\b) - \eps$, and
		\item[(ii)] $\|\theta_{[0;n)} - \l_{[0;n)}\| < \eps$
	\end{itemize}
	for some $\eps > 0$.
	
	Since $\l_{[0;n)}$ and $\theta_{[0;n)}$ have the same marginal on $B^n$ (namely, $\nu_{[0;n)}$), assumption (ii) implies that the following also holds:
	\begin{itemize}
		\item[(ii$'$)] (conditional distributions close in $\ol{d_n}$)
		\[\int \ol{d_n}\big(\l_{[0;n)}^\bl(\,\cdot\,|\,b_{[0;n)}),\theta_{[0;n)}^\bl(\,\cdot\,|\,b_{[0;n)})\big)\ \nu(\d b) < \eps.\]
	\end{itemize}
If $\eps$ is sufficiently small in terms of $r$, then we can apply Lemma~\ref{lem:ext-ker-dbar-stab} to conclude that
\begin{equation}\label{eq:i'}
\hbox{$\a_{[0;n)}$ is $(\k n,2r)$-extremal over $\b_{[0;n)}$ according to $\theta$,}
\end{equation}
in addition to the extremality according to $\l$.
	
	On the other hand, the quantity
	\[\rmH_\theta(\a_{[0;n)}\,|\,\b_{[0;n)}) = \rmH_\theta(\a_{[0;n)}\vee\b_{[0;n)}) - \rmH_\theta(\b_{[0;n)})\]
is a continuous function of the joint distribution $\theta_{[0;n)}$.  Therefore, provided we chose $\eps$ suffficiently small, assumption (ii) implies that
	\[\rmH_\theta(\a_{[0;n)}\,|\,\b_{[0;n)}) < \rmH_\l(\a_{[0;n)}\,|\,\b_{[0;n)}) + \k rn/2. \]
	Provided also that we chose $\eps < \k r/2$, condition (b) and assumption (i) give that this right-hand side is less than
	\[(\rmh(\l\,|\,\b) + \k r + \k r/2)\cdot n < (\rmh(\theta\,|\,\b) + 2\k r)\cdot n.\]
	Combining these inequalities, we obtain
\begin{equation}\label{eq:ii'}
\rmH_\theta(\a_{[0;n)}\,|\,\b_{[0;n)}) < (\rmh(\theta\,|\,\b) + 2\k r)\cdot n.
\end{equation}
	
Conclusions~\eqref{eq:i'} and~\eqref{eq:ii'} now provide the hypotheses for another application of Proposition~\ref{prop:ext-ABI}, this time to the system ${(B^\bbZ\times A^\bbZ,\theta,T_{B\times A})}$ and with $r$ replaced by $2r$.  We conclude that
	\begin{equation}\label{eq:thetas-close}
\int \ol{d_{kn}}\Big(\theta^\bl_{[0;kn)}(\,\cdot\,|\,b_{[0;kn)}),\ \bigtimes_{j=0}^{k-1}\theta^\bl_{[0;n)}(\,\cdot\,|\,b_{[jn;(j+1)n)})\Big)\ \nu(\d b) < 4r
\end{equation}
for all $k\in \bbN$.

Finally, for any $k\in \bbN$ and $b \in B^\bbZ$, a $(k-1)$-fold application of Lemma~\ref{lem:dist-from-prod-meas} gives
\begin{multline*}
\ol{d_{kn}}\Big(\bigtimes_{j=0}^{k-1}\l^\bl_{[0;n)}(\,\cdot\,|\,b_{[jn;(j+1)n)}),\ \bigtimes_{j=0}^{k-1}\theta^\bl_{[0;n)}(\,\cdot\,|\,b_{[jn;(j+1)n)})\Big) \\ \leq \frac{1}{k}\sum_{j=0}^{k-1}\ol{d_n}\Big(\l^\bl_{[0;n)}(\,\cdot\,|\,b_{[jn;(j+1)n)}),\ \theta^\bl_{[0;n)}(\,\cdot\,|\,b_{[jn;(j+1)n)})\Big).
\end{multline*}
Here we have used only the special case of Lemma~\ref{lem:dist-from-prod-meas} that compares two product measures.  Integrating this inequality with respect to $\nu(\d b)$, and recalling condition (ii$'$), it follows that
\begin{equation}\label{eq:lambda-theta-prods}
	\int \ol{d_{kn}}\Big(\bigtimes_{j=0}^{k-1}\l^\bl_{[0;n)}(\,\cdot\,|\,b_{[jn;(j+1)n)}),\ \bigtimes_{j=0}^{k-1}\theta^\bl_{[0;n)}(\,\cdot\,|\,b_{[jn;(j+1)n)})\Big)\ \nu(\d b) < \eps,
\end{equation}
which we now assume is less than $r$.  Combined with~\eqref{eq:lambdas-close},~\eqref{eq:thetas-close}, and the triangle inequality for $\ol{d_{kn}}$, we obtain that
\[	\int \ol{d_{kn}}\Big(\l_{[0;kn)}^\bl(\,\cdot\,|\,b_{[0;kn)}),\ \theta_{[0;kn)}^\bl(\,\cdot\,|\,b_{[0;kn)})\Big)\ \nu(\d b) < 7r.\]
Letting $k\to\infty$, Lemma~\ref{lem:dbar-and-dbar} completes the proof.
\end{proof}

\part{THE WEAK PINSKER PROPERTY}

\section{A quantitative step towards Theorem A}\label{sec:preA}

In this section, we consider a factor map $\pi:(X,\mu,T) \to (Y,\nu,S)$ of ergodic systems such that $(Y,\nu,S)$ satisfies the following special conditions:
\begin{quote}
	The entropy $\rmh(\nu,S)$ is finite, and $S$ has an $N$-periodic set modulo $\nu$ for every $N \in \bbN$.
\end{quote}
We refer to this as \textbf{assumption (A)}.  There certainly exist atomless and ergodic automorphisms that satisfy these conditions. For instance, let
	\[Y := \prod_{N\geq 1}(\bbZ/N\bbZ),\]
let $S$ be the rotation of this compact group by the element $(1,1,\dots)$, and let $\nu$ be any $S$-invariant ergodic measure.

Since $\rmh(\nu,S)$ is finite, Krieger's generator theorem~\cite{Kri70} tells us that $(Y,\nu,S)$ is isomorphic to a shift-system with a finite state space.  Since the conclusion of Theorem A depends on $(Y,\nu,S)$ only up to isomorphism, we are free to replace $(Y,\nu,S)$ with that shift-system: that is, we may assume that $Y = B^\bbZ$ and $S = T_B$ for some finite set $B$.  Accordingly, $\pi$ is now a $B$-valued process on $(X,\mu,T)$. We retain these assumptions until the proof of Theorem A under assumption (A) is completed in Subsection~\ref{subs:AA}.

Now let $\xi:X\to A^\bbZ$ be another process in addition to $\pi$.

\begin{prop}\label{prop:basic-ext2}
	Suppose assumption (A) holds, and let $r > 0$ and $\eps > 0$.  Then there exist a new process $\phi$ on $(X,\mu,T)$ and a value $\k > 0$ such that
	\begin{itemize}
		\item[(a)] $\rmh(\phi,\mu,T) < \eps$, and
		\item[(b)] $\xi$ is relatively $(\k,r)$-extremal over $\pi\vee \phi$.
	\end{itemize}
\end{prop}

In the light of Proposition~\ref{prop:ABI-FD} and Theorem~\ref{thm:FD-Bern}, this should be seen as a quantitative step towards relative Bernoullicity for $\xi$ over $\pi \vee \phi$.  In Subsection~\ref{subs:AA}, we apply Proposition~\ref{prop:basic-ext2} repeatedly with smaller and smaller values of $r$ to prove Theorem A under assumption (A).  Then it only remains to use some orbit equivalence theory to deduce Theorem A in full.

Throughout this section we abbreviate $\rmH_\mu$ to $\rmH$ and $\rmh(\xi,\mu,T\,|\,\pi)$ to $\rmh(\xi\,|\,\pi)$, and similarly.  When the measure is omitted from the notation, the correct choice is always $\mu$.

\begin{rmk}
It is possible to prove Theorem A in a single step, without ever treating the special case of assumption (A), by replacing periodic sets with sufficiently good Rokhlin sets (see, for instance,~\cite[p71]{Hal60}).  However, the use of Rokhlin's lemma introduces another error tolerance, which then complicates several of the estimates in the proof.  Since we use orbit equivalences to extend Theorem A to general amenable groups anyway, there seems to be little value in trying to avoid assumption (A).  Of course, Rokhlin's lemma is still implicitly at work within the results that we cite from orbit equivalence theory.
\end{rmk}

\subsection{The construction of a new observable}\label{subs:preA1}

Before proving Proposition~\ref{prop:basic-ext2}, we formulate and prove a proposition that gives a new observable $\phi$ with some related but more technical properties.

\begin{prop}\label{prop:basic-ext}
Suppose assumption (A) holds, and let $r > 0$ and $\eps > 0$.  Then there exist a new process $\phi$ on $(X,\mu,T)$, a positive integer $N$, a value $\k > 0$, and a $\phi_0$-measurable and $N$-periodic set $F$ modulo $\mu$ such that:
\begin{enumerate}
	\item[(a)] $\rmh(\phi) < \eps$;
	\item[(b)] $\xi_{[0;N)}$ is relatively $(\k N,r)$-extremal over $(\pi\vee \phi)_{[0;N)}$ according to $\mu(\,\cdot\,|\,F)$;
	\item[(c)] $\rmH(\xi_{[0;N)}\,|\,(\pi\vee \phi)_{[0;N)};F) < \big(\rmh(\xi\,|\,\pi\vee \phi) + r\k\big)\cdot N$.
\end{enumerate}
\end{prop}

The proof of Proposition~\ref{prop:basic-ext} occupies this subsection.  It is the longest and trickiest part of the proof of Theorem A.  It includes the key point at which we apply Theorem C.

A warning is in order at this juncture.  In the statement of Theorem C we denote the alphabet by $A$.  However, in this subsection our application of Theorem C does not use the alphabet $A$ of the process $\xi$, but rather some large Cartesian power $A^\ell$. The integer $\ell$ is chosen in Step 1 of the proof below, based on the behaviour of the process $\xi$ (see choice (P2)).

The data $N$ and $F$ given by Proposition~\ref{prop:basic-ext} do not appear in Proposition~\ref{prop:basic-ext2}, but they play an essential auxiliary role in the construction of the new process $\phi$ and the proof that it has all the desired properties.  This is why we formulate and prove Proposition~\ref{prop:basic-ext} separately.

Conclusions (b) and (c) of Proposition~\ref{prop:basic-ext} match hypotheses (i) and (ii) of Proposition~\ref{prop:ext-ABI}, but for the conditioned measure $\mu(\,\cdot\,|\,F)$ rather than $\mu$.  In the next subsection we show how true extremality of $\xi$ over $\pi\vee \phi$ can be deduced from Proposition~\ref{prop:basic-ext} using Corollary~\ref{cor:stretch}. 

We break the construction of $F$ and $\phi$ into several steps, and complete the proof of Proposition~\ref{prop:basic-ext} at the end of this subsection.

\subsection*{Step 1: choice of parameters}

By reducing $\eps$ if necessary, we may assume it is less than $\log 2$. We now choose several more auxiliary parameters as follows.
\begin{itemize}
\item[(P1)] (Constants provided by Theorem C.) Let $c_{\rm{C}}$ and $\k_{\rm{C}}$ be the constants provided by Theorem C when both the input error tolerances in that theorem are set equal to $r$. Now choose some positive $\delta < \min\{\eps/c_{\rm{C}},r\}$.
\item[(P2)] (First time-scale.) Choose $\ell \in \bbN$ so large that
\[\rmH(\xi_{[0;\ell)}\,|\,\pi_{[0;\ell)}) < (\rmh(\xi\,|\,\pi) + \delta^2/2)\cdot \ell.\]
This choice of $\ell$ is possible by Lemma~\ref{lem:rel-ent-formula}.  Set $\k := \k_{\rm{C}}/\ell$.
\item[(P3)] (Second time-scale.) Lastly, choose $n \in\bbN$ so large that all of the following hold:
\begin{itemize}
\item[(P3.1)] (entropy estimate)
\[\rmH(\xi_{[0;n\ell)}\,|\,\pi_{[0;n\ell)}) < (\rmh(\xi\,|\,\pi) + \k r/2)\cdot n\cdot \ell;\]
\item[(P3.2)] (requirement for Theorem C) Theorem C applies to the $n$-fold power of the alphabet $A^\ell$ when both of the input error tolerances equal $r$;
\item[(P3.3)] (another auxiliary estimate needed below)
\[c_{\rm{C}}\exp(c_{\rm{C}}\delta n\ell) < \exp(\eps n\ell).\]
\end{itemize}
The fact that (P3.1) holds for all sufficiently large $n$ is a consequence of Lemma~\ref{lem:rel-ent-formula}.  The fact that (P3.2) can be satisfied is precisely Theorem C.

Let $N:= n\ell$.
\end{itemize}

Each of these choices is needed for a particular step in the argument below, and is explained when we reach that step.  The relation between $\ell$ and $n$ is very important: it allows us to regard certain conditional distributions on $A^{n\ell}$ as measures on the $n$-fold product of the alphabet $A^\ell$ for which we have control over the total correlation.  This control is established in Lemma~\ref{lem:good-set} below.

\subsection*{Step 2: introducing a periodic set}

Having chosen $N$, it is time to invoke assumption (A).  It provides a $\pi$-measurable set $F_0 \subseteq X$ which is $N$-periodic under $T$ modulo $\mu$.  Modulo negligible sets, we can now define a $\pi$-measurable factor map $\tau:X\to \bbZ/N\bbZ$ to a finite rotation by requiring that
\[\tau(x) = j \!\!\mod N\quad \hbox{when} \quad T^{-j}x \in F_0.\]
Let $F_j := \{\tau = j\!\!\mod N\}$ for each $j\in \bbZ$, so this equals $T^j F_0$ modulo $\mu$, and let $\ol{F}_j := \{\tau = j \!\!\mod \ell\}$, which agrees with
\[F_j \cup T^{-\ell}F_j\cup \cup\cdots \cup T^{-(n-1)\ell}F_j\]
modulo $\mu$. Each $F_j$ is $N$-periodic under $T$ modulo $\mu$, each $\ol{F}_j$ is $\ell$-periodic under $T$ modulo $\mu$, and $\ol{F}_{j+\ell} = \ol{F}_j$ modulo $\mu$ for every $j \in \bbZ$.

\begin{lem}\label{lem:choice-of-F}
	There exists $j \in \{0,1,\dots,N-1\}$ satisfying both
	\begin{equation}\tag{\emph{a}}
		\rmH(\xi_{[0;N)}\,|\,\pi_{[0;N)};F_j) < (\rmh(\xi\,|\,\pi) + \k r)\cdot N
		\end{equation}
	and
	\begin{equation}\tag{\emph{b}}
\frac{1}{n}\sum_{i=0}^{n-1}\rmH(\xi_{[0;\ell)}\,|\,\pi_{[0;\ell)};T^{-i\ell}F_j) < (\rmh(\xi\,|\,\pi) + \delta^2)\cdot \ell.
\end{equation}
\end{lem}

\begin{proof}
Up to a negligible set, the sets $F_j$ for $0 \leq j \leq N-1$ agree with the partition of $X$ generated by $\tau$.  Therefore we can compute $\rmH(\xi_{[0;N)}\,|\,\tau\vee \pi_{[0;N)})$ by first conditioning onto each of the sets $F_j$, and then refining that partition by $\pi_{[0;N)}$.  This gives
	\begin{multline*}
		\frac{1}{N}\sum_{j =0}^{N-1}\rmH(\xi_{[0;N)}\,|\,\pi_{[0;N)};F_j) = \rmH(\xi_{[0;N)}\,|\,\tau\vee \pi_{[0;N)}) \\ \leq \rmH(\xi_{[0;N)}\,|\,\pi_{[0;N)}) < (\rmh(\xi\,|\,\pi) + \k r/2)\cdot N,
	\end{multline*}
	by assumption (P3.1). On the other hand, each of the sets $F_j$ is $\pi$-measurable and $N$-periodic under $T$, so the terms in the left-hand average here all satisfy
	\[\rmH(\xi_{[0;N)}\,|\,\pi_{[0;N)};F_j) \geq \rmH(\xi_{[0;N)}\,|\,\pi;F_j) \geq \rmh(\xi\,|\,\pi)\cdot N,\]
	using Lemma~\ref{lem:rel-ent-over-periodic} for the second inequality.  Therefore, by Markov's inequality, inequality (a) is satisfied for more than half of the values $j \in \{0,1,\dots,N-1\}$.
	
Turning to inequality (b), we start by observing the following.  Since $\ol{F}_{j+\ell} = \ol{F}_j$ modulo $\mu$ and $N$ is a multiple of $\ell$, we have
\[\frac{1}{N}\sum_{j=0}^{N-1}\rmH(\xi_{[0;\ell)}\,|\,\pi_{[0;\ell)};\ol{F}_j) = \frac{1}{\ell}\sum_{j=0}^{\ell-1}\rmH(\xi_{[0;\ell)}\,|\,\pi_{[0;\ell)};\ol{F}_j).\]
This lets us re-apply the argument for inequality (a).  This time we use the $\ell$-periodic sets $\ol{F}_j$ in place of $F_j$, the maps $\xi_{[0;\ell)}$ and $\pi_{[0;\ell)}$, and assumption (P2) in place of (P3.1).  We deduce that more than half of the values $j\in \{0,1,\dots,N-1\}$ satisfy
\begin{equation}\label{eq:bound-like-a}
\rmH(\xi_{[0;\ell)}\,|\,\pi_{[0;\ell)};\ol{F}_j) < (\rmh(\xi\,|\,\pi) + \delta^2)\cdot \ell.
\end{equation}
On the other hand, for each $j$, the sets $T^{-i\ell}F_j$ for $0 \leq i \leq n-1$ are a partition of $\ol{F}_j$ into $n$ further subsets of equal measure (again up to a negligible set), and this partition of $\ol{F}_j$ is generated by the restriction $\tau|\ol{F}_j$. From this, the chain rule, and monotonicity under conditioning, we deduce that
\[	\frac{1}{n}\sum_{i=0}^{n-1}\rmH(\xi_{[0;\ell)}\,|\,\pi_{[0;\ell)};T^{-i\ell}F_j) = \rmH(\xi_{[0;\ell)}\,|\,\pi_{[0;\ell)}\vee \tau;\ol{F}_j) 
	\leq \rmH(\xi_{[0;\ell)}\,|\,\pi_{[0;\ell)};\ol{F}_j).\]
Therefore any value of $j$ which satisfies~\eqref{eq:bound-like-a} also satisfies inequality (b).

Combining the conclusions above, there is at least one $j$ for which both (a) and (b) hold.
\end{proof}

With $j$ as given by Lemma~\ref{lem:choice-of-F}, fix $F:= F_j$.

\subsection*{Step 3: some properties of the conditional marginals}

Set $\t{A} := A^\ell$, so that we may identify $A^{n\ell} = A^N$ with $\t{A}^n$. Now define a new kernel by
\[\t{\l}(\bf{a}\,|\,\bf{b}) := \mu(\xi_{[0;N)} = \bf{a}\,|\,\{\pi_{[0;N)} = \bf{b}\}\cap F) \quad (\bf{a} \in A^N = \t{A}^n,\ \bf{b} \in B^N).\]
This is an instance of an $N$-block kernel, as in Definition~\ref{dfn:block-kernel}, except for the extra conditioning on $F$.  At this point it is important that we regard this as a kernel from $B^N$ to $\t{A}^n$, rather than to $A^N$.  In addition, define the probability measure $\t{\nu}$ on $B^N$ by $\t{\nu}(\bf{b}) := \mu(\pi_{[0;N)} = \bf{b}\,|\,F)$.

\begin{lem}\label{lem:good-set}
	Let $W$ be the set of all $\bf{b} \in B^N$ which satisfy
	\[\t{\nu}(\bf{b}) > 0 \quad \hbox{and} \quad \TC(\,\t{\l}(\,\cdot\,|\,\bf{b})\,) \leq \delta N.\]
	Then $\t{\nu}(W) = \mu(\pi_{[0;N)} \in W\,|\,F) > 1-\delta$.
\end{lem}

\begin{proof}
	Consider a string $\bf{b} \in B^N$ for which $\t{\nu}(\bf{b})  > 0$.  Then $\t{\l}(\,\cdot\,|\,\bf{b})$ is the conditional distribution of $\xi_{[0;N)}$ given the event $\{\pi_{[0;N)} = \bf{b}\}\cap F$, where we regard $\xi_{[0;N)}$ as taking values in $\t{A}^n$.  Accordingly, the $n$ marginals of $\t{\l}(\,\cdot\,|\,\bf{b})$ on the space $\t{A}$ are the conditional distributions of the $\t{A}$-valued observables
	\[\xi_{[0;\ell)},\ \xi_{[\ell;2\ell)},\ \dots,\ \xi_{[(n-1)\ell;n\ell)}\]
	given the event $\{\pi_{[0;N)} = \bf{b}\}\cap F$.  Therefore, by the definition of total correlation,
	\begin{align*}
\TC\big(\,\t{\l}(\,\cdot\,|\,\bf{b})\,\big) &= \sum_{i=1}^n \rmH(\xi_{[(i-1)\ell;i\ell)}\,|\,\{\pi_{[0;N)} = \bf{b}\}\cap F) - \rmH(\xi_{[0;N)}\,|\,\{\pi_{[0;N)} = \bf{b}\}\cap F)\\
&= \sum_{i=1}^n \rmH_{\mu(\,\cdot\,|\,F)}(\xi_{[(i-1)\ell;i\ell)}\,|\,\{\pi_{[0;N)} = \bf{b}\}) - \rmH_{\mu(\,\cdot\,|\,F)}(\xi_{[0;N)}\,|\,\{\pi_{[0;N)} = \bf{b}\}).
\end{align*}
	
	Integrating this equality with respect to $\t{\nu}(\d\bf{b})$, we obtain:
	\begin{equation}\label{eq:int-TC}
	\int \TC\big(\,\t{\l}(\,\cdot\,|\,\bf{b})\,\big)\ \t{\nu}(\d \bf{b}) = \sum_{i=1}^n \rmH(\xi_{[(i-1)\ell;i\ell)}\,|\,\pi_{[0;N)};F) - \rmH(\xi_{[0;N)}\,|\,\pi_{[0;N)};F).
	\end{equation}
	By the monotonicity of Shannon entropy under conditioning, and then by the $T$-invariance of $\mu$, this is at most
	\begin{align}\label{eq:int-over-b}
	&\sum_{i=1}^n \rmH(\xi_{[(i-1)\ell;i\ell)}\,|\,\pi_{[(i-1)\ell;i\ell)};F) - \rmH(\xi_{[0;N)}\,|\,\pi_{[0;N)};F)\nonumber \\
	&= \sum_{i=0}^{n-1}\rmH(\xi_{[0;\ell)}\,|\,\pi_{[0;\ell)};T^{-i\ell}F) - \rmH(\xi_{[0;N)}\,|\,\pi_{[0;N)};F).
	\end{align}
By inequality (b) from Lemma~\ref{lem:choice-of-F}, the sum of positive terms here is less than
\[(\rmh(\xi\,|\,\pi) + \delta^2)\cdot \ell\cdot n = (\rmh(\xi\,|\,\pi) + \delta^2)\cdot N.\]
On the other hand, $\rmH(\xi_{[0;N)}\,|\,\pi_{[0;N)};F)$ is at least $\rmH(\xi_{[0;N)}\,|\,\pi;F)$ by the monotonicity of Shannon entropy under conditioning, and the latter is at least $\rmh(\xi\,|\,\pi)\cdot N$ by Lemma~\ref{lem:rel-ent-over-periodic}.  Therefore the whole of~\eqref{eq:int-over-b} is strictly less than
	\[N\cdot \big(\rmh(\xi\,|\,\pi) + \delta^2\big) - N\cdot \rmh(\xi\,|\,\pi) = \delta^2 N,\]
	and hence so is the integral on the left-hand side of~\eqref{eq:int-TC}. So $\t{\nu}(W) > 1 - \delta$ by Markov's inequality.
\end{proof}

\subsection*{Step 4: application of Theorem C}

For each $\bf{b} \in W$, we now apply Theorem C to the measure $\t{\l}(\,\cdot\,|\,\bf{b})$ with the choice of parameters in (P1) above.  For each $\bf{b} \in W$, that theorem provides a partition
\begin{equation}\label{eq:decomp-from-C}
\t{A}^n = U_{\bf{b},1} \cup \cdots \cup U_{\bf{b},m_{\bf{b}}}
	\end{equation}
with the following properties:
\begin{enumerate}
	\item[(a)] The number of cells satisfies $m_{\bf{b}} \leq c_\rm{C}\exp(c_\rm{C} \delta N)$, which is less than $\rme^{\eps N}$ by our choice in (P3.3).
	\item[(b)] The first cell satisfies $\t{\l}(U_{\bf{b},1}\,|\,\bf{b}) < r$.
	\item[(c)] The twice-conditioned measure $(\t{\l}(\,\cdot\,|\,\bf{b}))_{|U_{\bf{b},j}}$ is well-defined and satisfies T$(\k_\rm{C} n,r)$ for every $j=2,3,\dots,m_{\bf{b}}$, where the underlying metric is the normalized Hamming metric $d_n$ on the $n$-fold product $\t{A}^n$. This is greater than or equal to the normalized Hamming metric $d_N$ when we identify $\t{A}^n$ with the $N$-fold product $A^N$, so $(\t{\l}(\,\cdot\,|\,\bf{b}))_{|U_{\bf{b},j}}$ satisfies T$(\k_\rm{C} n,r)$ for that choice of metric also.  Recalling that $\k = \k_\rm{C}/\ell$, let us write this last inequality as T$(\k N,r)$.
\end{enumerate}
By including a few copies of the empty set, let us suppose that $m_{\bf{b}}$ equals a fixed value $m$ for each $\bf{b} \in W$, with the modification that conclusion (c) above holds only when $U_{\bf{b},j}\neq \emptyset$. Also, for $\bf{b} \in B^N\setminus W$, choose an arbitrary partition of $\t{A}^n$ into $m$ subsets, so that the notation in~\eqref{eq:decomp-from-C} still makes sense for those $\bf{b}$.

We now switch back from $\t{A}^n$ to $A^N$ in our discussion.

In the rest of the construction, we need the right-hand side of~\eqref{eq:decomp-from-C} to be indexed by strings over some alphabet, rather than by the integers $1$, \ldots, $m$.  Since we have chosen $\eps < \log 2$, we have $m \leq 2^N$.  So let $S\subseteq \{0,1\}^N$ be any subset of cardinality $m$, and re-write the partitions on the right-hand side of~\eqref{eq:decomp-from-C} so that the parts are indexed by the strings in $S$:
\begin{equation}\label{eq:decomp-from-C2}
A^N = \bigcup_{\bf{c} \in S}U_{\bf{b},\bf{c}}.
\end{equation}
Do this in such a way that some fixed element $\bf{c}_1 \in S$ has $U_{\bf{b},\bf{c}_1}$ equal to the `bad' cell $U_{\bf{b},1}$ from property (b) for each $\bf{b}$.

Define $\Psi:B^N\times A^N\to S$ by
\[\Psi(\bf{b},\bf{a}) = \bf{c} \quad \hbox{whenever}\ \bf{a} \in U_{\bf{b},\bf{c}},\]
and define $\Phi:F\to S$ by
\[\Phi(x) := \Psi\big(\pi_{[0;N)}(x),\xi_{[0;N)}(x)\big).\]

\begin{lem}\label{lem:little-bit-ext}
	The map $\xi_{[0;N)}$ is $(\k N,3r)$-extremal over $\pi_{[0;N)}\vee \Phi$ according to $\mu(\,\cdot\,|\,F)$.
	\end{lem}

\begin{proof}
The point is that the measures
\[\l_{\bf{b},\bf{c}} := (\t{\l}(\,\cdot\,|\,\bf{b}))_{|U_{\bf{b},\bf{c}}}\]
constitute a conditional distribution of $\xi_{[0;N)}$ given $\pi_{[0;N)}\vee \Phi$ according to $\mu(\,\cdot\,|\,F)$.  To see this, let $\bf{a} \in A^N$, $\bf{b} \in B^N$ and $\bf{c} \in S$, and suppose that the conditional probability
\[\frac{\mu\big(\xi_{[0;N)}\vee \pi_{[0;N)}\vee \Phi = (\bf{a},\bf{b},\bf{c})\,\big|\,F\big)}{\mu\big(\pi_{[0;N)}\vee \Phi = (\bf{b},\bf{c})\,\big|\,F\big)} = \frac{\mu\big(\{\xi_{[0;N)}\vee \pi_{[0;N)}\vee \Phi = (\bf{a},\bf{b},\bf{c})\}\cap  F\big)}{\mu\big(\{\pi_{[0;N)}\vee \Phi = (\bf{b},\bf{c})\}\cap F\big)}\]
has nonzero denominator. Then this fraction satisfies the following:
\begin{itemize}
	\item It is zero if $\bf{c} \neq \Psi(\bf{b},\bf{a})$, or equivalently if $\bf{a} \not\in U_{\bf{b},\bf{c}}$. In this case $\l_{\bf{b},\bf{c}}(\bf{a})$ is also zero.
\item If $\bf{c} = \Psi(\bf{b},\bf{a})$, then the fraction equals
\begin{multline*}
\frac{\mu\big(\{\pi_{[0;N)} = \bf{b}\}\cap \{\xi_{[0;N)} = \bf{a}\}\cap F\big)}{\mu\big(\{\pi_{[0;N)} = \bf{b}\}\cap \{\Psi(\bf{b},\xi_{[0;N)}) = \bf{c}\}\cap F\big)} \\
= \frac{\mu\big(\xi_{[0;N)} = \bf{a}\,\big|\,\{\pi_{[0;N)} = \bf{b}\}\cap F\big)}{\mu\big(\xi_{[0;N)}\in U_{\bf{b},\bf{c}}\,\big|\,\{\pi_{[0;N)} = \bf{b}\}\cap F\big)}
= \frac{\t{\l}(\bf{a}\,|\,\bf{b})}{\t{\l}(U_{\bf{b},\bf{c}}\,|\,\bf{b})} = \l_{\bf{b},\bf{c}}(\bf{a}).
\end{multline*}
\end{itemize}

Having identified this conditional distribution, we may use the following consequence of our appeal to Theorem C:
\begin{align*}
&\mu\big(\pi_{[0;N)}\vee \Phi \in \big\{(\bf{b},\bf{c}):\ \l_{\bf{b},\bf{c}}\ \hbox{satisfies T$(\k N,r)$}\big\}\,\big|\,F\big)\\
&\geq 1 - \mu\big(\pi_{[0;N)} \not\in W \ \hbox{or}\ \xi_{[0;N)} \in U_{\pi_{[0;N)},\bf{c}_1}\,\big|\,F\big)\\
&\geq 1 - \mu(\pi_{[0;N)}\not\in W\,|\,F) - \int_W \t{\l}(U_{\bf{b},\bf{c}_1}\,|\,\bf{b})\,\t{\nu}(\d\bf{b}).
\end{align*}
By Lemma~\ref{lem:good-set} and conclusion (b) obtained from Theorem C, this is at least $1 - \delta - r$, which is at least $1-2r$.  Now Lemma~\ref{lem:ker-ext-bad-set} completes the proof.
\end{proof}

\subsection*{Step 5: finishing the construction}

It remains to construct the new process $\phi$.  It is obtained from $\tau$ together with another new observable $\psi_0:X \to \{0,1\}$. The idea behind $\psi_0$ is that names should be written as length-$N$ blocks that are given by $\Phi$, where the blocks start at the return times to $F$.  In notation, we wish to define $\psi_0$ so that
\begin{equation}\label{eq:def-psi}
	\psi_{[0;N)}|F = \Phi.
\end{equation}
This is possible because $F$ is $N$-periodic --- indeed, this is precisely why we introduce an $N$-periodic set into the construction.  The formal definition that achieves~\eqref{eq:def-psi} is as follows.  Given $x \in X$, choose $t \in [0;N)$ so that $T^{-t}x \in F$, and now let $\psi_0(x)$ be the $(t+1)^\rm{th}$ letter of the string
\[\Phi(T^{-t}x).\]
(We need the $(t+1)^\rm{th}$ letter because we index $\{0,1\}^N$ using $\{1,\dots,N\}$, not $[0;N)$.)

Let $\psi:X\to \{0,1\}^\bbZ$ be the process generated by $\psi_0$. Now let $\phi:= \tau\vee \psi$, and observe that $\pi \vee \phi$ and $\pi \vee \psi$ generate the same factor of $(X,\mu,T)$.

\begin{proof}[Proof of Proposition~\ref{prop:basic-ext}]
First,~\eqref{eq:def-psi} and Lemma~\ref{lem:rel-ent-over-periodic} give that
\[\rmh(\phi) = \rmh(\psi) \leq \frac{1}{N}\rmH(\psi_{[0;N)}\,|\,F) = \frac{1}{N}\rmH(\Phi\,|\,F)  \leq \frac{\log |S|}{N} < \eps.\]
This proves (a).
	
Next, the definition~\eqref{eq:def-psi} has the consequence that the conditional distribution of $\xi_{[0;N)}$ given $\pi_{[0;N)}\vee \phi_{[0;N)}$ according to $\mu(\,\cdot\,|\,F)$ is the same as its conditional distribution given $\pi_{[0;N)}\vee \Phi$ according to $\mu(\,\cdot\,|\,F)$.  In particular, $\tau$ has no effect on this conditional distribution, because $\tau|F$ is constant.  Therefore $\xi_{[0;N)}$ is relatively $(\k N,3r)$-extremal over $(\pi\vee \phi)_{[0;N)}$ according to $\mu(\,\cdot\,|\,F)$, by Lemma~\ref{lem:little-bit-ext}.

Now we turn to (c).  On the event $F$, we have from~\eqref{eq:def-psi} that the map $\psi_{[0;N)}$ agrees with $\Phi$, which is a function of $\pi_{[0;N)} \vee \xi_{[0;N)}$. Therefore the chain rule for Shannon entropy gives
\begin{align*}
\rmH(\xi_{[0;N)}\,|\,\phi_{[0;N)}\vee \pi_{[0;N)}\,;\,F) &= \rmH(\xi_{[0;N)}\,|\,\psi_{[0;N)}\vee \pi_{[0;N)}\,;\,F) \\ &= \rmH(\xi_{[0;N)}\,|\,\Phi\vee \pi_{[0;N)}\,;\,F)\\ &= \rmH(\xi_{[0;N)}\,|\,\pi_{[0;N)}\,;\,F) - \rmH(\Phi\,|\,\pi_{[0;N)}\,;\,F)\\
&=\rmH(\xi_{[0;N)}\,|\,\pi_{[0;N)}\,;\,F) - \rmH(\psi_{[0;N)}\,|\,\pi_{[0;N)}\,;\,F).
\end{align*}
By inequality (a) from Lemma~\ref{lem:choice-of-F}, this is less than
\[(\rmh(\xi\,|\,\pi) + \k r)\cdot N  - \rmH(\psi_{[0;N)}\,|\,\pi_{[0;N)}\,;\,F).\]
On the other hand, Lemma~\ref{lem:rel-ent-over-periodic} gives
\[\rmH(\psi_{[0;N)}\,|\,\pi_{[0;N)}\,;\,F) \geq \rmh(\psi\,|\,\pi)\cdot N,\]
so it follows that
\[\rmH(\xi_{[0;N)}\,|\,\phi_{[0;N)}\vee \pi_{[0;N)}\,;\,F) < (\rmh(\xi\,|\,\pi) - \rmh(\psi\,|\,\pi) + \k r)\cdot N.\]
Finally, $\pi \vee \psi$ and $\pi \vee \phi$ generate the same $\s$-algebra modulo $\mu$ and are both measurable with respect to $\pi \vee \xi$, so Lemma~\ref{lem:chain} can be applied and re-arranged to give
\[\rmh(\xi\,|\,\pi) - \rmh(\psi\,|\,\pi) = \rmh(\xi \vee \phi\,|\,\pi) - \rmh(\phi\,|\,\pi) = \rmh(\xi\,|\,\pi \vee \phi).\]
Substituting this into the previous inequality, we have shown that
\[\rmH(\xi_{[0;N)}\,|\,(\pi \vee \phi)_{[0;N)}\,;\,F) < \big(\rmh(\xi\,|\,\pi\vee \phi) + \k r\big)\cdot N.\]

Since $r < 3r$, this completes the proof of all three desired properties with $3r$ in place of $r$. This suffices because $r > 0$ was arbitrary.
\end{proof}

\subsection{Completion of the quantitative step}\label{subs:preA2}

We now complete the proof of Proposition~\ref{prop:basic-ext2}. Let $r' := r/17$, and apply Proposition~\ref{prop:basic-ext} with $r'$ in place of $r$. Let $\phi$, $N \in \bbN$, $\k > 0$ and $F\subseteq X$ be the data given by that proposition.

We show that this $\phi$ satisfies the conclusions of Proposition~\ref{prop:basic-ext2}.  Essentially, this requires that we replace conclusion (b) of Proposition~\ref{prop:basic-ext} with extremality that holds without reference to the periodic set $F$ and over arbitrarily long time-scales.

This replacement is enabled by the next two lemmas.  Let $\t{\pi} := \pi\vee \phi$, and suppose it takes values in $C^\bbZ$.  The first lemma gives extremality over longer time-scales, but still according to $\mu(\,\cdot\,|\,F)$.

\begin{lem}\label{lem:more-ext1}
For any $k \in \bbN$, the observable $\xi_{[0;kN)}$ is relatively $(\k kN,5r')$-extremal over $\t{\pi}_{[0;kN)}$ according to $\mu(\,\cdot\,|\,F)$.
\end{lem}

\begin{proof}
This is our application of the `stretching' result for extremality given by Corollary~\ref{cor:stretch}. Let $\mu' := \mu(\,\cdot\,|\,F)$.  Since $F$ is $N$-periodic, $\mu'$ is invariant under $T^N$, and so we may apply Corollary~\ref{cor:stretch} to the joint distribution of $\xi$ and $\t{\pi}$ under $\mu'$.  Conclusions (b) and (c) of Proposition~\ref{prop:basic-ext} give precisely the hypotheses (i) and (ii) needed by that corollary.  The result follows.
\end{proof}

Using Lemma~\ref{lem:more-ext1}, we now prove a similar result for each of the shifted sets $T^{-i}F$, and for all sufficiently large time-scales rather than just multiples of $N$.

\begin{lem}\label{lem:more-ext2}
Let $i \in \{0,1,\dots,N-1\}$.  For all sufficiently large $t$, the observable $\xi_{[0;t)}$ is $(\k t,r)$-extremal over $\t{\pi}_{[0;t)}$ according to $\mu(\,\cdot\,|\,T^{-i}F)$.
\end{lem}

\begin{proof}
	For this we must combine Lemma~\ref{lem:more-ext1}, the invariance of $\mu$, and some of the stability properties of extremality from Section~\ref{sec:ext}.
	
\vspace{7pt}
	
\emph{Step 1.}\quad Assume $t \geq N$, and let $k \geq 0$ be the largest integer satisfying ${i + kN \leq t}$.  Let $j:= i + kN$.  Then $[i;j)$ is the largest subinterval of $[0;t)$ that starts at $i$ and has length a multiple of $N$.

For any $(\bf{c},\bf{a}) \in (C\times A)^{kN}$, the $T$-invariance of $\mu$ gives
\[\mu\big((\t{\pi} \vee \xi)_{[i;j)} = (\bf{c},\bf{a})\,\big|\,T^{-i}F\big) = \mu\big((\t{\pi} \vee \xi)_{[0;kN)} = (\bf{c},\bf{a})\,\big|\,F\big).\]
Therefore, by Lemma~\ref{lem:more-ext1}, $\xi_{[i;j)}$ is $(\k kN,5r')$-extremal over $\t{\pi}_{[i;j)}$ according to $\mu(\,\cdot\,|\,T^{-i}F)$.

\vspace{7pt}

\emph{Step 2.}\quad The observable $\t{\pi}_{[0;t)}$ generates a finer partition of $C^\bbZ$ than the observable $\t{\pi}_{[i;j)}$.  On the other hand,
\[	\rmH(\xi_{[i;j)}\,|\,\t{\pi}_{[i;j)}) - \rmH(\xi_{[i;j)}\,|\,\t{\pi}_{[0;t)}) \leq \rmH(\t{\pi}_{[0;t)}\,|\,\t{\pi}_{[i;j)})
	\leq \rmH(\t{\pi}_{[0;i)}\vee \t{\pi}_{[j;t)}) \leq 2N\rmH(\t{\pi}_0),\]
where the last estimate holds by stationarity and because $i\leq N$ and $t-j \leq N$. Therefore we may combine Step 1 with Corollary~\ref{cor:refinement-still-ext} to conclude that $\xi_{[i;j)}$ is extremal over $\t{\pi}_{[0;t)}$ according to $\mu(\,\cdot\,|\,T^{-i}F)$ with parameters
\[\Big(\k kN,\frac{4N\rmH(\t{\pi}_0)}{\k kN} + 15r'\Big).\]
If $t$ is sufficiently large, and hence $k$ is sufficiently large, then these parameters are at least as strong as $(\k kN,16r')$.

\vspace{7pt}

\emph{Step 3.}\quad Finally, assume $t$ is so large that $kN \geq (1 - r')t$.  Since $\xi_{[i;j)}$ is an image of $\xi_{[0;t)}$ under coordinate projection, we may apply Lemma~\ref{lem:ext-stab-lift}.  Starting from the conclusion of Step 2, that lemma gives that $\xi_{[0;t)}$ is $(\k t,17r')$-extremal over $\t{\pi}_{[0;t)}$ according to $\mu(\,\cdot\,|\,T^{-i}F)$. Since $r = 17r'$, this completes the proof.
\end{proof}

\begin{proof}[Proof of Proposition~\ref{prop:basic-ext2}]
Let $\theta := (\t{\pi}\vee \xi)_\ast\mu$ and $\g := \t{\pi}_\ast\mu$. For $t \in \bbN$, consider the $t$-block kernel of $\theta$:
\[\theta_{[0;t)}^\bl(\bf{a}\,|\,\bf{c}) := \mu(\xi_{[0;t)} = \bf{a}\,|\,\t{\pi}_{[0;t)} = \bf{c}) \quad (\bf{a} \in A^t,\ \bf{c} \in C^t)\]
(see Definition~\ref{dfn:block-kernel}). (It does not matter how we interpret this in case ${\mu(\t{\pi}_{[0;t)} = \bf{c})}$ is zero.)

Each of the sets $T^{-i}F$ is measurable with respect to $\t{\pi}_0$, hence certainly with respect to $\t{\pi}_{[0;t)}$.  So there is a partition
\[C^t = G_1 \cup \cdots \cup G_N\]
such that $T^{-i}F = \{\t{\pi}_{[0;t)} \in G_i\}$ modulo $\mu$ for each $i=1,2,\dots,N$. By Lemma~\ref{lem:more-ext2} the measure $\g_{[0;t)}(\,\cdot\,|\,G_i)$ and kernel $\theta_{[0;t)}^\bl(\,\cdot\,|\,\cdot)$ are $(\k t,r)$-extremal for each $i$.  Since the $G_i$s are a partition of $C^t$, it follows directly from Definition~\ref{dfn:ext} that $\g_{[0;t)}$ itself and $\theta_{[0;t)}^\bl(\,\cdot\,|\,\cdot)$ are $(\k t,r)$-extremal.
\end{proof}

\section{Completed proof of Theorem A}\label{sec:A}

\subsection{Proof under assumption (A)}\label{subs:AA}

Let $\pi:(X,\mu,T)\to (Y,\nu,S)$ be the factor map in the statement of Theorem A.  We must prove the relative weak Pinsker property over this factor. We first do this under the additional assumption (A) from the previous section. Note that we do \emph{not} assume that $(X,\mu,T)$ itself has finite entropy.

\begin{proof}[Proof of Theorem A under assumption (A)]
As in the previous section, under assumption (A) we may assume that $\pi$ is a finite-valued process, not just an arbitrary factor map.

Let $\xi^{(k)}:X\to A_k^\bbZ$, $k\geq 1$, be a sequence of finite-valued processes which toegether generate the whole of $\B_X$.  By replacing each $\xi^{(k)}$ with the common refinement $\xi^{(1)}\vee \cdots \vee \xi^{(k)}$ if necessary, we may assume that $\xi^{(k)}$ is $\xi^{(k+1)}$-measurable for every $k$.  For the sake of notation, let $\xi^{(0)}$ be a trivial process.

(In case $\rmh(\mu,T) < \infty$, Krieger's generator theorem provides a single generating process, but this makes the rest of the proof only slightly simpler.)

\vspace{7pt}

\emph{Construction.}\quad Fix a sequence $(k_i)_{i=1}^\infty \in \bbN^\bbN$ in which every natural number appears infinitely often.  We use a recursion to construct sequences of finite-valued processes $\phi^{(i)}$ and real values $\k_i > 0$ such that the following hold for every $i$:
\begin{itemize}
\item[(a)] we have
\[\rmh(\phi^{(i)},\mu,T) < 2^{-i}\min\big\{\eps,\,\k_1,\,\dots,\,\k_{i-1}\big\}.\]
\item[(b)] $\xi^{(k_i)}$ is conditionally $(\k_i,2^{-i})$-extremal over the process
\[\psi^{(i)}:= \pi\vee \xi^{(k_i-1)}\vee \phi^{(1)}\vee \cdots \vee\phi^{(i)}.\]
\end{itemize}

To start the recursion, apply Proposition~\ref{prop:basic-ext2} to the processes $\pi\vee\xi^{(k_1-1)}$ and $\xi^{(k_1)}$.  If $i\geq 2$ and we have already found $\phi^{(j)}$ for all $j < i$, then apply Proposition~\ref{prop:basic-ext2} to the processes
\[\pi \vee \xi^{(k_i-1)}\vee \phi^{(1)}\vee \cdots \vee \phi^{(i-1)}\ \quad \hbox{and} \quad \xi^{(k_i)}.\]

\vspace{7pt}

\emph{Completion of the proof.}\quad Consider the factor of $(X,\mu,T)$ generated by $\pi$ and all the new processes $\phi^{(i)}$. By property (a), it satisfies
\begin{multline*}
\rmh\Big(\pi \vee \bigvee_{i\geq 1}\phi^{(i)},\mu,T\,\Big|\,\pi\Big) \leq \lim_{\ell \to\infty}\rmh(\phi^{(1)}\vee\cdots\vee \phi^{(\ell)},\mu,T) \\ \leq \lim_{\ell\to\infty} \sum_{j=1}^\ell\rmh(\phi^{(j)},\mu,T) < \eps.
\end{multline*}
Therefore, by Krieger's generator theorem, there is a finite-valued process $\t{\pi}$ that generates the same factor of $(X,\mu,T)$ as $\pi \vee \bigvee_{i\geq 1}\phi^{(i)}$, modulo $\mu$-negligible sets.

More generally, for any $i \in \bbN$, a similar estimate using property (a) gives
\begin{equation}\label{eq:G-rel-Gi}
\rmh(\t{\pi},\mu,T\,|\,\pi \vee \phi^{(1)} \vee \cdots \vee \phi^{(i)} ) \leq \sum_{j=i+1}^\infty\rmh(\phi^{(j)},\mu,T) < \sum_{j=i+1}^\infty 2^{-j}\k_i = 2^{-i} \k_i.
\end{equation}

Now consider one of the processes $\xi^{(k)}$.  Since $k$ appears in the sequence $(k_i)_i$ infinitely often, there are arbitrarily large $i \in \bbN$ for which $k_i = k$.  We consider such a choice of $i$, and argue as follows:
\begin{itemize}
	\item According to property (b) above, $\xi^{(k)} = \xi^{(k_i)}$ is conditionally $(\k_i,2^{-i})$-extremal over $\psi^{(i)}$.
	\item Using estimate~\eqref{eq:G-rel-Gi} and Lemma~\ref{lem:ext-robust}, it follows that $\xi^{(k)}$ is also $(\k_i,6\cdot 2^{-i})$-extremal over $\t{\pi} \vee \psi^{(i)}$.
	\item Next, Proposition~\ref{prop:ABI-FD} gives that $\xi^{(k)}$ is relatively $(42\cdot 2^{-i})$-FD over $\t{\pi}\vee \psi^{(i)}$.
	\item Finally, Lemma~\ref{lem:FD-and-FD} gives the same conclusion for $\xi^{(k)}$ over $\t{\pi} \vee \xi^{(k-1)}$, since the the two processes $\t{\pi}\vee \psi^{(i)}$ and $\t{\pi}\vee \xi^{(k-1)}$ generate the same factor of $(X,\mu,T)$ modulo $\mu$.
\end{itemize}

Thus, we have shown that $\xi^{(k)}$ is relatively $(42\cdot 2^{-i})$-FD over $\t{\pi}\vee \xi^{(k-1)}$ for arbitrarily large values of $i$.  Now Thouvenot's theorem (Theorem~\ref{thm:FD-Bern}) gives that it is also relatively Bernoulli.

For each $k$, this fact promises an i.i.d. process $\zeta^{(k)}:X \to C^\bbZ_k$ such that (i) the map $\zeta^{(k)}$ is independent from the map
\[\t{\pi}\vee \xi^{(1)} \vee \cdots \vee \xi^{(k-1)},\]
and (ii) the maps
\[\t{\pi}\vee \xi^{(1)} \vee \cdots \vee \xi^{(k-1)} \vee \zeta^{(k)} \quad \hbox{and} \quad \t{\pi}\vee \xi^{(1)} \vee \cdots \vee \xi^{(k-1)}\vee \xi^{(k)}\]
generate the same factor of $(X,\mu,T)$ modulo $\mu$. It follows that the factor map
\[\bigvee_{k\geq 1}\zeta^{(k)}: X\to \prod_{k\geq 1}C_k^\bbZ\]
has image a Bernoulli shift (possibly of infinite entropy), is independent from the factor map $\t{\pi}$, and is such that
\[\t{\pi}\vee \bigvee_{k\geq 1}\zeta^{(k)} \quad \hbox{and} \quad \t{\pi}\vee \bigvee_{k\geq 1}\xi^{(k)}\]
generate the same factor of $(X,\mu,T)$ modulo $\mu$.  This factor is the whole of $\B_X$ by our choice of the $\xi^{(k)}$s, so this completes the proof.
	\end{proof}

\begin{rmk}
	An important part of the above proof is our appeal to Lemma~\ref{lem:ext-robust}. Given the extremality of $\xi^{(k_i)}$ over $\psi^{(i)}$ and the entropy bound~\eqref{eq:G-rel-Gi}, that lemma promises that $\xi^{(k_i)}$ is still extremal, with slightly worse parameters, over the limiting factor map $\t{\pi}\vee \psi^{(i)}$.
	
	It is natural to ask whether a more qualitative fact could be used in place of this argument. Specifically, suppose that $\xi$ is a process and $\pi^{(i)}$ is a sequence of factor maps such that $\pi^{(i)}$ is $\pi^{(i+1)}$-measurable for every $i$.  If $\xi$ is relatively Bernoulli over each $\pi^{(i)}$, is it also relatively Bernoulli over $\bigvee_{i\geq 1}\pi^{(i)}$?  Unfortunately, Jean-Paul Thouvenot has indicated to me an example in which the answer is No.  It seems that some more quantitative argument, like our appeal to Lemma~\ref{lem:ext-robust}, really is needed.
\end{rmk}

\subsection{Removing assumption (A), and general amenable groups}

We now complete the proof of Theorem A in full.

In fact, for the work that is left to do, it makes no difference if we prove a result for actions of a general countable amenable group $G$.  So suppose now that $T = (T^g)_{g \in G}$ is a free and ergodic measure-preserving action of $G$ on a standard probability space $(X,\mu)$, and let
\[\pi:(X,\mu,T) \to (Y,\nu,S)\]
be a factor map of measure-preserving $G$-actions.

\begin{thm}\label{thm:A+}
For every $\eps > 0$, there is another factor map $\t{\pi}$ of the $G$-system $(X,\mu,T)$ such that $\pi$ is $\t{\pi}$-measurable, $(X,\mu,T)$ is relatively Bernoulli over $\t{\pi}$, and
	\[\rmh(\t{\pi},\mu,T\,|\,\pi) < \eps.\]
\end{thm}

The deduction of Theorem~\ref{thm:A+} from the special case of the preceding subsection uses some machinery from orbit equivalence theory. The results we need can mostly be found in Danilenko and Park's development~\cite{DanPark02} of the `orbital approach' to Ornstein theory for amenable group actions.

\begin{proof}[Proof of Theorem~\ref{thm:A+}, and hence of Theorem A]
	First, we may enlarge $\pi$ with as little additional entropy as we please in order to assume that the $G$-system $(Y,\nu,S)$ is free (see~\cite[Theorem 5.4]{DanPark02}).
	
	Having done so, Connes, Feldman and Weiss' classic generalization of Dye's theorem~\cite{ConFelWei81} provides a single automorphism $S_1$ of $(Y,\nu)$ that has the same orbits as $S$.  Moreover, we may choose $S_1$ to be isomorphic to any atomless and ergodic automorphism we like.  In particular, we can insist that $S_1$ satisfy the conditions in assumption (A).

Since $(Y,\nu,S)$ is free, the factor map $\pi$ restricts to a bijection on almost every individual orbit of $T$.  Using these bijections we obtain a unique lift of $S_1$ to an automorphism $T_1$ of $(X,\mu)$ with the same orbits as $T$.
		
	Now let
	\[\t{\pi}:(X,\mu,T_1)\to (\t{Y},\t{\nu},\t{S}_1)\]
	be a new factor map provided by the special case of Theorem A that we have already proved.  Then $\t{\pi}$ is also a factor map for the original $G$-action $T$, because the $\s$-algebra generated by $\t{\pi}$ contains the $\s$-algebra generated by $\pi$, and the orbit-equivalence cocycle is $\pi$-measurable.  By Rudolph and Weiss' relative entropy theorem~\cite{RudWei00}, we have
	\[\rmh(\t{\pi},\mu,T\,|\,\pi) = \rmh(\t{\pi},\mu,T_1\,|\,\pi) < \eps.\]
	On the other hand, the automorphism $T_1$ is relatively Bernoulli over $\t{\pi}$, and the orbit-equivalence cocycle is also $\t{\pi}$-measurable.  To finish,~\cite[Proposition 3.8]{DanPark02} shows that this property of relative Bernoullicity really depends on only (i) the orbit equivalence relation on $(\t{Y},\t{\nu})$ inherited from $(X,\mu,T_1)$, and (ii) the associated cocycle of automorphisms of $(X,\mu)$.  Therefore the automorphism $T_1$ is relatively Bernoulli over $\t{\pi}$ if and only if this holds for the $G$-action $T$.
\end{proof}

\section{Some known consequences and open questions}\label{sec:cors}

Since Thouvenot proposed the weak Pinsker property, it has been shown to have many consequences in ergodic theory.  Here we recall some of these, now stated as unconditional results about ergodic automorphisms.  We also present a couple of new results and collect a few open questions.  Jean-Paul Thouvenot contributed a great deal to the contents of this section.

\subsection{Results about joinings, factors, and isomorphisms}

\begin{prop}[From~\cite{Tho08}]\label{prop:cancel}
Let $\bf{X}$ and $\bf{X}'$ be ergodic automorphisms, let $\bf{B}$ be a Bernoulli shift with finite entropy, and suppose that $\bf{X} \times \bf{B}$ is isomorphic to $\bf{X}' \times \bf{B}$.  Then $\bf{X}$ and $\bf{X}'$ are themselves isomorphic. \qed
\end{prop}

(The statement of Proposition~\ref{prop:cancel} in~\cite{Tho08} is restricted to finite-entropy automorphisms, but one sees easily from the proof that the assumption of finite entropy for $\bf{X}$ and $\bf{X}'$ can be removed once we have Theorem A.)

The following question, also posed by Thouvenot, remains open.

\begin{ques}
Does Proposition~\ref{prop:cancel} still hold if both occurrences of `isomorphic' are replaced with `Kakutani equivalent'?
\end{ques}

It is not clear whether the weak Pinsker property offers any insight here.  The question is open even when $\bf{X}$ and $\bf{X}'$ both have entropy zero.  The following related question may be even more basic, and is also open.

\begin{ques}
Does every ergodic automorphism have an induced transformation that splits into a system of entropy zero and a Bernoulli system?
\end{ques}

Curiously, the weak Pinsker property also gives a result about canceling more general K-automorphisms from direct products in the opposite direction to Proposition~\ref{prop:cancel}.

\begin{prop}
Let $\bf{X}$ and $\bf{X}'$ be ergodic systems of equal entropy $h$, possibly infinite, and let $\eps > 0$.  Then there is a system $\bf{Z}$ of entropy less than $\eps$ such that
\[\bf{X} \times \bf{Z} \quad \hbox{is isomorphic to} \quad \bf{X}' \times \bf{Z}.\]
If the two original systems are K-automorphisms, then $\bf{Z}$ may be chosen to be K.
\end{prop}

\begin{proof}
Let $\bf{X}_0 := \bf{X}$ and $\bf{X}_0' := \bf{X}'$. A recursive application of the weak Pinsker property gives sequences of systems $\bf{X}_i$ and $\bf{B}_i$, $i\geq 1$, such that $\bf{X}_i$ has entropy less than $2^{-i-1}\eps$, each $\bf{B}_i$ is Bernoulli, and such that
\[\bf{X}_i \quad \hbox{is isomorphic to} \quad \bf{X}_{i+1} \times \bf{B}_{i+1} \quad \forall i\geq 0.\]
If $\bf{X}_0$ has the K-property, then so do all its factors $\bf{X}_i$. Apply the same argument to $\bf{X}_0'$ to obtain $\bf{X}'_i$ and $\bf{B}'_i$ for $i\geq 1$.  If $\bf{X}'_0$ has the K-property, then so does every $\bf{X}'_i$.

Now let
\[\bf{Z} := \bf{X}_1\times \bf{X}_2 \times \cdots \times \bf{X}_1'\times \bf{X}'_2 \times \cdots.\]
This has entropy less than
\[\sum_{i=1}^\infty 2^{-i-1}\eps + \sum_{i=1}^\infty 2^{-i-1}\eps = \eps.\]
If $\bf{X}_0$ and $\bf{X}_0'$ have the K-property, then $\bf{Z}$ is a direct product of K-automorphisms, so it also has the K-property.

To finish the proof, consider the direct product
\[\bf{X}_0 \times \bf{Z} = \bf{X}_0\times \bf{X}_1 \times \bf{X}_2 \times \cdots \times \bf{X}_1'\times \bf{X}'_2 \times \cdots.\]
In the product on the right, we know that each coordinate factor $\bf{X}_i$ may be replaced with $\bf{X}_{i+1}\times \bf{B}_{i+1}$ up to isomorphism. Therefore this product is isomorphic to
\[(\bf{X}_1\times \bf{B}_1) \times (\bf{X}_2\times \bf{B}_2)\times \cdots \times \bf{X}_1' \times \bf{X}_2' \times.\]
Re-ordering the factors, this is isomorphic to
\[\bf{X}_1 \times \bf{X}_2 \times \cdots \times \bf{X}_1' \times \bf{X}_2' \times \cdots \times \bf{B}_1\times \bf{B}_2\times \cdots,\]
which is equal to the direct product of $\bf{Z}$ with a Bernoulli shift.  Similarly, $\bf{X}_0'\times \bf{Z}$ is isomorphic to the direct product of $\bf{Z}$ and a Bernoulli shift.  Since $\bf{X}_0\times \bf{Z}$ and $\bf{X}_0'\times \bf{Z}$ have the same entropy, it follows that they are isomorphic to each other.
\end{proof}

\begin{prop}[From~\cite{SmoTho79}]\label{prop:SmoTho}
If $\bf{X}$ is ergodic and has positive entropy, then it has two Bernoulli factors that together generate the whole $\sigma$-algebra modulo $\mu$; equivalently, $\bf{X}$ is isomorphic to a joining of two Bernoulli systems.  For any $\eps > 0$, we may choose one of those Bernoulli factors to have entropy less than $\eps$. \qed
\end{prop}

The next result answers a question of Benjamin Weiss which has circulated informally for some time.  I was introduced to the question by Yonatan Gutman. The neat proof below was shown to me by Jean-Paul Thouvenot.  Thouvenot discusses the question further as problem (1) in~\cite{Tho09}.

\begin{prop}\label{prop:common-ext}
If $\bf{X}$ and $\bf{X}'$ are K-automorphisms of equal entropy $h$, possibly infinite, then they have a common extension that is also a K-automorphism of entropy $h$.
\end{prop}

\begin{proof}
The desired conclusion depends on our two systems only up to isomorphism.  Therefore, by the weak Pinsker property, we may assume that
\[\bf{X} = \bf{X}_1\times \bf{B}\]
and
\[\bf{X}' = \bf{X}'_1\times \bf{B}',\]
where $\bf{B}$ and $\bf{B}'$ are Bernoulli and where the entropies $a := \rmh(\bf{X}_1)$ and $b := \rmh(\bf{X}'_1)$ satisfy $a+b < h$.  The new systems $\bf{X}_1$ and $\bf{X}_1'$ inherit the K-property from the original systems.

Let $\bf{B}''$ be a Bernoulli shift of entropy $h - a - b$, interpreting this as $\infty$ if $h = \infty$.  Now consider the system
\[\t{\bf{X}} := \bf{X}_1 \times \bf{X}_1' \times \bf{B}''.\]
This is a product of K-automorphisms, hence still has the K-property.  Its entropy is
\[a + b + (h-a-b) = h.\]
Finally, it admits both of our original systems as factors.  Indeed, by Sinai's theorem, there is a factor map from $\bf{X}_1'$ to another Bernoulli shift $\bf{B}_1'$ of entropy $b$.  Applying this map to the middle coordinate of $\t{\bf{X}}$ gives a factor map from $\t{\bf{X}}$ to
\[\bf{X}_1 \times \bf{B}_1'\times \bf{B}'',\]
which is isomorphic to $\bf{X}_1 \times \bf{B}$ and hence to $\bf{X}$.  The argument for $\bf{X}'$ is analogous.
\end{proof}

Given an automorphism $(X,\mu,T)$, another object of interest in ergodic theory is its \textbf{centralizer}: the group of all other $\mu$-preserving transformations of $X$ that commute with $T$.  Informally, it is known that the centralizer of a Bernoulli system is very large.  A precise result in this direction is due to Rudolph, who proved in~\cite{Rud78--cent} that if $T$ is Bernoulli and $S$ is another automorphism that commutes with the whole centralizer of $T$, then $S$ must be a power of $T$.

If $(X,\mu,T)$ is any ergodic automorphism with positive entropy, then Theorem A provides a non-trivial splitting
\begin{equation}\label{eq:cent-split}
(X,\mu,T) \stackrel{\rm{iso}}{\to} (Y,\nu,S) \times \bf{B}.
\end{equation}
Under this splitting, we obtain many elements of the centralizer of $T$ of the form $S^m \times U$, where $m \in \bbZ$ and $U$ is any centralizer element of $\bf{B}$.  In the first place, this answers a simple question posed in~\cite{Rud78--cent}: if $(X,\mu,T)$ has positive entropy, then its centralizer contains many elements that are not just powers of $T$.  I understand from Jean-Paul Thouvenot that by combining these centralizer elements with unpublished work of Ferenczi and the ideas in~\cite[Remark 2]{SmoTho79}, one can also extend Rudolph's main result from~\cite{Rud78--cent} to any ergodic automorphism of positive entropy. We do not explore this fully here, but leave the reader with the next natural question:

\begin{ques}
	Is there a K-automorphism $(X,\mu,T)$ such that every element of the centralizer arises as $S^m\times U$ for some splitting as in~\eqref{eq:cent-split}?
	\end{ques}

An old heuristic of Ornstein asserts that the classification of factors within a fixed Bernoulli shift should mirror the complexity of the classification of general K-automorphisms: see the discussion in~\cite{Orn75} and the examples in~\cite{Hoff99b}.  Can the methods of the present paper shed additional light on the lattice of all factors of a fixed Bernoulli shift?

Finally, here is a question that seems to lie strictly beyond the weak Pinsker property.  Theorem A promises that any ergodic system with positive entropy may be split into a non-trivial direct product in which one factor is Bernoulli.

\begin{ques}
Is there a non-Bernoulli K-automorphism with the property that any non-trivial splitting of it consists of one Bernoulli and one non-Bernoulli factor?
\end{ques}

\subsection{Results about generating observables and processes}

Recall that a process $\a:(X,\mu,T)\to (A^\bbZ,\a_\ast\mu,T_A)$ is uniquely determined by its time-zero observable $\a_0$.  That observable is called \textbf{generating} if $\a$ is an isomorphism of systems.

\begin{prop}[From~\cite{Tho08}]
	Any ergodic system of finite entropy has a generating observable whose two-sided tail equals the Pinsker factor. \qed
\end{prop}

\begin{prop}[From~\cite{PetTho04}]
	Any ergodic system of finite entropy has a generating observable whose one-sided super-tail equals the Pinsker factor (see~\cite{PetTho04} for the relevant definitions). \qed
\end{prop}

\begin{prop}[From~\cite{Ver00}]
For any ergodic finite-state process, Vershik's secondary entropy (defined in~\cite{Ver00}) grows slower than any doubly exponential function (although it can grow faster than any single exponential~\cite{MarShi02}). \qed
\end{prop}

The next result is a consequence of Proposition~\ref{prop:common-ext}.  It was brought to my attention by Yonatan Gutman, who established the connection together with Michael Hochman.

\begin{cor}[From~\cite{GutHoc08}]
	The class of finite-entropy, non-Bernoulli K-processes has no finitely observable invariants other than entropy (see~\cite{GutHoc08} for the relevant definitions). \qed
\end{cor}

If $(X,\mu,T)$ is a system and $\a:X\to A^\bbZ$ is a finite-valued generating process, then the conditional entropy $\rmH_\mu(\a_0\,|\,\a_{(-\infty;0)})$ equals the Kolmogorov--Sinai entropy $\rmh(\mu,T)$.  However, if the process is real-valued (so $A = \bbR$), then in general one knows only that
\[\rmH_\mu(\a_0\,|\,\a_{(-\infty;0)}) \leq \rmh(\mu,T).\]
In connection with these quantities, the next result answers an early question of Rokhlin and Sinai~\cite[paragraph 12.7]{Rok67--survey}.

\begin{prop}\label{prop:RokSin}
	If $(X,\mu,T)$ is a finite-entropy K-automorphism, then for any $c \in (0,\rmh(\mu,T)]$ it has a real-valued generating process $\a:X\to \bbR^\bbZ$ such that
	\[\rmH_\mu(\a_0\,|\,\a_{(-\infty;0)}) = c.\]
\end{prop}

\begin{proof}
If $c = \rmh(\mu,T)$ then we simply take a finite-valued generating observable, as provided by Krieger's generator theorem.  So let us assume that $c < \rmh(\mu,T)$.

	In case $(X,\mu,T)$ is Bernoulli, the result is proved by Lindenstrauss, Peres and Schlag in~\cite{LinPerSch02} using an elegant construction of $\a$ in terms of Bernoulli convolutions.
	
	For the general case, by Theorem A we may split $(X,\mu,T)$ into a factor of entropy less than $\rmh(\mu,T) - c$ and a Bernoulli factor.  Take any finite-valued generating observable of the low-entropy factor, and apply the result from~\cite{LinPerSch02} to the Bernoulli factor.  The resulting product observable gives conditional entropy equal to $c$.
	\end{proof}

Another natural question about generating processes asks for a refinement of Theorem A.  It was suggested to me by G\'abor Pete.

\begin{ques}
	Let $(A^\bbZ,\mu,T_A)$ be an ergodic finite-state process and let $\eps > 0$.  Must the process be \emph{finitarily} isomorphic to a direct product of the form
	\[\underbrace{(C^\bbZ,\nu,T_C)}_{\rm{entropy}\ < \eps} \times \underbrace{(B^\bbZ,p^{\times \bbZ},T_B)}_{\rm{i.i.d.}}?\]
	If this does not always hold, are there natural sufficient conditions?
	
	What if we ask instead for an isomorphism that is unilateral in one direction or the other?
\end{ques}

\subsection{More general groups and actions}\label{subs:other-groups}

Many of the corollaries listed above should extend to actions of general countable amenable groups.  Indeed, the orbital methods of~\cite{DanPark02} yield some such generalizations quite easily: for instance, Proposition~\ref{prop:SmoTho} above should imply its extension to actions of amenable groups using the same arguments as for~\cite[Theorem 6.1]{DanPark02}.

Here is another consequence of the weak Pinsker property for actions of an arbitrary countable amenable group $G$:

\begin{prop}\label{prop:erasure-rate}
	Let $G$ be a countable amenable group. For any $\eps > 0$, any free and ergodic $G$-system $(X,\mu,T)$ of finite entropy $h$ has a generating observable $\a$ with the property that
	\begin{equation}\label{eq:erasure-rate}
	\rmH_\mu\Big(\a\,\Big|\,\bigvee_{g\in G\setminus \{e\}}\a\circ T^g\Big) > h - \eps.
	\end{equation}
	\qed
\end{prop}

To prove this, simply combine the two time-zero observables in a direct product of a low-entropy system and a Bernoulli shift.  Proposition~\ref{prop:erasure-rate} resulted from discussions with Brandon Seward about `percolative entropy', a quantity of current interest in the ergodic theory of non-amenable groups. In case $G = \bbZ$, the left-hand side of~\eqref{eq:erasure-rate} is the same as Verd\'u and Weissman's `erasure entropy rate' from~\cite{VerWei08}. In that setting, Proposition~\ref{prop:erasure-rate} asserts that any stationary ergodic source with a finite alphabet admits generating observables whose erasure entropy rate is arbitrarily close to the Kolmogorov--Sinai entropy.  This can be seen as an opposite to Ornstein and Weiss' celebrated result that any such source has a generating observable which is bilaterally deterministic~\cite{OrnWei75}.

Fieldsteel showed in~\cite[Theorem 3]{Fie84} that an ergodic flow $(T^t)_{t\in\bbR}$ (that is, a bi-measurable and measure-preserving action of $\bbR$) has the weak Pinsker property if and only if the map $T^t$ is ergodic and has the weak Pinsker property for at least one value of $t$. Combined with Theorem A, this immediately gives the following.

\begin{prop}\label{prop:Fie}
	Every ergodic flow has the weak Pinsker property. \qed
	\end{prop}

Do Theorem~\ref{thm:A+} and Proposition~\ref{prop:Fie} have a common generalization to actions of a larger class of locally compact, second countable groups?  It would be natural to start with those that have a `good entropy theory'~\cite{OrnWei87}.  Some relevant techniques may be available in~\cite{Avn10}.  In another direction, is there a sensible generalization to non-free actions?

Beyond the class of amenable groups one cannot hope for too much.  For sofic groups, one can ask after a version of the weak Pinsker property with sofic entropy~\cite{Bowen10} or Rokhlin entropy~\cite{Seward--KriI} in place of KS entropy, but Lewis Bowen has found natural examples of actions of non-amenable free groups for which those properties fail.

\bibliographystyle{abbrv}
\bibliography{bibfile}

\vspace{7pt}

\noindent\small{\textsc{Tim Austin}\\ \textsc{University of California, Los Angeles}\\ \textsc{Department of Mathematics}\\ \textsc{Los Angeles CA 90095-1555, U.S.A.}}

\vspace{7pt}

\noindent Email: \texttt{tim@math.ucla.edu}

\noindent URL: \texttt{math.ucla.edu/$\sim$tim}

\end{document}